\documentclass[a4paper] {article}[12pt]

\usepackage[top=3cm, bottom=3cm, left=2.2cm, right= 2.2cm]{geometry}

\makeatletter
\renewcommand*\l@section{\@dottedtocline{1}{1.5em}{2.3em}}
\makeatother

\usepackage{amsfonts}
\usepackage{amssymb}
\usepackage[T1]{fontenc}

\usepackage{tikz}
\usetikzlibrary{calc}

\usepackage{CJK}
\usepackage{amsmath}
 
\usepackage{amsfonts}
\usepackage{amssymb}
\usepackage{amsthm}
\usepackage{amssymb}
\usepackage{enumerate}
\usepackage[calc]{picture}
\usepackage[all,cmtip]{xy}

\usepackage[mathscr]{eucal}
\usepackage{eqlist}

\usepackage{color}
\usepackage{abstract} 
\usepackage[T1]{fontenc}
 
\theoremstyle{plain}
\newtheorem{theorem}{Theorem}
\newtheorem{proposition}[theorem]{Proposition}
\newtheorem{lemma}[theorem]{Lemma}

\newtheorem{example}[theorem]{Example}
\newtheorem{corollary}[theorem]{Corollary}

\theoremstyle{definition}
\newtheorem{definition}{Definition}

\usepackage{etoolbox}
\newtheoremstyle{myrem}
 {3pt}
 {3pt}
 {\normalsize}
 { }
 {\itshape}
 {:}
 { }
 {}

 \theoremstyle{myrem}
 \newtheorem{remark}{Remark}
 \appto\remark{\leftskip\parindent}
 \appto\remark{\rightskip\parindent}

\usepackage{amsmath}

\numberwithin{equation}{section}
\numberwithin{theorem}{section}

\begin{document}

\begin{center}
~~~
 \vspace{0.7cm}

{\Large{\textbf{
 Hypergraphs   on  Riemannian  
  manifolds  and  their  double  homology 
 }}}

 \vspace{0.7cm}
 
 Shiquan Ren 

\bigskip

\bigskip

 \parbox{24cc}{{\small

{\textbf{Abstract}.}  
In  this  paper,
by  constructing     double  homology  for  graded  submanifolds  of  $\Delta$-manifolds,    
we  study    the  double  homology  of  hyper(di)graphs 
 on  Riemannian  manifolds.  
 With  the  help  of   the  canonical  projection  $\pi$  from  ordered  sequences  to  
 their  underlying  unordered  sets,   we   prove   a   commutative  diagram  of  the  
 double  homology  
   and 
   a  commutative  diagram  of  the spaces  of  differential  forms
  of  hyper(di)graphs  on  Riemannian  manifolds.    
   Moreover,  we  give  some  relations  between  the  fundamental  group(oid)s
   of  hyper(di)graphs  on  surfaces  and  the  (pure)  braid  groups  on  surfaces
   by   certain  homomorphisms   induced  from  $\pi$.   
 As   applications,  we  study  the  homotopy  types  of  
 filtrations  for  packings  and  coverings  of  balls  on  Riemannian  manifolds 
 with  hypergraph  constraints,  and   
 we     characterize    regular   maps  on  Riemannian  manifolds  
 by  geometric realizations  of  the  independence  complexes  and  
 hypergraphs  on  the  manifolds.  
  }

\bigskip

}

\begin{quote}
 {\bf 2020 Mathematics Subject Classification.}  	Primary  57N65,  57N75; Secondary   55U05, 	55U15.

{\bf Keywords and Phrases.}    configuration  spaces,   hypergraphs,   independence  complexes,  
double  homology,   braid  groups,    
 simplicial  manifolds  
\end{quote}

\end{center}

\section{Introduction}

\subsection{The     homology  of  configuration  spaces}
\label{ss1.1}

Let   $X$  be  a  topological  space.  
The  {\it  $k$-th  ordered  configuration  space}  ${\rm  Conf}_k(X)$  is  the  space  of  
ordered  $k$-tuples  of  distinct  points  in  $X$
and  the  {\it  $k$-th  unordered  configuration  space}  ${\rm  Conf}_k(X)/\Sigma_k$  is  the  space  of  
unordered  $k$-tuples  of  distinct  points  in  $X$.  
Suppose   that   
${\rm  Conf}_k(X) $  is   both    path-connected
and  locally path-connected. 
Then   ${\rm  Conf}_k(X)/\Sigma_k$  is     path-connected
and   locally path-connected  as  well.   
We  call  the  fundamental  group  $\pi_1({\rm  Conf}_k(X))$
 the  {\it    pure   braid  group  of  $k$-strands}  on  $X$  and   
 call  the  fundamental  group  $\pi_1({\rm  Conf}_k(X)/\Sigma_k)$
 the  {\it      braid  group  of  $k$-strands}  on  $X$
 (cf.  \cite[Sect.~1.3]{birman}). 
 The  $k!$-sheeted     covering    $\pi: {\rm  Conf}_k(X)\longrightarrow {\rm  Conf}_k(X)/\Sigma_k$  
 is  normal  thus  it    induces  
 a   homomorphism  
 $\pi_*: \pi_1({\rm  Conf}_k(X))\longrightarrow \pi_1({\rm  Conf}_k(X)/\Sigma_k)$
 where  the  image    of  $\pi_*$   is  a  normal  subgroup  of  $\pi_1({\rm  Conf}_k(X)/\Sigma_k)$. 
 The  deck  transformation  group  of  $\pi$  is isomorphic  to   the  quotient  group 
 $\pi_1({\rm  Conf}_k(X)/\Sigma_k)/{\rm  Im}(\pi_*)$.   
  In  particular,  when  $X$  is  the  plane,  
   $\pi_1({\rm  Conf}_k(\mathbb{R}^2))$  is  the  usual  pure braid  group  of  $k$-strands 
   and  $\pi_1({\rm  Conf}_k(\mathbb{R}^2)/\Sigma_k)$  is  the  usual    braid  group  of  $k$-strands
  (cf.  \cite[Sect.~1.4]{birman}).   
  In  2006,  A.~J.   Berrick,  F.~R.  Cohen, Y.~L.   Wong  and  J.   Wu
  \cite{jams-06}  investigated   the  $\Delta$-group  structures   
  induced  by  configurations  (cf.  \cite[Section~4]{jams-06})
  and  revealed  some  essential  relations  between  the  braid  groups and  the  homotopy  groups
  of  spheres  (cf.  \cite[Section~8]{jams-06}).

     The  $\Sigma_k$-action  on  ${\rm  Conf}_k(X)$
   induces  a  $\Sigma_k$-action  on  $ H_\bullet({\rm  Conf}_k(X))$.  
   The  covering  $\pi$  induces  a   homomorphism  of  homology  groups  
   $\pi_*:  H_\bullet({\rm  Conf}_k(X))\longrightarrow  H_\bullet({\rm  Conf}_k(X)/\Sigma_k)$
   which  is  equivariant  under  the   $\Sigma_k$-action  on  $ H_\bullet({\rm  Conf}_k(X))$.   
  Since  the  1960s,  
  the  (equivariant)  (co)homology  of  $ {\rm  Conf}_k(X) $  as  well  as  the  (co)homology  of 
  ${\rm  Conf}_k(X)/\Sigma_k$
   has  been     extensively  studied.   For  example,   V. I.  Arnol’d
   \cite{arnold},  
   F. R.  Cohen  \cite{cohen1976}, C.-F. B\"odigheimer,   F. Cohen   and L. Taylor \cite{homol},
   C.-F. B\"odigheimer and F. R. Cohen   \cite{homol2},   
   N. H. V.  Hu'ng   \cite{hung1990}, 
   Y. F\'{e}lix and J.-C.  Thomas  
    \cite{rational},  and   
    P. Blagojevi$\check{\text{c}}$,
     F. Cohen,    M. C.  Crabb,   W. L{\"u}ck and G. Ziegler 
      \cite{equiv}.

   Let ${\bf  FI}$  be  the  category  whose  objects  are  finite sets  and  
   whose 
    morphisms  are  injective  maps.  
    Let   $R$  be  a  commutative ring  with  unit.  
   An   {\it   ${\bf  FI}$-module}   is  a  functor  from  the  ${\bf  FI}$  category  
   to  the  category  of  $R$-modules   (cf.  \cite[Sect.~2]{fi2015}).  
   Let   $k$  run  over  all  positive  integers.  
   Then 
   ${\rm  Conf}_\bullet(X)/\Sigma_\bullet$  is  a  functor  from  
    ${\bf  FI}$  to  the  category  of   topological  spaces.     Thus   
   $H_n({\rm  Conf}_\bullet(X)/\Sigma_\bullet)$  is  an   ${\bf  FI}$-module
    for  any  $n\geq  0$.  
    Since the  2010s,  
   the  theory  of     ${\bf  FI}$-modules   has  been  developed  and  applied  to     
    the  (co)homology  of       ${\rm  Conf}_\bullet(X)/\Sigma_\bullet$ 
    by   T.  Church  \cite{fi1},  
    T.  Church,  J. S.  Ellenberg  and   B.  Farb  \cite{fi2015},  etc.

Suppose   in  addition  that  $X$  is  a  metric  space.  
Let  $r\geq  0$.  
The  {\it  $k$-th  ordered  configuration  space  of  $r$-balls}   ${\rm  Conf}_k(X,r)$  
is  the  subspace  of    ${\rm  Conf}_k(X)$  such  that  the  mutual  distances  of  the  
$k$-points  are  greater than  $2r$.  
Note  that  ${\rm  Conf}_k(X,r)$   is  $\Sigma_k$-invariant  thus  the  orbit  space 
${\rm  Conf}_k(X,r)/\Sigma_k$   gives     
the  {\it  $k$-th  unordered  configuration  space  of  $r$-balls}.     
 Let  $r$   run  from  $0$  to  $+\infty$.
 We  say  that  $r_0\in[0,+\infty)$  is  a  {\it  critical  point}  for  
    the  homotopy  types  of  ${\rm  Conf}_k(X,r)$   and 
 ${\rm  Conf}_k(X,r)/\Sigma_k$  if  the   homotopy  types  of  ${\rm  Conf}_k(X,r)$   and 
 ${\rm  Conf}_k(X,r)/\Sigma_k$ 
  change  in  any  small  open  neighborhood   of  $r_0$
  (cf.   Subsection~\ref{ss4.1-26apr}).   
  On  the  other  hand,    
the  filtration    ${\rm  Conf}_k(X,-)$  induces  the  persistence  homology 
$H_\bullet({\rm  Conf}_k(X,-))$  (cf.  \cite{ph1,cph1})   
as  well  as  the  $\Sigma_k$-equivariant  persistence  
homology    $H_\bullet^{\Sigma_k}({\rm  Conf}_k(X,-))$.  
The  filtration    ${\rm  Conf}_k(X,-)/\Sigma_k$  induces the  persistence  homology 
$H_\bullet({\rm  Conf}_k(X,-)/\Sigma_k)$.  
 The  persistence covering  $\pi:  {\rm  Conf}_k(X,-)\longrightarrow
  {\rm  Conf}_k(X,-)/\Sigma_k$  induces  
  the  persistence  homomorphism  $\pi_*:  H_\bullet({\rm  Conf}_k(X,-)) \longrightarrow
H_\bullet({\rm  Conf}_k(X,-)/\Sigma_k)$.    
With  the  help  of  persistence  diagrams  (cf.  \cite[Sect.~2]{pd1}),  
the  homological  critical  points   of  
the  persistence  homology  groups  of  ${\rm  Conf}_k(X,-)$  and  ${\rm  Conf}_k(X,-)/\Sigma_k$
 partially  detect    the  set  of  critical  points    
that    the  homotopy  types  of  ${\rm  Conf}_k(X,-)$  and  ${\rm  Conf}_k(X,-)/\Sigma_k$
 change    (cf.     \cite[Sect.  8-10]{gromov1}).   
 The   min-type Morse  theory  for  the  
  configuration  spaces  of  $r$-balls  
    was  studied  by  
    Y.  Baryshnikov, P.   Bubenik  and  M. Kahle
 \cite{imrn1}.   
 Recently,  
 the  persistence  homology  for     the  filtration  of  configuration  spaces 
 is   studied by  H. Alpert, M.  Kahle  and  R. MacPherson  \cite{confcomp}
 and  H.   Alpert  and  F.   Manin  \cite{gt}.

 \subsection{The  homotopy  types  and  the  homology  of  independence  complexes}

Let   $G=(V,E)$  be  a  graph.  
   A     (finite)  {\it   independent  set}  of   $G$  is  a    subset  $\sigma$  of  $V$  
   such  that  the  vertices  in  $\sigma$  are  mutually  non-adjacent.  
   The   {\it  independence  complex}  ${\rm  Ind}(G)$  is  the  simplicial  complex 
   whose  simplices  are  the  independent   sets.  
   So  far,  
the  homotopy  types  of  ${\rm   Ind}(G)$  have  attracted  a  lot  of  attention.    
For  example,   this  topic  has   been    studied  by  
  D.   Kozlov  \cite{kozlov},   J. A.  Barmak  \cite{barmak}, 
   M.  Adamaszek \cite{indhomotopy},  etc.  
Moreover, 
the  homology  groups  of  ${\rm   Ind}(G)$  have  been  studied  by  J.  Jonsson 
 \cite{jacob},   
 M.   Berghoff  \cite{berghoff},  etc.

Let  $d_G$ 
   be the   geodesic  distance  of   $G$  on  $V$.  
     For  any  $r\geq  0$,
     define  ${\rm  Ind}(G,r) $
     to  be  the  simplicial  complex  whose  simplices  are 
     finite  subsets    of  $V$
     with  the mutual  distances  of  the   vertices     greater  than   $2r$. 
When  $r$  runs  from  $1/2$  to  $+\infty$,   
${\rm  Ind}(G,-)$  gives  a  filtration  of  ${\rm  Ind}(G)$.  
  Let   $G^{2r}$   be  the  $2r$-distance  power  of   $G$, 
  i.e.  $G^{2r}$  is  the  graph  obtained    by  adding  certain  edges  to   $G$ 
  such  that  the  vertex  set  of  $G^{2r}$   is   $V$   and  $(u,v)$  is  an  edge  of  
  $G^{2r}$  if  and  only  if  $d_G(u,v)\leq   2r$.  
Then  $G^1= G$.    
When  $r$  runs  from  $1/2$  to  $+\infty$,   
we  have  a  filtration  $G^{2-}$  of  graphs 
 and  consequently  an  induced  filtration  
${\rm  Ind}(G^{2-})$  of   ${\rm  Ind}(G)$.   
Note  that  ${\rm  Ind}(G,r)={\rm  Ind}(G^{2r})$  for  any  $r\geq  0$  
thus  the  filtration  ${\rm  Ind}(G,-)$  coincides  with  the  filtration  
${\rm  Ind}(G^{2-})$.  
In  particular,  let   $G$  be   the  cycle  $C_n$.     
The  homotopy  types  of  ${\rm  Ind}(C_n^r)$  for  $r\geq  1$  and  $n\geq   5r+4$  
 were  studied  by  
M.  Adamaszek \cite[Theorem~1.1]{indhomotopy}. 
Consequently,   certain  information  of    the  homotopy  types  
for the  filtration   ${\rm  Ind}(G,-)$  can  be  derived.  
 Similar  to   Subsection~\ref{ss1.1},
 the  persistence  homology  groups  of  ${\rm  Ind}(G,-)$   
  partially  detect    the  set  of  critical  points    
where   the  homotopy  types   of  ${\rm  Ind}(G,-)$   
 change.

The  independence  complex  ${\rm  Ind}(X)$  of  a   space  $X$
    is  a  generalization  of  the  
   classical  notion  of  independence  complexes  of    graphs.  
Let  $k$  run  over  all  positive  integers.  
The  union  of  all  the  unordered  configuration  spaces  ${\rm  Conf}_k(X)/\Sigma_k$
is  a  simplicial  complex  with  vertices  in  $X$,  which  will  be  called   
the  {\it  independence  complex}  on  $X$  and  denoted  as   ${\rm  Ind}(X)$. 
In  other  words,  ${\rm  Ind}(X)$  is  the  simplicial  complex  whose  simplices  
are  finite   subsets  of   $X$.  
In  addition,  if  $X$  is  a  metric  space,  
then     ${\rm  Ind}(X)$     has  a  filtration   ${\rm  Ind}(X,-)$
 where  for  any  $r\geq  0$,  the  simplicial  complex  
   ${\rm  Ind}(X,r)$   is  given  by  the  union  of     ${\rm  Conf}_k(X,r)/\Sigma_k$
   for  all  $k\geq  1$.  
   In  particular,   let   $X=V$.  
   Then  
       ${\rm  Ind}(G,r)= {\rm  Ind}(V,  r)$
 where  the   right-hand   side  is  defined  
with  respect  to    $ d_G $.

As  an  analogue  of  ${\rm  Ind}(X)$,   
 we  define the  {\it  directed  independence  complex}  $\overrightarrow{\rm  Ind}(X)$  
 to  be  the  union  of  the  ordered  configuration  spaces   ${\rm  Conf}_k(X)$  for  all  $k\geq  1$.  
 In  addition,  if  $X$  is  a  metric  space,  
then     $\overrightarrow {\rm  Ind}(X)$     has  a  filtration   $\overrightarrow{\rm  Ind}(X,-)$
 where  for  any  $r\geq  0$,  
   $\overrightarrow{\rm  Ind}(X,r)$   is  given  by  the  union  of     ${\rm  Conf}_k(X,r)$
   for  all  $k\geq  1$.  
The  canonical  coverings  from the  ordered  configuration  spaces  to  the  
unordered  configuration  spaces  induce a  projection 
$\pi:  \overrightarrow {\rm  Ind}(X)\longrightarrow   {\rm  Ind}(X)$ 
such that  the  filtrations are  preserved.

 Let   $X$  be   a   differentiable  manifold  $M$.  
 In  \cite{jgp},   $\overrightarrow {\rm  Ind}(M)$   
 as  well as  $ {\rm  Ind}(M)$   is  regarded as  a  $\Delta$-manifold  
 and  thereby    double  chain  complexes  for  
 the  differential  forms  on     $\overrightarrow {\rm  Ind}(M)$   
   are  constructed.   In  addition,  if   $M$  is  a  submanifold  of  $\mathbb{R}$,  
   then  
  double  chain  complexes  for  
 the  differential  forms  on     $ {\rm  Ind}(M)$   
   are  constructed   as  well. 
   Let  $g$  be  a  Riemannian  metric  on  $M$  and  $d$  be  the geodesic  distance 
   induced  by  $g$.  
   Then  with  the  help  of  \cite{jgp}, 
   the  filtration  $\overrightarrow {\rm  Ind}(M,-)$ 
 induces  persistence  double  complexes  of  differential  forms  
and  if  $M$  is  a  submanifold  of  $\mathbb{R}$,
then   the  filtration  $ {\rm  Ind}(M,-)$ 
 induces  persistence  double  complexes  of  differential  forms.

\subsection{Hyper(di)graphs  on  manifolds  and  their  double  complexes}

Let  $V$  be  a  discrete  set  whose elements are vertices.  
Let  $\Delta[V]$  be  the  collection  of  all  the   non-empty  finite  subsets  of   $V$.  
 A   {\it   hypergraph}   $\mathcal{H}$  on   $V$  is  a  subset  of  $\Delta[V]$,  
 i.e.  $\mathcal{H}$  is  a  collection  of  certain   non-empty  finite  subsets  of   $V$.  
 The  elements  of  $\mathcal{H}$  are  {\it  hyperedges}
 (cf.  \cite{berge}).  
 In  particular,  if  
 for  any  $\sigma\in \mathcal{H}$  and  any   non-empty  subset  $\tau$  of   $\sigma$,  
 it  holds  that  $\tau\in \mathcal{H}$,  then  $\mathcal{H}$  is     a  
 {\it  simplicial  complex}  and  is  denoted  by  $\mathcal{K}$  (cf.  \cite{parks,h1});
 and  if    for  any  $\sigma\in \mathcal{H}$  and  any  finite  superset  $\tau$  of   $\sigma$,  
 it  holds  that  $\tau\in \mathcal{H}$,  then  $\mathcal{H}$  is  called  an  
 {\it  independence  hypergraph}  and  is  denoted  by  $\mathcal{L}$
 (cf.  \cite{camb2023,jktr2023,ca26}). 
 Note  that     $\Delta[V]$  is  a  simplicial  complex  and   is  also  an  independence  hypergraph. 
 Some  dualities  between  simplicial  complexes  and  independence  hypergraphs 
 are   investigated  by  the  present  author  \cite{camb2023, ca26} from  the  perspective  of   differential  calculus.   
 Inspired  by  the  Alexander  duality  of  (random)  simplicial  complexes  
 by  M. Farber, L. Mead and T. Nowik
 \cite{duality-jta}, 
 some  dualities  between 
  random  simplicial  complexes  and  random  independence  hypergraphs 
 are   investigated 
     by  C.  Wu,  J.  Wu  and  the  present  author 
       \cite{jktr2023}  from  the   perspective  of  the  map  algebra
       (cf.  \cite{jktr2}).

 Given  a  hypergraph  $\mathcal{H}$  on  $V$,  
 the   {\it  associated  simplicial  complex}  $\Delta\mathcal{H}$
   is  the  smallest  
simplicial  complex  containing  $\mathcal{H}$  (cf.  \cite{parks,h1}),  
the  {\it  lower-associated  simplicial  complex}  $\delta\mathcal{H}$  is  the  largest  
simplicial  complex  contained in   $\mathcal{H}$ (cf.  \cite{jktr2,jktr2022rel}),  
the   {\it  associated  independence  hypergraph}  $\bar\Delta\mathcal{H}$
   is  the  smallest  
independence  hypergraph  containing  $\mathcal{H}$ (cf.  \cite{jktr2023,ca26}),  
and  the  {\it   lower-associated  independence  hypergraph}  $\bar\delta\mathcal{H}$  is  the  largest  
independence  hypergraph  contained in   $\mathcal{H}$ (cf.  \cite{jktr2023,ca26}).

Let  $R_\bullet(\mathcal{H})$  be  the  graded  free  $R$-module  spanned  by  the  hyperedges 
of  $\mathcal{H}$.  
The  {\it  infimum  chain  complex}  ${\rm  Inf}_\bullet(\mathcal{H})$  
is  the  largest  sub-chain  complex  of  $C_\bullet(\Delta\mathcal{H})$
contained  in  $R_\bullet(\mathcal{H})$  and  the  
 {\it  supremum  chain  complex}  ${\rm  Sup}_\bullet(\mathcal{H})$  
is  the  smallest sub-chain  complex  of  $C_\bullet(\Delta\mathcal{H})$
containing  $R_\bullet(\mathcal{H})$  (cf.  \cite{h1}).  
In  2019,  S.  Bressan,   J.   Li,   J.  Wu  and  the  present  author 
  \cite[Sect.~3]{h1}  proved   that   the  canonical  inclusion  of  
${\rm  Inf}_\bullet(\mathcal{H})$   into  ${\rm  Sup}_\bullet(\mathcal{H})$ 
is  a  quasi-isomorphism. 
 In  particular,  if  $\mathcal{H}$  is  a  simplicial  complex  $\mathcal{K}$,  
 then  both ${\rm  Inf}_\bullet(\mathcal{K})$   and   ${\rm  Sup}_\bullet(\mathcal{K})$
 are  the  usual   chain  complexes  of  $\mathcal{K}$.

A  {\it  directed  hyperedge}  on  $V$  is  an  ordered  sequence  of  finite  distinct  vertices  in  $V$.  
A   {\it   hyperdigraph}   $\vec {\mathcal{H}}$  on   $V$  is  a
collection  of  directed  hyperedges  on  $V$
(cf.  \cite[Sect.~3.1]{hdg}).   
In  2023,  D. Chen,  J.  Liu,  G.~W.~ Wei  and  J.  Wu \cite{hdg}
studied  the     infimum  chain  complex  ${\rm  Inf}_\bullet(\vec{\mathcal{H}})$
as  well  as  the   supremum  chain  complex  ${\rm  Sup}_\bullet(\vec{\mathcal{H}})$
for  hyperdigraphs  and  proved  that  they  are  quasi-isomorphic. 
 It  is   proved  in  \cite[Lemma~3.8  and  Theorem~3.9]{hdg}
 that  the  canonical  projection  from  a  hyperdigraph  to  its  underlying  
 hypergraph  induces     chain  maps    between  the  infimum/supremum  
 chain  complexes.

Consider  a  family  of    embeddings  
$\{\varphi_\lambda:  V\longrightarrow  X\mid  \lambda\in  \Lambda\}$,
where  $\Lambda$  is  an  index  set.   
Then   
a  hypergraph  $\mathcal{H}$  on  $V$  together  with  the  images  of   the  family  of  embeddings 
${\rm  Im}\{\varphi_\lambda\mid  \lambda\in \Lambda\}$  will
   be  called   a  hypergraph  on  $X$,   denoted  as  $\mathcal{H}(X)$.  
In  other  words,  a  hypergraph  $\mathcal{H}(X)$  on   $X$   is  a  hypergraph 
 with  its  vertices  moving  in  $X$  (cf.  \cite{jgp, pmlr26}).   
A  hyperedge  $\sigma(X)\in  \mathcal{H}(X)$  is  a  finite  set  of  distinct  points  in  $X$.   
 Thus  $\sigma(X)$  is  an  element  of    ${\rm   Ind}(X)$.  
Hence   $\mathcal{H}(X)$  is a  sub-hypergraph  of  ${\rm   Ind}(X)$.   
Therefore,  
${\rm   Ind}(X)$  is  the  largest  hypergraph  on  $X$.  
Note  that    ${\rm   Ind}(X)$  is  a  simplicial  complex  on  $X$  and   is  also  
an  independence  hypergraph  on   $X$.  
The  filtration  ${\rm   Ind}(X,-)$   induces  a  filtration  
  $\mathcal{H}(X,-)$  where    $\mathcal{H}(X,r)$  is  the    
 intersection  of    $\mathcal{H}(X)$   and  ${\rm   Ind}(X,r)$  for  any  $r\geq  0$.

The  space  of  all  the  finite  $r$-ball  packings  in  $(X,d)$  
is  the  independence  complex  ${\rm  Ind}(X,r)$,  
which  is  a  simplicial  complex  on  $X$.  
On  the  other  hand,  
the  space  of  all  the  finite  $r$-ball  coverings  of   $(X,d)$ 
is  an  independence  hypergraph  on  $X$  (cf.  Subsect.~\ref{ss-4.4-cover}).  
 The  maximal   packing  problem  of  $X$  and  the  minimal  covering  problem  of  $X$
 have  dualities  in  optimization  theory.  
 Hopefully,  
 this  would   give      potential      dualities 
   between  simplicial complexes and  independence  hypergraphs
 in  addition  to  the  dualities  investigated  in  \cite{camb2023,ca26,duality-jta,jktr2023}.

 Given  a  differentiable  manifold  $M$,  
 a  hyperdigraph  $\vec{\mathcal{H}}(M)$  on  $M$  is  a  graded  submanifold  of  
 $\overrightarrow{\rm   Ind}(M)$  
  and  a  hypergraph  $\mathcal{H}(M)$  on  $M$  is  a  graded  submanifold  of  
  ${\rm  Ind}(M)$.  
  The  canonical  projection  from  $\overrightarrow{\rm   Ind}(M)$    to  
  ${\rm  Ind}(M)$   induces  the   
  canonical  projection  from   $\vec{\mathcal{H}}(M)$  to 
  $\mathcal{H}(M)$,  
  where  $\mathcal{H}(M)$  is   the  underlying  hypergraph 
   of   $\vec{\mathcal{H}}(M)$.
  In  \cite{jgp},  
  by  considering  the  spaces  of  differential  forms,  
  the  present  author  generalized  the  
   infimum  chain  complexes   
as  well  as  the   supremum  chain  complexes  for  hyper(di)graphs 
and constructed  the  infimum  double   complexes   
as  well  as  the   supremum  double   complexes
for  hyper(di)graphs  on  $M$.  
In  particular,  if  
$M$  is  a  discrete  set   $V$,  
 then   the  infimum  double   complexes   
as  well  as  the   supremum  double   complexes 
reduce   to  the  usual  
infimum  chain   complexes   
as  well  as  the   usual  supremum  chain   complexes.  
 A   brief  explanation  for  the motivations  and  the  applications  
of  the  double  complexes  of  hyper(di)graphs  on  manifolds   
 was  given  in  \cite[Parts III,  IV]{pmlr26}.  
  From  a  perspective  of 
    computer science,   J.A.   Stokke,   R.  Bergmann,  M.  Hanik    and  C.  von Tycowicz 
 \cite{app1}  recently  investigated 
the  $p$-Laplacians  for  the  manifold-valued  hypergraphs.

\subsection{Results  of  this  paper}

In  this  paper,  we  develop  \cite{jgp} 
and  study  the   double  homology   as  well  as  the  spaces  of  differential  forms 
 for  hyper(di)graphs  on  Riemannian  manifolds.  
 We  discuss    the relations  between  the  fundamental  group(oid)s  of 
 hyper(di)graphs  on  surfaces  and  the  braid  group(oid)s  on  surfaces.  
 As  applications,  we  study  the  hyper(di)graphs  of  packings  and  the  
 hyper(di)graphs  of  coverings.  We  use  the  geometric  realizations  of  
 hypergraphs    to  characterize  certain  regular  embeddings  of  manifolds.

Firstly,  we    study  the  double  homology  of  hyper(di)graphs  on  manifolds  
by  using  the  models  of      $\Delta$-manifolds  (cf.   \cite[Section~3]{jgp}).  
We  define    $\overrightarrow {\bf  FI}$  
to  be   the  category  whose  objects  are   finite  sequences  of  distinct  vertices 
and  whose  morphisms are  injections   preserving  the  orders  (cf.   
(\ref{eq-26.3.17.m1})).  
 A   ${\bf  FI}$-manifold  (resp.  $\overrightarrow{\bf  FI}$-manifold)
   is  a  covariant  functor  from  ${\bf  FI}$  (resp.  $\overrightarrow{\bf  FI}$)
 to  the  category  of  manifolds. 
 Then  ${\rm  Ind}(M)$  is  a   ${\bf  FI}$-manifold 
 and  $\overrightarrow {\rm  Ind}(M)$  is  a   $\overrightarrow{\bf  FI}$-manifold.   
 Consequently,  the
   $\Delta$-manifold  structure     on   $\overrightarrow {\rm  Ind}(M)$ 
 is  derived  from  the     $\overrightarrow{\bf  FI}$-manifold  structure   
 and   the
   $\Delta$-manifold   structure  on   $ {\rm  Ind}(M)$ 
 is  derived  from  the     $ {\bf  FI}$-manifold  structure.
 A   hypergraph  (resp.  hyperdigraph)   on  $M$   can  be  regarded  as  a  covariant  functor  from  a  
subcategory  of  ${\bf  FI}$   (resp.  $\overrightarrow{\bf  FI}$)  to  the  category  of  manifolds.   
We  prove  a   commutative  diagram  of  the  double  homology  groups  
 of  the   hyper(di)graphs  on  $M$   as  well  as  their   associated  (directed)
 simplicial  complexes  on  $M$  and  lower-associated  (directed)
 simplicial  complexes  on  $M$,  in  Theorem~\ref{th-6.13.1}  (Main  Result  I).

  Besides  the  double  homology  in  Theorem~\ref{th-6.13.1},  
  we  study  the  fundamental  groups  as  well as  the  fundamental  groupoids  
  of  hyper(di)graphs  on  surfaces  
  from  the  perspective   of  braid  groups.
  We  prove  that  the  canonical  inclusions  of  
    hyper(di)graphs  on  surfaces   into  the  configuration  spaces  of  surfaces
      induce
    homomorphisms  from the  fundamental  group(oid)s   of  
    the  hyper(di)graphs  on  surfaces  to  the  braid  group(iod)s   on   surfaces,    
  in  Proposition~\ref{pr-4.23.1}
  and  Proposition~\ref{pr-4.23.2}.

  As  applications  of  hypergraphs  on   $M$,  
  we  study  regular  maps  on  $M$  
  from  the  perspective  of  the  geometric  realizations  of  
  hypergraphs  on  $M$.  
  We  generalize  the  notion  of  geometric  realizations  of  
  hypergraphs  on  discrete  sets  and  define  the   notion  of  geometric  realizations  
  of  hypergraphs  on  $M$.  
  We  generalize  the  simplicial  characterizations  of  regular  embeddings  of  graphs  in  
  \cite{reg}  and  construct   the  simplicial  characterizations  of  regular  embeddings  of  manifolds.   
  We  prove  in  Proposition~\ref{pr-4.18-1}  
  that  a  $k$-regular  embedding  on  $M$  is  equivalent to  a  simplicial embedding 
  of  the  $(k-1)$-skeleton  of  the  independence  complex  on   $M$    
  into  the  orbit  space  of  
  $\Sigma_k$  in  the  Stiefel  manifold.
  Generalizing    the  notion  of  
  $k$-regular  maps  and  Proposition~\ref{pr-4.18-1},    
  we  prove in  Proposition~\ref{pr-4.19-01}  that  
  a  geometric realization  of  a  hypergraph  $\mathcal{H}(M)$  
  is  equivalent  to  a  simplicial  embedding  of  $\mathcal{H}(M)$
    into  the  orbit  space  of  
  $\Sigma_k$  in  the  Stiefel  manifold.

  Secondly,  
  motivated  by  the  double  complexes  of  differential  forms  in  
  \cite{jgp},  
  we   further  explore  the  space  of  differential  forms  for  hyper(di)graphs  
  on  manifolds.  
For    any  hyperdigraph  
$\vec{\mathcal{H}}(M)$  on  $M$  with  its  underlying  hypergraph 
$\mathcal{H}(M)$,  
let  ${\rm  Sym}(\vec{\mathcal{H}}(M))$  be  the  largest  
hyperdigraph  whose  underlying  
hypergraph  is
$\mathcal{H}(M)$   (cf.    eq.  
(\ref{eq-26.7.27})).   
By  considering  the   associated  (directed)   simplicial  complexes  
and  the  lower-associated  (directed)   simplicial  complexes  for  
$\vec{\mathcal{H}}(M)$,   ${\rm  Sym}(\vec{\mathcal{H}}(M))$  and  
$\mathcal{H}(M)$   respectively,  
 we  prove   a   commutative  diagram    with   three  rows  and  four  columns
  (cf.   diagram  (\ref{eq-5.7.21}))   
 of  double-graded  vector  spaces of   differential  forms 
 in   Theorem~\ref{th-26.5.3.1}  (Main  Result  II), 
   where     
 the  last  three  spaces  in  the  second  row  may  not  be  
 double  complexes  and  all  the  remaining  spaces  are  double  complexes.

    Thirdly,   suppose  in  addition  that   $M$  has  a  Riemannian  metric.  
  With  respect  to  the  induced  geodesic  distance  on  $M$,   
     the  configuration  spaces  of  $M$  have  induced  filtrations  
   ${\rm  Conf}_k(M,(-,-]) $  (cf.  eq.   (\ref{eq-filt-1}))   and  $ {\rm  Conf}_k(M,(-,-])/\Sigma_k$ 
   (cf.  eq.   (\ref{eq-filt-2})).    
   We  consider  the  homotopy  types  of  
   ${\rm  Conf}_k(M,(-,-]) $  and  $ {\rm  Conf}_k(M,(-,-])/\Sigma_k$.
   We   
    prove  in  Theorem~\ref{pr-4.5.1}    
    that  the  set  of   critical  points  where  the  homotopy  types  of  
    ${\rm  Conf}_k(M,(-,-]) $  change  equals    
    the  set  of   critical  points  where  the  homotopy  types  of  
    ${\rm  Conf}_k(M,(-,-])/\Sigma_k $  change. 
    Moreover, 
    all  homological  critical  points  of  the  persistence  homology  of  
    ${\rm  Conf}_k(M,(-,-]) $  and  $ {\rm  Conf}_k(M,(-,-])/\Sigma_k$ 
    are  contained  in  the  above  set  of   the  critical  points.

    As  examples  of  hypergraphs  on  $M$,  
    we  consider  the  space  of  all  the  packings  of  closed     $r$-neighborhoods  in  $M$, 
    which  is  a  simplicial  complex  on  $M$,  in  Subsection~\ref{ss-26.7.27.1}.  
    We  also  consider  the  space  of    
      all  the  coverings  of  open    $r$-neighborhoods  in  $M$, 
    which  is  an  independence  hypergraph  on  $M$,  in  Subsection~\ref{ss-4.4-cover}.  
    Given  a  hypergraph   $\mathcal{H}(M)$  on  $M$,  
    we  consider  the  hyperedges  in  $\mathcal{H}(M)$
      which  are  packings  of  closed     $r$-neighborhoods  in  $M$
      and  consider    the  hyperedges  in  $\mathcal{H}(M)$
      which  are  coverings   of  open    $r$-neighborhoods  in  $M$.
      We  call  such  packings  the  {\it  packings of  closed     $r$-neighborhoods 
       with  constraint  $\mathcal{H}(M)$}   (cf.  Subsection~\ref{ss-26.7.27-2})
       and call  such  coverings 
       the  {\it  coverings of  open     $r$-neighborhoods 
       with  constraint  $\mathcal{H}(M)$}     
       (cf.  Subsection~\ref{ss-26.7.27.3}). 
             As  $r$  varies,  
      the  filtration      $ {\rm  Conf}_k(M,(-,-])/\Sigma_k$ 
       induces  filtrations  on  the  space  of  packings of  closed     $r$-neighborhoods 
       with  constraint  $\mathcal{H}(M)$ as well  as  
         filtrations  on  the  space  of
         coverings of  open     $r$-neighborhoods 
       with  constraint  $\mathcal{H}(M)$.  
  The   critical  points   where  the  homotopy  types  change  
  and  homological  critical  points  of  the  persistence  homology
  apply  to  these  induced  filtrations  as  well.

\subsection{Organization}

In   Section~\ref{sect.26.7.27.2},  
we  study  the  double  (co)homology  as  well  as  
the  fundamental  group(oid)s  of  graded  submanifolds  of  $\Delta$-manifolds,
 which  gives   algebraic  foundations  for   the  double  homology  of  hyper(di)graphs  on  manifolds 
 in  Section~\ref{s.26.7.28.3}
   and  the  fundamental  group(oid)s  of  hyper(di)graphs  on  manifolds  in  Section~\ref{s.26.7.28.6}.

 In  Sections~\ref{s.26.7.28.3}-\ref{s.26.7.28.7},
 we  study  the  double  homology  and  the  fundamental  group(oid)s 
 for  hyper(di)graphs  on  manifolds.    In   Section~\ref{s.26.7.28.3},  
we  study the  double  chain  complexes  as  well  as  the  double homology  of  
configuration  spaces  and  hyper(di)graphs  on  manifolds  from  the  perspective  of  
${\bf   FI}$-manifolds,  $\overrightarrow{\bf   FI}$-manifolds  and  $\Delta$-manifolds.  
 In   Section~\ref{s.27.7.28.4},  
 we  study  the  space  of  packings   of  closed  $r$-neighborhoods  with  certain  constraints and 
 the  space  of  coverings  of  open  $r$-neighborhoods  with  certain  constraints  
 on  Riemannian  manifolds, 
 as  examples  for  Section~\ref{s.26.7.28.3}.  
 In  Section~\ref{s-27.7.28.5}, 
 we  study  the   morphisms  bewteen hyper(di)graphs  induced  by  maps  between  manifolds.  
 In    particular,  we  study  the  morphisms  between  the  spaces  of  packings  
 as  well  as   the  morphisms  between  the  spaces  of  coverings 
 induced  by  isometric  maps  between  Riemannian  manifolds.  
 In   Section~\ref{s.26.7.28.6},  
 we  study  the  fundamental  group(oid)s  of  hyper(di)graphs  on  surfaces  
 as  well as  their  relations  with  the   braid  groups on  surfaces.  
In   Section~\ref{s.26.7.28.7},  
we  study  the  hyper(di)graphs  of  frames  on  manifolds  as   extended  examples
for  Section~\ref{s.26.7.28.3},  
where  the  underlying  space  of  the  vertices  is  an  
 infinite-dimensional    space  of  cross-sections  of  tangent  bundles.

In   Section~\ref{s-27.7.28.8},  
we  characterize  the  regular  embedding  problem  of  manifolds  
equivalently   in  terms  of   geometric  realizations  of  skeletons  of  the  independence  complexes.    
 With  the  help  of  Section~\ref{s-27.7.28.5}  and  Section~\ref{s.26.7.28.7},   
  we  characterize  geometric  realizations  of  hypergraphs  on  manifolds 
  equivalently  by  injective   morphisms  of  hypergraphs  into  the  
  space  of  frames  in  Euclidean  spaces.

  In   Section~\ref{26.6.5-s7},  
  we  study  the  space  of  differential  forms  for  hyper(di)graphs  on  manifolds  as  well  
  as  their  associated  (directed)  simplicial  complexes  and  lower-associated  
  (directed)  simplicial  complexes.  
  We  prove   a  commutative diagram  of  these  spaces  in  Theorem~\ref{th-26.5.3.1}.

\section{Homology  of  $\Delta$-manifolds  and  their  graded  submanifolds}
\label{sect.26.7.27.2}

In  this  section,  we  construct  double  homology groups  for 
$\Delta$-manifolds  as  well as  their  graded  submanifolds.  
In  Subsection~\ref{26.5-ss-2.1},  
 we  review  the  chain  complexes  of  $\Delta$-sets  
 and  the  infimum/supremum  chain  complexes  of  graded  subsets  of  
 $\Delta$-sets.  
 In  Subsection~\ref{ss-2.fi}, 
  we  construct  $\Delta$-sets  for   ${\bf  FI}$   categories 
  and  consequently  construct 
    the  infimum/supremum  chain  complexes  for   subcategories  of  
       ${\bf  FI}$   categories.   
 In  Subsection~\ref{ss-2.2-26.5.22},  
 we   give  the 
chain  complexes  of  $\Delta$-modules  
 and  the  infimum/supremum  chain  complexes  of  graded  submodules  of  
 $\Delta$-modules.
   In  Subsection~\ref{ss2.3-26.5.22},  
   we  prove  that the  homology  of     $\Delta$-manifolds  are  $\Delta$-modules    
   and  thereby    construct  double  homology groups   for  $\Delta$-manifolds 
   as  well  as  their  graded  submanifolds.

\subsection{Chain  complexes  for  
$\Delta$-sets  and  their  graded  subsets}\label{26.5-ss-2.1}

A  {\it  $\Delta$-set}   $X $  is  a  collection  of  sets  $X_\bullet=\{X_n\}_{n\geq  0}$  
together  with  maps  $\partial_n^i:  X_n\longrightarrow  X_{n-1}$  
for  any  $n\geq  1$  and  any   $0\leq  i\leq  n$
such  that  
\begin{eqnarray}\label{eq-26.5.17.55}
\partial^i_{n-1}\partial^j_n= \partial^{j-1}_{n-1}\partial^i_n 
\end{eqnarray}
  for  any  $0\leq i<j\leq  n$.  
  Let  $X $  be a  $\Delta$-set. 
  If  necessary,  
  we  may  regard     $X$  as  a category  whose  objects  are  
 $X_n$,  where  $n\geq  0$,  and  whose  morphisms are  
 the  maps  $\partial_n^i$,   where  $n\geq  1$  and      $0\leq  i\leq  n$.  
A  graded  subset  $U $  of  $X $   is  a  collection  of  sets  
$U_\bullet=\{U_n\}_{n\geq  0}$   such  that  $U_n\subseteq  X_n$  for  any  $n\geq  0$.  
Let  $\Delta   U$  be  the  smallest   $\Delta$-subset  of  $X$  containing  $U$.  
 Let  $\delta  U$  be  the  largest  $\Delta$-subset  of  $X$  contained  in  $U$.

The  embedded  homology  of  hypergraphs  \cite{h1}
has  been  generalized  to  the  embedded  homology  of   graded  subsets  of  $\Delta$-sets  
by  C.  Wang  and  the  present  author    \cite[Sections~2-3]{deltaset1}.   
Let  $R$  be  a  commutative  ring  with  unit.  
Let  $R(X )=\bigoplus_{n\geq  0} R(X_n)$ 
 be  the  graded  free  $R$-module  spanned  by  $X $. 
 Then   $R(U )=\bigoplus_{n\geq  0} R(U_n)$
 is  a  graded  sub-$R$-module  of    $R(X )$.  
 Extend  each  $\partial_n^i$  to  an  $R$-linear  map  
 $\partial_n^i:  R(X_n)\longrightarrow  R(X_{n-1})$  and  take  the   alternating  sum    
 \begin{eqnarray}\label{eq-26.5.22-3}
 \partial_n=\sum_{i=0}^n (-1)^i \partial_n^i.
 \end{eqnarray} 
 Then  $\partial_{n-1}\partial_n=0$.   We  obtain   a  chain  complex  
 $(R(X ),\partial_\bullet)$.  
 By  \cite[Def.~1 and  Def.~2]{h1},   
 the    infimum  chain  complex  
 \begin{eqnarray*}
 {\rm  Inf}_n(U ,X )= R(U_n)\cap \partial_n^{-1}(R(U_{n-1})),~~~~~~  n\geq  0, 
 \end{eqnarray*}
  is  the  largest  sub-chain  complex  of   $(R(X ),\partial_\bullet)$
 contained  in  $R(U )$,     
and    the   supremum  chain  complex  
 \begin{eqnarray*}
 {\rm  Sup}_n(U ,X )= R(U_n)+ \partial_{n+1}(R(U_{n+1})),  ~~~~~~  n\geq  0,
 \end{eqnarray*}
  is  the  smallest  sub-chain  complex  of   $(R(X ),\partial_\bullet)$
 containing   $R(U )$.        
      Let  $X'$  be  another  $\Delta$-set   such  that  
   $U$  is  a  graded  subset  of  both  $X$  and  $X'$.  
   By  \cite[Prop.~3.2]{h1}, 
    \begin{eqnarray}\label{eq-26-3-15-h3}
 {\rm  Inf}_\bullet(U ,X )=  {\rm  Inf}_\bullet(U ,X' ), ~~~~~~ 
 {\rm  Sup}_\bullet(U ,X )=  {\rm  Sup}_\bullet(U ,X' )
 \end{eqnarray}
 Therefore,  
 the   chain  complexes (\ref{eq-26-3-15-h3})
 only  depend   on  $U$,  which  will  be    
   simply  denoted  as  ${\rm  Inf}_\bullet(U)$
 and  ${\rm  Sup}_\bullet(U)$  respectively.

Let  $Y$  be  another  $\Delta$-set  given  by  $Y_\bullet=\{Y_n\}_{n\geq  0}$.    
A  {\it  $\Delta$-map}  $f:  X\longrightarrow  Y$ is  a  sequence  of  maps  
$f_\bullet:  X_\bullet\longrightarrow  Y_\bullet$   such  that  
\begin{eqnarray}\label{eq-26.5.20.9}
f_{n-1}\partial_n^i= \partial_{n}^i  f_n
\end{eqnarray}
  for  any  $n\geq  1$ and  any  $0\leq  i\leq  n$.  
Let   $f:  X\longrightarrow  Y$  be  a  $\Delta$-map.   
Then  $f(U)$  is  a  graded  subset  of  $Y$.  
Moreover,  for  any  graded  subset  $W$  of  $Y$,
the  preimage  $f^{-1}(W)$  is  a  graded  subset  of $X$. 
 If  $U$  is  a  sub-$\Delta$-set  of  $X$,  
then  $f(U)$  is a  sub-$\Delta$-set  of  $Y$.  Conversely,  
if  $W$  is  a  sub-$\Delta$-set  of  $Y$, then  
$f^{-1}(W)$   is  a  sub-$\Delta$-set  of  $X$.  
Therefore,  $f(\delta  U)=  \delta   f(U)$   and   $f(\Delta  U)=\Delta  f(U)$.

 \begin{proposition}\label{le-2.2.inc}
 For  any  graded  subset  $U$  of  any  $\Delta$-set  $X$
 and  any  $\Delta$-map  $f: X\longrightarrow  Y$,  
 we  have  a   commutative  diagram   of  chain  complexes  
 \begin{eqnarray}\label{eq-26.5.18.qu2}
 \xymatrix{
 R(\delta  U)\ar[r]\ar[d]
 &  {\rm  Inf}_\bullet(U)\ar[r]\ar[d]
 &   {\rm   Sup}_\bullet(U)\ar[r]\ar[d] 
 &  R(\Delta  U) \ar[r]\ar[d] 
 &R(X)\ar[d]^-{f_\#}\\
    R(\delta  f(U))\ar[r] 
 &  {\rm  Inf}_\bullet(f(U))\ar[r] 
 &   {\rm   Sup}_\bullet(f(U))\ar[r] 
 &  R(\Delta  f(U))\ar[r]
 &R(Y)
  }  
 \end{eqnarray}
 where 
 \begin{enumerate}[(i)]
 \item
  in  each  row,  the  second  map  is  a  quasi-isomorphism
 and  all  the  horizontal  maps  are  inclusions;
 \item
 all  the  vertical  maps  are  induced  by  $f$,
 which  implies  that  the  first  row  of  (\ref{eq-26.5.18.qu2})  is   functorial  with  respect  to  
$\Delta$-maps  between  $\Delta$-sets.   
 \end{enumerate}  
 \end{proposition}

\begin{proof}
By  \cite[Prop. 2.4]{h1}, 
 the  canonical  inclusion  
 $
 {\rm  Inf}_\bullet(U)\longrightarrow  
  {\rm  Sup}_\bullet(U)
 $
 induces  an  isomorphism  of  homology  groups. 
 Thus  in  each  row  of  (\ref{eq-26.5.18.qu2}),
   the  second   map  
  is  a  quasi-isomorphism. 
  It  is  straightforward   that  
    the  diagram  (\ref{eq-26.5.18.qu2})   commutes.  
  Hence  the  first  row  of   (\ref{eq-26.5.18.qu2})
  is   functorial   
with  respect to  $f$.  
\end{proof}

\subsection{Chain  complexes  for  ${\bf  FI}$   categories  and  their   subcategories}
\label{ss-2.fi}

Church-Eilenberg-Farb  \cite{fi2015}  used  
${\bf  FI}$  for  the  category  whose  objects  are  all  finite  sets 
$\{1,2,\ldots,k\}$  
 and  whose  morphisms are  all  injective  maps.  
In  this  subsection,  
we   let   $\overrightarrow{\bf  FI}$  denote  the 
category  whose 
  objects  are  all   finite  sequences    
$(i_0,i_1,\ldots,i_n)$  of  distinct  elements  
 and  whose  morphisms  from  
$  (i_0,i_1,\ldots,i_n)$  to   $  (j_0,j_1,\ldots,j_m)$  
are  all  injective  maps between  the  underlying  sets
\begin{eqnarray} \label{eq-26.3.17.m1}
\vec\iota:   \{i_0,i_1,\ldots,i_n\}\longrightarrow  \{j_0,j_1,\ldots,j_m\}
 \end{eqnarray}
preserving  the  orders  of  the  sequences.  
We  call  $i_0,i_1,\ldots,i_n$  the  {\it  vertices}.  
Here  $n\leq  m$,
$(i_0,i_1,\ldots,i_n)$  is  a  permutation  of  $(0,1,\ldots,n)$,
$(j_0,j_1,\ldots,j_m)$  is  a  permutation  of  $(0,1,\ldots,m)$,
and     $\vec\iota$  in  (\ref{eq-26.3.17.m1})   an  injection  such  that  
\begin{eqnarray*}
\vec\iota(i_0,i_1,\ldots,i_n)=(\vec\iota(i_0),\vec\iota(i_1),\ldots,\vec\iota(i_n)) 
\end{eqnarray*}
  is  a  subsequence  of    $(j_0,j_1,\ldots,j_m)$.  
The  canonical  projection   
\begin{eqnarray}\label{eq-3-16-p1}
 \pi:  \overrightarrow{\bf  FI}\longrightarrow   {\bf  FI}
\end{eqnarray}
 is  a  functor  
sending  a  finite  sequence  $(i_0,i_1,\ldots,i_n)$  of  distinct  elements
to  its  underlying  set  $\{0,1,\ldots,n\}$  
and  sending  an  injection   $\vec\iota$  of   sequences  
to  the  induced  injection  of  the  underlying  sets 
\begin{eqnarray}\label{eq-26.3.17.m2} 
\iota:   \{0,1,\ldots, n\}\longrightarrow  \{0,1,\ldots, m\}.  
 \end{eqnarray}

 \begin{lemma}\label{pr-26.5.23.fi1}
 \begin{enumerate}[(1)]
 \item
 The    collection  ${\rm  Obj}_\bullet(\overrightarrow{\bf  FI})$  of  the   objects  
 in  $\overrightarrow{\bf  FI}$  
  is     a      $\Delta$-set.  Consequently,  
  we  have  a   chain  complex 
   $(R({\rm  Obj}_\bullet(\overrightarrow{\bf  FI})),  \partial_\bullet)$; 
  \item
  The  collection  ${\rm  Obj}_\bullet({\bf  FI})$  of  the  objects  in  ${\bf  FI}$  
  is     a      $\Delta$-set. 
   Consequently,
   with  respect  to  the  canonical   total  order on  positive  integers,
   we  have  a   chain  complex 
   $(R({\rm  Obj}_\bullet({\bf  FI})),  \partial_\bullet)$. 
  \end{enumerate}
 \end{lemma}

\begin{proof}
(1)  
For  any $n\geq  0$,   let  
$ {\rm  Obj}_n(\overrightarrow{\bf  FI})$  be  the  
 collection  of   $(n+1)$-sequences    
$(i_0,i_1,\ldots,i_n)$  of  distinct  elements. 
Let    
$ {\rm  Obj}_{-1}(\overrightarrow{\bf  FI})=\{\emptyset\}$.  
Then    
\begin{eqnarray}\label{eq-26-3.13-1}
{\rm  Obj}(\overrightarrow{\bf  FI})=\bigcup_{n\geq  -1} {\rm  Obj}_n(\overrightarrow{\bf  FI}). 
\end{eqnarray}
 For   any  $0\leq  \lambda\leq  n$,  
 define  the   face  map 
   \begin{eqnarray*} 
    \partial_n^\lambda: {\rm  Obj}_n(\overrightarrow{\bf  FI})\longrightarrow 
    {\rm  Obj}_{n-1}(\overrightarrow{\bf  FI})
    \end{eqnarray*}
    by 
 \begin{eqnarray*}
 \partial_n^\lambda:   (i_0,i_1,\ldots,i_n)\longrightarrow(i_0,\ldots, \widehat{i_\lambda}, \ldots,i_n). 
 \end{eqnarray*}
  Then  (\ref{eq-26-3.13-1})   
  is    a    $\Delta$-set.  
  By  Subsection~\ref{26.5-ss-2.1},  we  have  a   chain  complex 
   $(R({\rm  Obj}_\bullet(\overrightarrow{\bf  FI})),  \partial_\bullet)$
   where  $\partial_n=\sum_{\lambda=0}^n(-1)^{\lambda}\partial_n^\lambda$
     for any  $n\geq  0$.

 (2)   
 For  any $n\geq  0$,   let  
$ {\rm  Obj}_n({\bf  FI})$  be  the  
 collection  of   $(n+1)$-sets    
$\{0,1,\ldots,n\}$. 
Let    
$ {\rm  Obj}_{-1}({\bf  FI})=\{\emptyset\}$.  
Then    
\begin{eqnarray}\label{eq-26-3.13-2}
{\rm  Obj}({\bf  FI})=\bigcup_{n\geq  -1} {\rm  Obj}_n({\bf  FI}). 
\end{eqnarray}
 For    any  $0\leq  \lambda\leq  n$,  
 define  the   face  map 
   \begin{eqnarray*}
    \partial_n^\lambda: {\rm  Obj}_n({\bf  FI})\longrightarrow 
    {\rm  Obj}_{n-1}({\bf  FI})
    \end{eqnarray*}
    by 
 \begin{eqnarray*}
 \partial_n^\lambda:   \{0,1,\ldots,n\}\longrightarrow\{0,\ldots,\widehat{\lambda},\ldots,n\}. 
 \end{eqnarray*}
  Then  (\ref{eq-26-3.13-2})   
  is    a    $\Delta$-set.   
  By  Subsection~\ref{26.5-ss-2.1},  we  have  a   chain  complex 
   $(R({\rm  Obj}_\bullet({\bf  FI})),  \partial_\bullet)$
   where  the  alternating  sum   $\partial_n$  is  given 
    with  respect  to  the  canonical   total  order on  positive  integers.  
   \end{proof}
   
   \begin{lemma}\label{le-3.16-1}
 The  projection  $\pi$  in  (\ref{eq-3-16-p1})  induces  a  projection  chain  map  
 \begin{eqnarray}\label{eq-26.3.16.p9}
 \pi_\#:  R({\rm  Obj}_\bullet(\overrightarrow{\bf  FI})) \longrightarrow
   R({\rm  Obj}_\bullet({\bf  FI}))
 \end{eqnarray}
 given  by  
  \begin{eqnarray*}
 \pi_\#(s(0),s(1),\ldots,s(n))={\rm  sgn}(s) \{0,1,\ldots,n\}
 \end{eqnarray*}
 for  any  permutation  $s\in \Sigma_{n+1}$  and  any  $n\geq  0$.  
 \end{lemma}
 
 \begin{proof}
 The  proof  is  an  analogue   of   \cite[Lemma~3.8  and  Theorem~3.9]{hdg}   
 and    \cite[Proposition~4.3]{reg}.  
 \end{proof}

Let  $\overrightarrow{\bf  HI}$  be  a  subcategory  of  $\overrightarrow{\bf  FI}$.  
The    collection  of   objects  of  $\overrightarrow{\bf  HI}$  
  is     a  graded  subset  of      (\ref{eq-26-3.13-1}),  i.e.   
\begin{eqnarray*}
{\rm   Obj}(\overrightarrow{\bf  HI})=\bigcup_{n\geq  -1}{\rm   Obj}_n(\overrightarrow{\bf  HI}), 
\end{eqnarray*}
where  for  any  $n\geq  -1$,  
\begin{eqnarray*}
{\rm   Obj}_n(\overrightarrow{\bf  HI})={\rm   Obj}_n(\overrightarrow{\bf  FI})\cap
  {\rm   Obj}(\overrightarrow{\bf  HI}).
\end{eqnarray*}
Define   $\Delta  \overrightarrow{\bf  HI}$  
to  be   the  smallest   subcategory  of  
$\overrightarrow{\bf  FI}$   containing   $\overrightarrow{\bf  HI}$   
such  that   
   ${\rm   Obj}_\bullet(\Delta  \overrightarrow{\bf  HI})$  
 is   a   sub-$\Delta$-set  of  (\ref{eq-26-3.13-1}).  
 Define   $\delta  \overrightarrow{\bf  HI}$  
to  be   the  largest   subcategory  of  
$\overrightarrow{\bf  HI}$   
such  that     ${\rm   Obj}_\bullet(\delta  \overrightarrow{\bf  HI})$  
 is   a   sub-$\Delta$-set  of  (\ref{eq-26-3.13-1}).  
By  Proposition~\ref{le-2.2.inc}, 
we  have   inclusions  of   chain  complexes  
\begin{eqnarray*}
R({\rm   Obj}_\bullet (\delta  \overrightarrow{\bf  HI}))\longrightarrow 
{\rm   Inf}_\bullet  ( \overrightarrow{\bf  HI}) \longrightarrow 
{\rm   Sup}_\bullet  ( \overrightarrow{\bf  HI}) \longrightarrow 
R({\rm   Obj}_\bullet (\Delta  \overrightarrow{\bf  HI}) )\longrightarrow 
R({\rm  Obj}_\bullet(\overrightarrow{\bf  FI})).  
\end{eqnarray*}

Let  ${\bf  HI}$  be  a  subcategory  of  ${\bf  FI}$. 
The    collection  of   objects  of  ${\bf  HI}$  
  is     a  graded  subset  of      (\ref{eq-26-3.13-2}),  i.e.   
\begin{eqnarray*}
{\rm   Obj}({\bf  HI})=\bigcup_{n\geq  -1}{\rm   Obj}_n( {\bf  HI}), 
\end{eqnarray*}
where  for  any  $n\geq  -1$,  
\begin{eqnarray*}
{\rm   Obj}_n( {\bf  HI})={\rm   Obj}_n( {\bf  FI})\cap
  {\rm   Obj}( {\bf  HI}).
\end{eqnarray*}
Define   $\Delta   {\bf  HI}$  
to  be   the  smallest   subcategory  of  
$ {\bf  FI}$   containing  $ {\bf  HI}$  
such  that  
   ${\rm   Obj}_\bullet(\Delta  {\bf  HI})$  
 is   a   sub-$\Delta$-set  of  (\ref{eq-26-3.13-2}).  
 Define   $\delta  {\bf  HI}$  
to  be   the  largest   subcategory  of  
$ {\bf  HI}$   
such  that     ${\rm   Obj}_\bullet(\delta   {\bf  HI})$  
 is   a   sub-$\Delta$-set  of  (\ref{eq-26-3.13-2}). 
 By  Proposition~\ref{le-2.2.inc},
 we  have   inclusions  of   chain  complexes  
\begin{eqnarray*}
R({\rm   Obj}_\bullet (\delta  {\bf  HI}))\longrightarrow 
{\rm   Inf}_\bullet  ( {\bf  HI}) \longrightarrow 
{\rm   Sup}_\bullet  (  {\bf  HI}) \longrightarrow 
R({\rm   Obj}_\bullet (\Delta   {\bf  HI}) )\longrightarrow 
R({\rm  Obj}_\bullet( {\bf  FI})).  
\end{eqnarray*}

 \begin{proposition}\label{pr-26.3.13-9}
 For  any  subcategory  $\overrightarrow{\bf  HI}$  of  $\overrightarrow{\bf  FI}$,  
 let  ${\bf  HI}=\pi(\overrightarrow{\bf  HI})$.
 Then  we  have  a  commutative  diagram  of  chain  complexes
 \begin{eqnarray*}
 \xymatrix{
 R({\rm   Obj}_\bullet (\delta  \overrightarrow{\bf  HI}))\ar[r]\ar[d]
 & 
{\rm   Inf}_\bullet  ( \overrightarrow{\bf  HI}) \ar[r]\ar[d]
&
{\rm   Sup}_\bullet  ( \overrightarrow{\bf  HI}) \ar[r]\ar[d]
& 
R({\rm   Obj}_\bullet (\Delta  \overrightarrow{\bf  HI}) )\ar[r]\ar[d]
& 
R({\rm  Obj}_\bullet(\overrightarrow{\bf  FI}))\ar[d]\\
R({\rm   Obj}_\bullet (\delta   {\bf  HI}))\ar[r] 
 & 
{\rm   Inf}_\bullet  (  {\bf  HI}) \ar[r] 
&
{\rm   Sup}_\bullet  (  {\bf  HI}) \ar[r] 
& 
R({\rm   Obj}_\bullet (\Delta   {\bf  HI}) )\ar[r] 
& 
R({\rm  Obj}_\bullet( {\bf  FI}))
 }
 \end{eqnarray*}
 where  the  horizontal  maps  are  canonical  inclusions  
 and  the  vertical  maps  are  canonical  projections  induced  from  
 (\ref{eq-26.3.16.p9}).  
 \end{proposition}
 
 \begin{proof}
 By   Lemma~\ref{le-3.16-1}, 
  the  last  vertical  map  in  the  diagram   is  given  by  
 the  canonical  projection  $\pi_\#$  in   (\ref{eq-26.3.16.p9}).

{\it  Claim~1}.     $\Delta{\bf  HI}=\pi(\Delta\overrightarrow{\bf  HI})$.

 Let  $\vec\sigma\in  {\rm  Obj}(\Delta\overrightarrow{\bf  HI})$.  
 Then  there  exists  $\vec\tau\in  {\rm  Obj}( \overrightarrow{\bf  HI})$  
 such that  $\vec\sigma$  is  a  subsequence  of  $\vec\tau$.  
Since   $\pi(\vec\sigma)$    and   $\pi(\vec\tau)$  are  the  
 underlying  sets  of  $\vec\sigma$  and  $\vec\tau$  respectively,
 $\pi(\vec\sigma)$  is  a  subset  of  $\pi(\vec\tau)$.      
  It  follows  from  
    $\pi(\vec\tau)\in  {\rm  Obj}({\bf  HI})$  that 
      $\pi(\vec\sigma)\in  {\rm  Obj}(\Delta  {\bf  HI})$.  
 Conversely,  let  $\sigma\in  {\rm  Obj}(\Delta{\bf  HI})$.  
 Then  there  exists  $ \tau\in  {\rm  Obj}( {\bf  HI})$  
 such that  $\sigma$  is  a  subset  of  $\tau$.  
 Choose  
     $\vec\tau\in  {\rm  Obj}(\overrightarrow{\bf  HI})$
     such  that  $\pi(\vec\tau)=\tau$. 
 Then  there  exists  a  subsequence  $\vec\sigma$  of  $\vec\tau$,
 which  satisfies   $\vec\sigma\in  {\rm  Obj}(\Delta\overrightarrow{\bf  HI})$,         
 such  that  $\pi(\vec\sigma)=\sigma$.  
Hence    
  $\pi({\rm  Obj}(\Delta\overrightarrow{\bf  HI}))= {\rm  Obj}(\Delta  {\bf  HI})$.    
Moreover,  
  any  injective  map  (\ref{eq-26.3.17.m1})  between  any   two  objects $\vec\sigma$  
 and  $\vec\sigma'$
  in  $\Delta\overrightarrow{\bf  HI}$
   induces  
 an  injective  map  (\ref{eq-26.3.17.m2})
   between  the   two  objects $\pi(\vec\sigma)$  
 and  $\pi(\vec\sigma')$
  in  $\Delta{\bf  HI}$.  
  Therefore,  we  obtain  Claim~1.

{\it  Claim~2}.    $\delta{\bf  HI}=\pi(\delta\overrightarrow{\bf  HI})$.

Let  $\vec\sigma\in  {\rm  Obj}(\delta\overrightarrow{\bf  HI})$.  
 Then   for  any   non-empty  
 subsequence   $\vec\tau$     of  $\vec\sigma$,   
  $\vec\tau\in  {\rm  Obj}( \overrightarrow{\bf  HI})$.   
  Note  that  when  $\vec\tau$  runs  over  all  
   non-empty  
 subsequences        of  $\vec\sigma$,   
 $\pi(\vec\tau)$  runs  over  all    non-empty  
 subsets        of  $\pi(\vec\sigma)$.  
 Thus   for  any    non-empty  subset  $\pi(\vec\tau)$    of  $\pi(\vec\sigma)$,
  $\pi(\vec\tau)\in  {\rm  Obj}( {\bf  HI})$.   
 Hence   $\pi(\vec\sigma)\in  {\rm  Obj}(\delta  {\bf  HI})$.  
  Conversely,  let  $\sigma\in  {\rm  Obj}(\delta{\bf  HI})$.  
 Then   for  any   non-empty  
 subset   $ \tau$     of  $ \sigma$,   
  $ \tau\in  {\rm  Obj}(  {\bf  HI})$. 
  Choose  
     $\vec\sigma\in  {\rm  Obj}(\overrightarrow{\bf  HI})$
     such  that  $\pi(\vec\sigma)=\sigma$  
     and  choose   $\vec\tau\in  {\rm  Obj}(\overrightarrow{\bf  HI})$
     such  that  $\pi(\vec\tau)=\tau$. 
  Note  that  if  we  fix  $\vec\sigma$  and  let 
   $\tau$  run   over  all  
   non-empty  
 subsets        of  $\sigma$,   
 then  
 $ \vec\tau$  runs  over  all    non-empty  
 subsequences        of  $ \vec\sigma$.  
 Thus  
   for  any   non-empty  
 subsequence   $\vec\tau$     of  $\vec\sigma$,   
  $\vec\tau\in  {\rm  Obj}( \overrightarrow{\bf  HI})$.
  Hence  $\vec\sigma\in  {\rm  Obj}(\delta\overrightarrow{\bf  HI})$.   
  We  obtain  Claim~2.

{\it  Claim~3}.  $\pi_\#({\rm   Inf}_\bullet  ( \overrightarrow{\bf  HI}))
={\rm   Inf}_\bullet  ( {\bf  HI})$.

Since  $\pi_\#$  is  a  chain  map,  $\pi_\# \partial_k=\partial_{k}\pi_\#$ 
for  any  $k\geq  1$.  
It  follows  that  $\pi_\# (\partial_k)^{-1}=(\partial_{k})^{-1}\pi_\#$  where 
$(\partial_k)^{-1}$  denotes  the  pre-image.    
This  implies  that   $\pi_\#({\rm   Inf}_k  ( \overrightarrow{\bf  HI}))
={\rm   Inf}_k  ( {\bf  HI})$. 
  We  obtain  Claim~3.

{\it  Claim~4}.  $\pi_\#({\rm   Sup}_\bullet  ( \overrightarrow{\bf  HI}))
={\rm   Sup}_\bullet  ( {\bf  HI})$.

By an argument   similar  to  that  of   Claim~3,  we   obtain  Claim~4.

By  Claim~1,  the  restriction  of  $\pi_\#$  to  
 $R({\rm   Obj}_\bullet (\Delta  \overrightarrow{\bf  HI}) ) $  
   is  a  canonical  projection  in the  fourth  vertical  map  of  the  diagram.  
   By  Claim~2,  the  restriction  of  $\pi_\#$  to  
 $R({\rm   Obj}_\bullet (\delta  \overrightarrow{\bf  HI}) ) $  
   is  a  canonical  projection  in the  first  vertical  map  of  the  diagram.  
   By  Claim~3,  the  restriction  of  $\pi_\#$  to  
 ${\rm   Inf}_\bullet  ( \overrightarrow{\bf  HI})$  
   is  a  canonical  projection  in the  second  vertical  map  of  the  diagram. 
    By  Claim~4,  the  restriction  of  $\pi_\#$  to  
 ${\rm   Sup}_\bullet  ( \overrightarrow{\bf  HI})$  
   is  a  canonical  projection  in the  third  vertical  map  of  the  diagram. 
 \end{proof}

 \subsection{$\Delta$-modules  and  their  graded  submodules}\label{ss-2.2-26.5.22}

 Let   ${\bf   V}$  be  the  category  whose  objects  are   $R$-modules 
 and  whose  morphisms   are   homomorphisms  of  $R$-modules.  
 A  {\it  $\Delta$-module}  on  $X$   is
 a  covariant  functor  $A$  from     $X$  to  ${\bf   V}$
 and  a  {\it   co-$\Delta$-module}  is
 a  contravariant  functor   from     $X$  to  ${\bf   V}$.  
 Equivalently,  a  $\Delta$-module  on   $X$ 
 is  a  collection  of  $R$-modules  $A_\bullet=\{A_n\}_{n\geq  0}$, 
   where  $A_n$  is  identified  with  the  image  $A(X_n)$,  
together  with  homomorphisms  
\begin{eqnarray}\label{eq-26.5.18-a1}
(\partial_n^i)_*:   A_n\longrightarrow  A_{n-1}
\end{eqnarray}  
for  any  $n\geq  1$  and  any   $0\leq  i\leq  n$
such  that 
 \begin{eqnarray}\label{eq-26.5.17.5}
(\partial^i_{n-1})_*(\partial^j_n)_*= (\partial^{j-1}_{n-1})_*(\partial^i_n )_*
\end{eqnarray}
  for  any  $0\leq i<j\leq  n$;      
and  a  co-$\Delta$-module  is  a  collection  of  $R$-modules  
$\bar  A_\bullet=\{\bar  A_n\}_{n\geq  0}$  
together  with  homomorphisms  
\begin{eqnarray}\label{eq-26.5.18.b1}
(\partial_n^i)^*:   \bar  A_{n-1}\longrightarrow  \bar  A_{n} 
\end{eqnarray} 
for  any  $n\geq  1$  and  any   $0\leq  i\leq  n$  such  that   
\begin{eqnarray}\label{eq-26.5.17.01}
(\partial^j_n)^*(\partial^i_{n-1})^*= (\partial^i_n )^*(\partial^{j-1}_{n-1})^*  
\end{eqnarray}
 for  any  $0\leq i<j\leq  n$.  
 For  any  $\Delta$-module  $A_\bullet$,  
 let  $\partial_n=\sum_{i=0}^n(-1)^i (\partial_n^i)_*$.
   For  any  co-$\Delta$-module  $\bar  A_\bullet$,  
 let  $\partial^{n-1}=\sum_{i=0}^n(-1)^i (\partial_n^i)^*$.

 Let  $A_\bullet$  be  a  $\Delta$-module  and  let 
 $\bar  A_\bullet$  be  a  co-$\Delta$-module.  
 A  graded  submodule    of  $A_\bullet$   (resp.    $\bar  A_\bullet$)  is  a  
      collection  of  $R$-modules  $A'_\bullet=\{A'_n\}_{n\geq  0}$  
   (resp.    $\bar  A'_\bullet=\{\bar  A'_n\}_{n\geq  0}$) 
  such that  $A'_n$  is  a  sub-$R$-module  of  $A_n$ 
  (resp.   $\bar  A'_n$  is  a  sub-$R$-module  of  $\bar  A_n$)  for  each  $n\geq  0$.  
  In  particular,  we  say  that        $A'_\bullet$  is  a  
   sub-$\Delta$-module  of  $A_\bullet$ 
   if     for  any  $n\geq  1$  and  any   $0\leq  i\leq  n$,   
   the  restriction  of  (\ref{eq-26.5.18-a1})  induces  a  homomorphism 
  $ (\partial_n^i)_*:   A'_n\longrightarrow  A'_{n-1}
  $;  
  and  say  that 
  $\bar  A'_\bullet$  is  a  
   sub-co-$\Delta$-module  of  $\bar  A_\bullet$   if   
  the  restriction  of  (\ref{eq-26.5.18.b1})  induces  a  homomorphism 
  $ (\partial_n^i)^*:   \bar  A'_{n-1}\longrightarrow  \bar  A'_{n}
  $. 
  For  any  graded  submodule  $A'_\bullet$  of  $A_\bullet$,  
  let 
  $\delta  A'_\bullet$  be the  largest  sub-$\Delta$-module  of 
  $A_\bullet$  contained  in  $A'_\bullet$  and 
  let  
   $\Delta  A'_\bullet$  be the  smallest  sub-$\Delta$-module   of 
  $A_\bullet$
   containing  $A'_\bullet$.     
   For  any  graded  submodule  $\bar  A'_\bullet$  of  $\bar  A_\bullet$,  
   let 
  $\delta  \bar  A'_\bullet$  be the  largest  sub-co-$\Delta$-module  of 
  $\bar  A_\bullet$  contained  in  $\bar  A'_\bullet$
  and  let  
   $\Delta  \bar  A'_\bullet$  be the  smallest  sub-co-$\Delta$-module   of 
  $\bar  A_\bullet$
   containing  $\bar  A'_\bullet$.

    For  any  $\Delta$-modules  $A_\bullet$  and  $B_\bullet$,   
  a  morphism  $f:  A_\bullet\longrightarrow  B_\bullet$  is a  sequence  of  
  homomorphisms  $f_n:  A_n\longrightarrow  B_n$  such  that  
  \begin{eqnarray}\label{eq-26.5.20.1}
  f_{n-1}(\partial_n^i)_*= (\partial_{n}^i)_*  f_n 
  \end{eqnarray}
    for  any  $n\geq  1$ and  any  $0\leq  i\leq  n$.  
    Similarly, 
    for  any  co-$\Delta$-modules  $\bar  A_\bullet$  and  $\bar  B_\bullet$,   
  a  morphism  $f:  \bar  A_\bullet\longrightarrow  \bar  B_\bullet$  is a  sequence  of  
  homomorphisms  $f_n:  \bar  A_n\longrightarrow  \bar   B_n$  such  that  
  \begin{eqnarray} \label{eq-26.5.20.1d}
  f_{n+1}(\partial_n^i)^*= (\partial_{n}^i)^*  f_n 
  \end{eqnarray}
    for  any  $n\geq  1$ and  any  $0\leq  i\leq  n$.  
   Let   $f:  A_\bullet\longrightarrow  B_\bullet$  be  a  morphism  of  $\Delta$-modules     
    (resp.   $f:  \bar  A_\bullet\longrightarrow  \bar   B_\bullet$ be  a  morphism 
    of  co-$\Delta$-modules).  
   Then  $f$ sends a  graded  submodule  $A'_\bullet$  of  $A_\bullet$ 
    (resp.  $\bar   A'_\bullet$  of  $\bar   A_\bullet$)  to  
    a   graded  submodule  $f(A'_\bullet)$  of  $B_\bullet$
    (resp.   $ f(\bar A'_\bullet)$  of  $\bar  B_\bullet$).
    Moreover,  
    if   $A'_\bullet$   is  a   sub-$\Delta$-module  of  $A_\bullet$
    (resp.   $\bar   A'_\bullet$  is  a   sub-co-$\Delta$-module   of  $\bar   A_\bullet$),  
    then  $f(A'_\bullet)$    is  a   sub-$\Delta$-module  of  $B_\bullet$
    (resp.    $ f(\bar A'_\bullet)$    is  a   sub-co-$\Delta$-module  of  $\bar  B_\bullet$).  
     Consequently,  by  a  similar  argument  with  
    Subsection~\ref{26.5-ss-2.1},  \cite[Sect.~2.2]{jktr2022rel} 
     and  \cite[Sect.~3.2]{stab26},  we  obtain  
    \begin{eqnarray}\label{eq-26.10.20.sub1}
    f(\delta  A'_\bullet)\subseteq  \delta  f( A'_\bullet) ~~~~~~
     ({\rm~resp.~~~}   f(\delta  \bar   A'_\bullet)\subseteq  \delta  f( \bar   A'_\bullet)~)  
     \end{eqnarray} 
     and     
     \begin{eqnarray}\label{eq-26.10.20.sub2}
     f(\Delta  A'_\bullet)\subseteq  \Delta  f( A'_\bullet) ~~~~~~
     ({\rm~resp.~~~}   f(\Delta  \bar   A'_\bullet)\subseteq  \Delta  f( \bar   A'_\bullet)~).
     \end{eqnarray}  
  Note  that  the  inclusions   in  (\ref{eq-26.10.20.sub1})  and   (\ref{eq-26.10.20.sub2}) 
    may   not  be   equalities  in  general.   Moreover,  
    the    equalities   in     (\ref{eq-26.10.20.sub2})  
    are  satisfied   if   $A'$   and  $B'$   are  the  free  modules
      spanned  by  the  hyperedges  of     hypergraphs 
     and  $f$  is    induced  by  a  morphism  of  hypergraphs.  
    Furthermore,  
      the    inclusions   in     (\ref{eq-26.10.20.sub1})  
    can  still  be  strict    even  if   $A'$   and  $B'$   are  the  free  modules
      spanned  by  the  hyperedges  of     hypergraphs 
     and  $f$  is    induced  by  a  morphism  of  hypergraphs.

  \begin{proposition}
  \label{le-26-5-18-1}
  \begin{enumerate}[(1)]
  \item
  For  any  graded  submodule $A'$  of  a  $\Delta$-module  $A$,  
  there  are  inclusions  of  chain  complexes 
  \begin{eqnarray}
  \label{eq-26.5.18.i1}
  (\delta  A'_\bullet, \partial_\bullet)\longrightarrow 
  {\rm  Inf}(A'_\bullet)\longrightarrow 
  {\rm  Sup}(A'_\bullet)\longrightarrow 
  (\Delta  A'_\bullet, \partial_\bullet)\longrightarrow 
  (A_\bullet, \partial_\bullet)
  \end{eqnarray}
  such  that  the  second  inclusion  is  a  quasi-isomorphism. 
  Moreover, 
(\ref{eq-26.5.18.i1})     is  functorial  with respect  to  
morphisms  of   $\Delta$-modules;  
  \item
  For  any  graded  submodule $\bar  A'$  of  a  co-$\Delta$-module  $\bar  A$, 
    there  are  inclusions  of  cochain  complexes 
     \begin{eqnarray}
  \label{eq-26.5.18.i2}
  (\delta  \bar  A'_\bullet, \partial^\bullet)\longrightarrow 
  {\rm  Inf}(\bar  A'_\bullet)\longrightarrow 
  {\rm  Sup}(\bar  A'_\bullet)\longrightarrow 
  (\Delta  \bar  A'_\bullet, \partial^\bullet)\longrightarrow 
  (\bar  A_\bullet, \partial^\bullet)
  \end{eqnarray}
  such  that  the  second  inclusion  is  a  quasi-isomorphism.   
  Moreover,    (\ref{eq-26.5.18.i2})    is  functorial  with respect  to  
morphisms  of   co-$\Delta$-modules.  
  \end{enumerate}
  \end{proposition}
  
  \begin{proof}
  Analogous  to   \cite[Lemma~3.1]{jgp},  
 it  can  be  verified  that  $\partial_{n-1}\partial_n=0$  for  any  $n\geq  1$.  
 Hence     $(A_\bullet,\partial_\bullet)$  is  a  chain  complex.
Since     $\Delta  A'_\bullet$  is  a 
sub-$\Delta$-module   of  $A$  and  $\delta  A'_\bullet$
   is  a 
sub-$\Delta$-module   of    $\Delta  A'_\bullet$,    
   we   obtain  that  $(\Delta  A'_\bullet,\partial_\bullet)$
   is  a  sub-chain  complex  of  $(A_\bullet,\partial_\bullet)$ 
and  
    $(\delta  A'_\bullet,\partial_\bullet)$
  is  a  sub-chain  complex   of  $(\Delta  A'_\bullet,\partial_\bullet)$.  
  With  the  help  of   \cite[Sect.~2]{h1}, 
  ${\rm  Inf}(A'_\bullet )$  is  the  largest  sub-chain  complex  
  of  $(A_\bullet,\partial_\bullet)$  contained  in  $A'_\bullet$  
  and  ${\rm  Sup}(A'_\bullet )$  is  the  smallest  sub-chain  complex  
  of  $(A_\bullet,\partial_\bullet)$  containing  $A'_\bullet$ 
  such  that  the  canonical  inclusion  of  ${\rm  Inf}(A'_\bullet )$
  in  ${\rm  Sup}(A'_\bullet )$  induces  
  an  isomorphism  of  homology  groups. 
  We  obtain  (\ref{eq-26.5.18.i1}).  
  By  a  similar  argument  with  Proposition~\ref{le-2.2.inc}, 
  the  functoriality  of  the  first  row  of  (\ref{eq-26.5.18.i1})  follows  from  
  (\ref{eq-26.5.20.1}).  
    The  proof  of  (\ref{eq-26.5.18.i2})  is  an  analogue  of  (\ref{eq-26.5.18.i1}).  
     The  functoriality  of  the  first  row  of  (\ref{eq-26.5.18.i2})  follows  from  
  (\ref{eq-26.5.20.1d}). 
  \end{proof}

   Similar  to  the  $\Delta$-modules,  
   we  can  define  $\Delta$-groupoids  as  a  generalization   of  \cite[Section~4]{jams-06}.  
 Let  ${\bf   Gpd}$  be  the  category  whose  objects are  groupoids  
 and  whose  morphisms  are      morphisms  of  groupoids.  
 A  {\it  $\Delta$-groupoid}   on   $X$   is
 a  covariant  functor   from     $X$  to  ${\bf   Gpd}$.  
Equivalently,  a  $\Delta$-groupoid 
 is  a  collection  of  groupoids  $G_\bullet=\{G_n\}_{n\geq  0}$  
together  with   morphisms
 $(\partial_n^i)_*:   G_n\longrightarrow  G_{n-1}$  
for  any  $n\geq  1$  and  any   $0\leq  i\leq  n$
such  that 
(\ref{eq-26.5.17.5})  is  satisfied 
  for  any  $0\leq i<j\leq  n$.    
    For  any  $\Delta$-groupoids  $G_\bullet$  and  $G'_\bullet$,   
  a  morphism  $f:  G_\bullet\longrightarrow  G'_\bullet$  is a  sequence  of  
  homomorphisms  $f_n:  G_n\longrightarrow  G'_n$  such  that  
 (\ref{eq-26.5.20.1})  is  satisfied  
    for  any  $n\geq  1$ and  any  $0\leq  i\leq  n$.

 \subsection{$\Delta$-manifolds  and  their  graded  submanifolds}
 \label{ss2.3-26.5.22}
   
 Let   ${\bf   M}$  be  the  category  whose  objects  are   differentiable  manifolds
 and  whose  morphisms   are   differentiable  maps.     
 A  {\it  $\Delta$-manifold}  on   $X$   is
 a  covariant  functor   from    $X$  to  ${\bf   M}$.  
 Equivalently,  a  $\Delta$-manifold     
 is  a  collection  of  differentiable  manifolds  $M_\bullet=\{M_n\}_{n\geq  0}$  
together  with  differentiable  maps  $\partial_n^i:  M_n\longrightarrow  M_{n-1}$  
for  any  $n\geq  1$  and  any   $0\leq  i\leq  n$
such  that  (\ref{eq-26.5.17.55})  is  satisfied.  
Let  $M_\bullet$  be  a  $\Delta$-manifold.
A   graded  submanifold  of  $M_\bullet$ 
is  a  collection  of  differentiable  manifolds  $M'_\bullet=\{M'_n\}_{n\geq  0}$
such  that  $M'_n$  is  a  submanifold  of  $M_n$  for  each  $n\geq  0$.  
In  particular,  if  $M'_\bullet$  is  a  $\Delta$-manifold  as  well,  
then   $M'_\bullet$  is  a  sub-$\Delta$-manifold  of  $M_\bullet$.  
For  any  graded  submanifold  $M'_\bullet$  of  $M_\bullet$,  
let  $\delta   M'_\bullet$  be  the  largest  sub-$\Delta$-manifold   of  $M_\bullet$
contained  in   $M'_\bullet$ 
and  let  $\Delta   M'_\bullet$  be  the  smallest  sub-$\Delta$-manifold   of  $M_\bullet$
containing   $M'_\bullet$.

Given  two   $\Delta$-manifolds  $M_\bullet$  and  $N_\bullet$,  
a  morphism  $f:  M_\bullet\longrightarrow  N_\bullet$  
is  a  sequence  of  smooth  maps  $f_n:  M_n\longrightarrow  N_n$
such that  (\ref{eq-26.5.20.9})  is  satisfied
for  any  $n\geq  1$ and  any  $0\leq  i\leq  n$.  
Let  $f:  M_\bullet\longrightarrow  N_\bullet$  be  a  morphism  of  $\Delta$-manifolds.  
 Suppose  in  addition  that  
  $f$ sends a  graded  submanifold  $M'_\bullet$  of  $M_\bullet$ 
to  the  graded  submanifold  $f(M'_\bullet)$   of  $N_\bullet$.  
For  instance,  $f$  can  be 
 embeddings,  submersions  and  local   diffeomorphisms.  
It  follows  that   if  $M'_\bullet$  is  a  sub-$\Delta$-manifold   of  $M_\bullet$,
then  $f(M'_\bullet)$  is  a  sub-$\Delta$-manifold  of  $N_\bullet$. 
Consequently,  
by  an  argument  similar to   
    Subsection~\ref{26.5-ss-2.1},
       (\ref{eq-26.10.20.sub1})  and  (\ref{eq-26.10.20.sub2}),  we  obtain  
    $f(\delta  M'_\bullet)\subseteq  \delta  f( M'_\bullet)$    
     and     $f(\Delta  M'_\bullet)\subseteq  \Delta  f( M'_\bullet)$.

     \begin{lemma}\label{le-5.22.1}
     Let  $M_\bullet$  be  a  $\Delta$-manifold
and  $M'_\bullet$  be  a  graded  submanifold  of  $M_\bullet$.  
Then  we  have  a  sequence  of   inclusions  of  $\Delta$-manifolds 
\begin{eqnarray}\label{eq-5.22.1}
\delta  M'_\bullet\longrightarrow  \Delta  M'_\bullet
\longrightarrow    M _\bullet. 
\end{eqnarray}
Moreover,  let   $f:  M_\bullet \longrightarrow  N_\bullet$ 
be  a  morphism  of  $\Delta$-manifolds
sending  any  graded  submanifolds  of  $M$  to  graded  submanifolds  of  $N$.  
Then  we  have  a  commutative  diagram 
\begin{eqnarray}\label{eq-5.22.2}
\xymatrix{
\delta  M'_\bullet\ar[r] \ar[d]
&  \Delta  M'_\bullet \ar[r] \ar[d]
&    M _\bullet\ar[d]^-{f}\\
\delta  f(M'_\bullet)\ar[r]  
&  \Delta  f(M'_\bullet) \ar[r]  
&    N _\bullet. 
}
\end{eqnarray}
     \end{lemma}
     
     \begin{proof}
     The  sequence  
     (\ref{eq-5.22.1})  follows  from  the  definitions  of  $\delta  M'_\bullet$ 
     and  $\Delta  M'_\bullet$.  
     It  is  straightforward  to  verify  that  the  diagram  (\ref{eq-5.22.2})  commutes,
     which  implies    
     the  functoriality  of  (\ref{eq-5.22.1}).  
     \end{proof}

\begin{proposition}\label{th-26.5.18-1}
Let  $M_\bullet$  be  a  $\Delta$-manifold
and  $M'_\bullet$  be  a  graded  submanifold  of  $M_\bullet$.  
Suppose  there  is  a  continuous  total  order  on  the  face  maps  
of   $M_\bullet$.  
Then  we  have    sequences  of   inclusions   of double    complexes
\begin{eqnarray}\label{eq-5.21.a1}
\xymatrix{
C_\bullet(\delta  M'_\bullet;R) \ar[r]
& {\rm  Inf}_\bullet (M'_\bullet) \ar[r]
&  {\rm  Sup}_\bullet (M'_\bullet)  \ar[r]
& C_\bullet(\Delta  M'_\bullet;R) \ar[r]
&C_\bullet(M_\bullet;R),
}\\
\label{eq-5.21.a11}
\xymatrix{
C^\bullet(\delta  M'_\bullet;R) \ar[r]
& {\rm  Inf}^\bullet (M'_\bullet) \ar[r]
&  {\rm  Sup}^\bullet (M'_\bullet)  \ar[r]
& C^\bullet(\Delta  M'_\bullet;R) \ar[r]
&C^\bullet(M_\bullet;R),
}
\end{eqnarray}
where    
\begin{enumerate}[(1)]
\item
the  infimum  (co)chain  complexes  and  the  supremum  (co)chain  complexes 
are  taken  with respect to  the  (co)boundary  map  of  the  second  index; 
\item
the  second  arrows  in  (\ref{eq-5.21.a1})  and (\ref{eq-5.21.a11})  are   
  quasi-isomorphisms
with respect  to the  second  index; 
\item
 (\ref{eq-5.21.a1})  and  (\ref{eq-5.21.a11})     are  
  functorial with  respect  to  the morphisms  of  $\Delta$-manifolds 
  sending  graded  submanifolds  to  graded  submanifolds. 
\end{enumerate}
\end{proposition}

\begin{proof}
For  any  $n\geq  0$, 
 the      (singular)   chain  complex   $C_\bullet(M_n;R)$ 
has  a  sequence  of   boundary  maps  
\begin{eqnarray*}
\delta_m:   C_m(M_n;R)\longrightarrow  C_{m-1}(M_n; R),~~~~~~  m\geq  0. 
\end{eqnarray*}
For  any  $m\geq  0$, 
 the  $m$-th   (singular)   chain  group   $C_m(M_\bullet;R)$  is  a  $\Delta$-module
 with  the  induced  maps  
 \begin{eqnarray*}
 (\partial_{n}^i)_{*}:  C_m(M_n;R)\longrightarrow  C_m(M_{n-1};R), ~~~~~~ n\geq  0. 
 \end{eqnarray*}
 Sicne  there  is  a  continuous  total  order  on  the  face  maps  
of   $M_\bullet$,  similar  to  the  argument  of   (\ref{eq-26.5.22-3}),  
 we  have  a  sequence  of  boundary  maps  
 \begin{eqnarray*}
 \partial_n=\sum_{i=0}^n(-1)^i  (\partial_{n}^i)_{*}:
   C_m(M_n;R)\longrightarrow C_m(M_{n-1};R),~~~~~~  n\geq  0.   
 \end{eqnarray*}
 Since  $ \delta_{m-1}\delta_m =\partial_{n-1}\partial_n =0$  and  
 $\delta_m \partial_n= \partial_n\delta_m$,  
 we  obtain  a  double  chain   complex 
  $(C_\bullet(M_\bullet;R), \delta_\bullet, \partial_\bullet)$.  
 With  the  help  of  (\ref{eq-5.22.1}),  
 we  have  inclusions  of  double  chain  complexes 
 \begin{eqnarray*}
 (C_\bullet(\delta  M'_\bullet;R), \delta_\bullet, \partial_\bullet)
  \longrightarrow 
   (C_\bullet(\Delta  M'_\bullet;R), \delta_\bullet, \partial_\bullet)
 \longrightarrow 
 (C_\bullet(M_\bullet;R), \delta_\bullet, \partial_\bullet). 
 \end{eqnarray*}
 Moreover, 
   \begin{eqnarray*}
   (C_\bullet(M'_\bullet; R), \delta_\bullet)\longrightarrow 
   (C_\bullet(\Delta  M'_\bullet;R),   \delta_\bullet)
   \end{eqnarray*}
     is  an  
   inclusion  of    chain  complexes  with  respect  to  the  first     index.  
   Hence  taking  the  infimum   chain  complex   and  the  supremum   chain  complex 
  with respect to   $\partial_\bullet$  in  the  second  index, 
  we  obtain  (\ref{eq-5.21.a1}).

  Let   $f:  M_\bullet \longrightarrow  N_\bullet$ 
be  a  morphism  of  $\Delta$-manifolds  
sending  any  graded  submanifolds  of  $M$  to  graded  submanifolds  of  $N$.  
  Then  
with  the  help  of  Proposition~\ref{le-2.2.inc},
Proposition~\ref{le-26-5-18-1}~(1)   and  (\ref{eq-5.22.2}),  
we  have   a  commutative  diagram  of   double  chain  complexes 
\begin{eqnarray*} 
\xymatrix{
C_\bullet(\delta  M'_\bullet;R) \ar[r]\ar[d]
& {\rm  Inf}_\bullet (M'_\bullet) \ar[r]\ar[d]
&  {\rm  Sup}_\bullet (M'_\bullet)  \ar[r]\ar[d]
& C_\bullet(\Delta  M'_\bullet;R) \ar[r]\ar[d]
&C_\bullet(M_\bullet;R)\ar[d]^-{f_\#}\\
C_\bullet(\delta   f(M'_\bullet);R) \ar[r] 
& {\rm  Inf}_\bullet (f(M'_\bullet)) \ar[r] 
&  {\rm  Sup}_\bullet (f(M'_\bullet))  \ar[r] 
& C_\bullet(\Delta  f(M'_\bullet);R) \ar[r] 
&C_\bullet(N_\bullet;R).
}
\end{eqnarray*}
We  obtain  the  functoriality  of   (\ref{eq-5.21.a1}).

    The  proof  of  (\ref{eq-5.21.a11})  is  an  analogue  of  (\ref{eq-5.21.a1})
    by  substituting  the  $\Delta$-modules 
    with  the  co-$\Delta$-modules    and     substituting the  (singular)  chain  complexes  with  the  (singular)  cochain  complexes.  
   \end{proof}

\begin{theorem}\label{pr-5.22.9}
Let  $M_\bullet$  be  a  $\Delta$-manifold
and  $M'_\bullet$  be  a  graded  submanifold  of  $M_\bullet$.  
Then  we  have   
a  sequence  of   morphisms  of  $\Delta$-groupoids 
\begin{eqnarray}\label{eq-5.21.a2}
\xymatrix{
\Pi_1  (\delta  M'_\bullet) \ar[r]
& \Pi_1  (\Delta  M'_\bullet)   \ar[r]
&\Pi_1  ( M _\bullet),   
}
\end{eqnarray}
   which  
 is   
  functorial with  respect  to   morphisms  of  $\Delta$-manifolds.  
 Moreover,  suppose  there  is  a  continuous  total  order  on  the  face  maps  
of   $M_\bullet$. 
  Then  we  have 
 sequences  of  homomorphisms   of double   (co)homology  groups
\begin{eqnarray}\label{eq-5.21.h1}
\xymatrix{
H_\bullet(\delta  M'_\bullet;R) \ar[r]
& H_\bullet (M'_\bullet) \ar[r]
& H_\bullet(\Delta  M'_\bullet;R) \ar[r]
&H_\bullet(M_\bullet;R),
}\\
\label{eq-5.21.h11}
\xymatrix{
H^\bullet(\delta  M'_\bullet;R) \ar[r]
& H^\bullet (M'_\bullet) \ar[r]
& H^\bullet(\Delta  M'_\bullet;R) \ar[r]
&H^\bullet(M_\bullet;R) 
}
\end{eqnarray}
   which  
 are  
  functorial with  respect  to  the  morphisms  of  $\Delta$-manifolds
   sending  graded  submanifolds  to  graded  submanifolds.  
\end{theorem}

\begin{proof}
 Applying the  fundamental  groupoid   functor  to  (\ref{eq-5.22.1}), 
 we  obtain  (\ref{eq-5.21.a2}).  
 Similarly, 
 applying  it  to  (\ref{eq-5.22.2})  yields  the  functoriality  of  (\ref{eq-5.21.a2}).

In  addition,   suppose  there  is  a  continuous  total  order  on  the  face  maps  
of   $M_\bullet$. 
Applying  the  double  homology  functor  to  (\ref{eq-5.21.a1})
and  (\ref{eq-5.21.a11})  respectively,  
we  obtain  (\ref{eq-5.21.h1})  and    (\ref{eq-5.21.h11}).   
By  the  functoriality  of (\ref{eq-5.21.a1})  and  (\ref{eq-5.21.a11})
respectively,  
we  obtain  the  functoriality  of
 (\ref{eq-5.21.h1})  and    (\ref{eq-5.21.h11}).   
\end{proof}

  \section{Configuration  spaces  and  hypergraphs  on  manifolds}\label{s.26.7.28.3}
  
  In  this  section,  
  we  construct  the  independence  complexes  on  manifolds  
  as  unions  of  configuration  spaces and  construct  the  hypergraphs  on
  manifolds  as  graded  submanifolds  of  the  independence  complexes.  
  In  Subsection~\ref{ss.26.6.3-3.1},  
  we  construct  the  directed  independence  complexes  on  manifolds   
    as  
   unions  of  the  ordered  configuration  spaces   
  and  construct   the  independence  complexes  on  manifolds  
  as     unions  of  the  unordered  configuration  spaces.   
  We  prove  that  the  (directed)  independence  complexes  on  manifolds
  are  $\Delta$-manifolds.   
  We  give  the 
   double  homology  for  the   (directed)  independence  complexes  
  in  Theorem~\ref{th-homol-conf-1}.  
    In  Subsection~\ref{26.6.3-ss.3.2},  
    we  construct  hyperdigraphs  on  manifolds  
    as  graded  submanifolds  of  the   directed  independence  complexes
    and  construct  hypergraphs  on  manifolds  as  graded  submanifolds  of  
    the  independence  complexes.  
    We  prove a  sequence  of  double  homology 
    for  the  hyper(di)graphs  in  Theorem~\ref{th-6.13.1} (Main Result I).  
    In  Subsection~\ref{26.5.25-ss3.5},  
    we  study    
    hyper(di)graphs       
      on  disjoint  unions  of  manifolds  by  free  products  
      and       simplicial   joins  respectively. 
  
  \subsection{Configuration  spaces     on  manifolds}\label{ss.26.6.3-3.1}
  
Let  $M$  be a    differentiable  manifold. 
For  any positive  integer  $k$,    
a  {\it  directed  $k$-hyperedge}  $\vec{\sigma}(M)$  or  $\vec{\sigma}_k(M)$  on  $M$  is  an  ordered  $k$-tuple 
$ (x_1,x_2,\ldots,x_k)$  where $x_1,x_2,\ldots,x_k\in  M$  are  distinct. 
The  {\it  $k$-th  ordered  configuration space}
${\rm  Conf}_k(M)$  is  the  open  submanifold  
of  $M^k$  consisting  of  the  points  $(x_1,\ldots,x_k)$  such  that  
$x_i\neq  x_j$  for  any  $i\neq  j$.
In other  words,
 ${\rm  Conf}_k(M)$ 
  consists      
of    all  the directed  $k$-hyperedges  on  $M$.
Let  ${\rm  Conf}_0(M)=\{\vec\emptyset\}$
where  $\vec\emptyset$  denotes  the  empty  sequence.  
There  is  a  graded  manifold  
\begin{eqnarray*}
{\rm  Conf}_\bullet (M) = \bigcup_{k\geq  0}  {\rm  Conf}_k(M),     
\end{eqnarray*}
which  will  be  called  the  {\it     directed  independence  complex}
 of  $M$  and  denoted  
as   $\overrightarrow{\rm  Ind}(M)$  \footnote[1]{
Here   the  directed  independence  complex  of  $M$  is  a  directed  
simplicial  complex  on  $M$  and   is  allowed  to  contain  the  empty  
sequence  as  a  directed  simplex,  see  Subsection~\ref{26.6.3-ss.3.2}.   },  
 equipped  with  smooth   face  maps  
\begin{eqnarray}\label{eq-2.1a}
\partial^i_k:   {\rm  Conf}_k(M)\longrightarrow  {\rm  Conf}_{k-1}(M)  
\end{eqnarray}
for   $1\leq  i\leq  k$  given  by  
\begin{eqnarray*}
\partial^i_k   (x_1,\ldots,x_k)= (x_1,\ldots, \widehat{x_i}, \ldots,x_k).   
\end{eqnarray*}
The  $k$-th  symmetric group  $\Sigma_k$  acts   on  ${\rm  Conf}_k(M)$ smoothly  
by  
permuting  the  coordinates 
\begin{eqnarray}\label{eq-26.4.23.1}
s(x_1,\ldots,x_k)=(x_{s(1)},\ldots,x_{s(k)}), ~~~~~~ s\in  \Sigma_k.  
\end{eqnarray}
There  are  canonical    
   group  monomorphisms  
   \begin{eqnarray}\label{eq-26-3-10-1}
   \varphi_k^i:  \Sigma_{k-1}\longrightarrow \Sigma_k
   \end{eqnarray} 
for  any  $1\leq  i\leq  k$  sending  any  $s\in \Sigma_{k-1}$  to the  permutation  
\begin{eqnarray*}
\varphi_k^i(s)(x_1,\ldots,x_k)= 
\left(
\begin{matrix} 
x_1   &\ldots   &x_{i-1}  &x_i  &x_{i+1}    &\ldots  &x_k\\ 
 x_{s(1)},&\ldots, &x_{s(i-1)}  &x_i   &x_{s(i)}  &\ldots   & x_{s(k-1)}  
\end{matrix}
\right) 
\end{eqnarray*}
such  that  the  diagram   commutes 
\begin{eqnarray}\label{eq-4.23.2}
\xymatrix{
{\rm  Conf}_k(M)\ar[r]^-{\partial^i_k} \ar[d]_-{\varphi_k^i(\Sigma_{k-1})}   
&{\rm  Conf}_{k-1}(M)\ar[d]^-{\Sigma_{k-1}}\\
{\rm  Conf}_k(M)\ar[r]^-{\partial^i_k}  &{\rm  Conf}_{k-1}(M).  
}
\end{eqnarray}
The  {\it  $k$-th  unordered  configuration space}
${\rm  Conf}_k(M)/\Sigma_k$  is  the  orbit  manifold  of  the  
$\Sigma_k$-action   on  ${\rm  Conf}_k(M)$, 
whose  points  are   $k$-subsets  $\{x_1,\ldots,x_k\}$  of  $M$  such  that  
$x_i\neq  x_j$  for  any  $i\neq  j$.
Let  ${\rm  Conf}_0(M)/\Sigma_0=\{\emptyset\}$  where  $\emptyset$  denotes  
the empty   set.  
There  is  a  graded  manifold  
\begin{eqnarray*}
{\rm  Conf}_\bullet (M) /\Sigma_\bullet= \bigcup_{k\geq  0}  {\rm  Conf}_k(M)/\Sigma_k,     
\end{eqnarray*}
which  will  be  called  the  {\it   independence  complex}  of  $M$
and  denoted  
as   ${\rm  Ind}(M)$  \footnote[2]{
Here   the     independence  complex  of  $M$   is  a    
simplicial  complex  on  $M$  and   is  allowed  to  contain  the  empty set   
  as  a     simplex,  see  Subsection~\ref{26.6.3-ss.3.2}.   }.
The  canonical  projections  
 \begin{eqnarray}\label{eq-2601-a}
 \pi_k:  {\rm  Conf}_k(M)\longrightarrow  {\rm  Conf}_k(M)/\Sigma_k
 \end{eqnarray}  
 for  all  $k\geq  1$  induce   a   canonical  projection 
 \begin{eqnarray}\label{eq-2601-aa}
 \pi:  \overrightarrow{\rm  Ind}(M)\longrightarrow   {\rm  Ind}(M) 
 \end{eqnarray} 
 sending  any    directed   hyperedge   $\vec{\sigma}(M)= (x_1,x_2,\ldots,x_k)$ 
to  its  underlying   hyperedge  $\sigma(M)=\{x_1,x_2,\ldots,x_k\}$.

Suppose  $M$  is  a  submanifold  of  $\mathbb{R}$.     
Then  the  total  order  on  $\mathbb{R}$  induces  a  total  order  $\prec$  on  $M$. 
 Consider  the 
 submanifold
 \begin{eqnarray*}
{\rm  Conf}_k (M,\prec) = \{(x_1,\ldots,x_k)\in   {\rm  Conf}_k(M)\mid
x_1\prec\cdots\prec  x_k\}      
\end{eqnarray*} 
of  $ {\rm  Conf}_k(M)$  
and  the
  graded  submanifold 
\begin{eqnarray*}
{\rm  Conf}_\bullet (M,\prec) = \bigcup_{k\geq  0}  {\rm  Conf}_k(M,\prec),       
\end{eqnarray*} 
denoted  as  $\overrightarrow {\rm  Ind}(M,\prec)$,  
of  $ \overrightarrow{\rm  Ind}(M)$.   
Then  the  map 
\begin{eqnarray}\label{eq-5.2.5}
\theta_k:   {\rm  Conf}_k(M)/\Sigma_k\longrightarrow  {\rm  Conf}_k (M,\prec)
\end{eqnarray}
sending  $\{x_1,\ldots, x_k\}$  to  
$(x_1,\ldots,x_k)$  such  that   $x_1\prec\cdots\prec  x_k$  
is   a  homeomorphism 
 and  is  a  cross-section  of    (\ref{eq-2601-a}).  
 Consequently,  
 the     map 
 \begin{eqnarray}\label{eq-5.2.1}
\theta:    {\rm  Ind}(M)  
\longrightarrow  
 \overrightarrow {\rm  Ind}(M,\prec)
\end{eqnarray}
  is  an  isomorphism  and  is  a  cross-section  of  (\ref{eq-2601-aa}).
  Moreover,    $ \overrightarrow {\rm  Ind}(M,\prec)$  is  invariant  under 
    the  face  maps  (\ref{eq-2.1a})  thus  the  restrictions  of   (\ref{eq-2.1a}) to  
   $ \overrightarrow {\rm  Ind}(M,\prec)$  give  face  maps  
\begin{eqnarray}\label{eq-5.1ab}
\partial^i_k:   {\rm  Conf}_k(M,\prec)\longrightarrow  {\rm  Conf}_{k-1}(M,\prec)  
\end{eqnarray}
for   $1\leq  i\leq  k$.

\begin{lemma}
\label{le-24.5.24.conf1}
\begin{enumerate}[(1)]
\item
Let  ${\bf  Sub}(M)$  be  the  category  whose  objects are  submanifolds  of  $M$
 and  whose  morphisms are  embeddings.   Then  
$ \overrightarrow{\rm  Ind}(M)$  is  a  covariant  functor  from   $\overrightarrow{\bf  FI}$
  to    ${\bf  Sub}(M)$ 
     and 
      ${\rm  Ind}(M)$  is  a    covariant  functor  from  ${\bf  FI}$  to  ${\bf  Sub}(M)$ 
     such  that     the  diagram  commutes 
     \begin{eqnarray*}
     \xymatrix{
     \overrightarrow{\bf  FI}\ar[rr]^-{\pi}  \ar[rd]_-{ \overrightarrow{\rm  Ind}(M)}  
     &&{\bf  FI}\ar[ld]^-{{\rm  Ind}(M)}  \\
     &  {\bf  Sub}(M); &
     }
     \end{eqnarray*}
     \item
     both  $ \overrightarrow{\rm  Ind}(M)$   and  ${\rm  Ind}(M)$  are     $\Delta$-manifolds;  
     \item
     there  is  a  continuous  total  order  on  the  face  maps  
of   $ \overrightarrow{\rm  Ind}(M)$.
  Moreover,  
     if  $M$  is  a  submanifold  of  $\mathbb{R}$,  
      then    
      there  is  a  continuous  total  order  on  the  face  maps  
of   $ {\rm  Ind}(M)$.     
     \end{enumerate}  
\end{lemma}

\begin{proof}
The  proof  of   (1)   follows from  a  straightforward  verification.  
The proofs  of  (2)  and  (3)  follow  from  a   refinement  of 
  \cite[Lemma~9.1  and  Lemma~9.4]{jgp}.  
\end{proof}

\begin{theorem}
\label{th-homol-conf-1}
For  any  differentiable  manifold  $M$,  
we  have  double  homology 
$H (\overrightarrow {\rm  Ind} (M);R)$.  
Moreover,   if   $M$  is  a  submanifold  of  $\mathbb{R}$,  
then 
we  have  double  homology
$H ({\rm  Ind}(M);R)$
such  that 
 the     canonical  projection  $\pi$  induces  a  
homomorphism  
\begin{eqnarray}\label{eq-26.6.25-1}
\pi_*:  H(\overrightarrow {\rm  Ind} (M);R)\longrightarrow
  H({\rm  Ind}(M);R).
  \end{eqnarray}  
\end{theorem}

\begin{proof}
By  Theorem~\ref{pr-5.22.9}  and   Lemma~\ref{le-24.5.24.conf1}~(2)- (3),  
we  have  the   double  homology 
$H (\overrightarrow {\rm  Ind} (M);R)$.  
Suppose  in  addition  that 
 $M$  is  a  submanifold  of  $\mathbb{R}$.   
  Then 
 by  Theorem~\ref{pr-5.22.9}  and   Lemma~\ref{le-24.5.24.conf1}~(2)-(3),  
we  have  the   double  homology 
$H (  {\rm  Ind} (M);R)$.  
 By  an  analogous  argument  of  
 \cite[Lemma~3.8 and  Theorem~3.9]{hdg} 
 and  
 \cite[Proposition~4.3]{reg},  
 the  diagram  commutes 
 \begin{eqnarray}\label{diag-26-6-25}
 \xymatrix{
 R({\rm  Conf}_k(M))\ar[r]^-{ {\partial}_k}  \ar[d]_-{\pi_\#}&
    R({\rm  Conf}_{k-1}(M))\ar[d]^-{\pi_\#}\\
   R({\rm  Conf}_k(M)/\Sigma_k)\ar[r]^-{   {\partial}_k}  &  
    R({\rm  Conf}_{k-1}(M)/\Sigma_{k-1}),      
 }
 \end{eqnarray}
where  the  first  horizontal  map  is  the  
$R$-linear  extension  of  the  alternating  sum  of  
(\ref{eq-2.1a})  and  the  second  horizontal  map  is  the  
$R$-linear  extension  of  the  alternating  sum  of  
(\ref{eq-5.1ab}). 
Since  the  projection  $\pi$  and  the  maps   (\ref{eq-2.1a})  and  (\ref{eq-5.1ab})
are  smooth  maps  of  differentiable  manifolds,  
(\ref{diag-26-6-25})  induces  a   homomorphism 
(\ref{eq-26.6.25-1})   of  double  homology  groups.  
\end{proof}

\subsection{Hypergraphs  on  manifolds}\label{26.6.3-ss.3.2}

A  {\it  $k$-uniform  hyperdigraph}  $\vec{\mathcal{H}}_k(M)$
 on  $M$  is  a  
submanifold   of ${\rm  Conf}_k(M)$, 
which   is   
a  collection  of  certain  directed  $k$-hyperedges  on  $M$
\footnote[3]{
In  \cite{h1,hdg},  
a  (directed)  $k$-hyperedge  in   a  hyper(di)graph  consists  of  
$(k+1)$-vertices.  
However,  in  this  paper,  in  order  to  be  consistent  with  the
context  of  configuration  spaces, 
we  let  a  (directed)  $k$-hyperedge  in  a  hyper(di)graph  consist    of  
$k$-vertices.  
}.   
Let  $\vec{\mathcal{H}}_0(M)$   be  $\emptyset$    (the  empty  set  containing  
no  directed  hyperedges)    or  $\{\emptyset\}$   (the  set  containing  $\emptyset$
as  the  single  directed  $0$-hyperedge).  
A    {\it     hyperdigraph}    $\vec{\mathcal{H}} (M)$  on  $M$  is  a  union 
$\bigcup_{k\geq  0} \vec{\mathcal{H}}_k(M)$
 \footnote[4]{
 In  \cite{h1, hdg,  jgp},
    a   (directed)  hyperedge   in  a  hyper(di)graph  
 is  not  allowed  to  be  the  emptyset.   
If  we  add  the  empty set  as  a  (directed)  hyperedge
to a    hyper(di)graph,
 then  by  substituting  the  usual  chain  complex   
 with   an    augmented  chain  complex 
  and  thereby  substituting  the  usual  homology  with   the    reduced  homology,
 the   homological   methods  in    \cite{h1, hdg,  jgp}  can  be  applied  analogously. 
 Therefore,  in  \cite{ca26}, 
 the  hyper(di)graphs  that are  allowed  to  contain   
  the  empty set  as  a  (directed)  hyperedge  are    precisely  called  
  {\it  augmented  hyper(di)graphs}.
 In  this  paper,  
we  allow  the  (directed) hyperedges  in     hyper(di)graphs  
to  be  the  empty  set  and  simply   refer  to   augmented  hyper(di)graphs
  as   hyper(di)graphs.   
    }.  
 Let   $\vec{\mathcal{H}} (M)$  be  a  hyperdigraph  on  $M$.  
 We  call    $\vec{\mathcal{H}} (M)$  {\it  symmetric}  if  
 $\vec{\mathcal{H}}_k(M)$   is  closed  under  the   $\Sigma_k$-permutation   
 (i.e.,         
  for  any  $\vec\sigma_k(M)\in   \vec{\mathcal{H}}_k(M)$  and  any  
 $s\in \Sigma_k$,  $s(\vec\sigma_k(M))\in  \vec{\mathcal{H}}_k(M)$)    
  for  any  $k\geq  1$.   
We  call    $\vec{\mathcal{H}} (M)$  a   
{\it       directed  simplicial  complex}  on  $M$
  if  
 for  any  directed  hyperedge  $\vec\sigma(M)\in \vec{\mathcal{H}} (M)$ 
 and  any    non-empty  subsequence  $\vec\tau(M)$  of   $\vec\sigma(M)$,  
 it  holds  that  $\vec\tau(M)\in \vec{\mathcal{H}} (M)$.     
 We  call    $\vec{\mathcal{H}} (M)$  an  
   {\it    independence  hyperdigraph}  on  $M$
     if  
 for  any  directed  hyperedge  $\vec\sigma(M)\in \vec{\mathcal{H}} (M)$ 
 and  any    directed  hyperedge  $\vec\theta(M)\in  \overrightarrow {\rm   Ind}(M)$  
 such  that  $\vec\sigma(M)$  is  a  subsequence  of   $\vec\theta(M)$
 (i.e.,   $\vec\theta(M)$  is  a  supersequence  of  $\vec\sigma(M)$),  
 it  holds  that  $\vec\theta(M)\in \vec{\mathcal{H}} (M)$.   
  For differentiation,  we  denote  a     directed  simplicial  complex  on  $M$  as  
      $\vec{\mathcal{K}} (M)$  
      and   denote  an         independence   hyperdigraph  on  $M$  as 
       $\vec{\mathcal{L}} (M)$.

       \begin{proposition}
       \label{le-26-1-25-1}
       \begin{enumerate}[(1)]
       \item
       For  any      directed  simplicial  complex
        $\vec{\mathcal{K}} (M)=\bigcup_{k\geq  0} \vec{\mathcal{K}}_k(M)$  on  $M$,  
        any  $k\geq  1$  and  any  $1\leq  i\leq  k$,   
       \begin{eqnarray}\label{eq-260125-b1}
       \partial_k^i(\vec{\mathcal{K}}_k(M))  \subseteq \vec{\mathcal{K}}_{k-1}(M);
       \end{eqnarray}

       \item
         For  any          independence  hyperdigraph  
        $\vec{\mathcal{L}} (M)=\bigcup_{k\geq  0} \vec{\mathcal{L}}_k(M)$  on  $M$,  
         any  $k\geq  0$  and  any  $1\leq  i\leq  k+1$,  
       \begin{eqnarray}\label{eq-260125-b2}
       (\partial_{k+1}^i)^{-1}(\vec{\mathcal{L}}_k(M))  \subseteq \vec{\mathcal{L}}_{k+1}(M). 
       \end{eqnarray}  
       \end{enumerate}
       \end{proposition}
       
       \begin{proof}
       (1)  Let  $k\geq  1$  and  
       let  $\vec\sigma_k(M)\in \vec{\mathcal{K}} _k(M)$.    
       Then  for  any   $1\leq  i\leq  k$, 
       $ \partial_k^i(\vec\sigma_k(M))$  is  either  the  empty  sequence  $\vec\emptyset$  
       or  a    non-empty  subsequence  of   $\vec\sigma_k(M)$.   
       Since   $\vec{\mathcal{K}} (M)$  is  a    directed  simplicial  complex,  
       it  follows  that    $ \partial_k^i(\vec\sigma_k(M))\in   \vec{\mathcal{K}} _{k-1}(M)$.  
       We  obtain  (\ref{eq-260125-b1}).

       (2)  
        Let  $k\geq  0$  and  
       let  $\vec\sigma_k(M)\in \vec{\mathcal{L}} _k(M)$. 
       For  any  $1\leq  i\leq  k+1$,  
       let  $\vec\tau_{k+1}(M)\in {\rm  Conf}_{k+1} (M)$  such  that 
       $     \partial_{k+1}^i(\vec\tau_{k+1}(M))=\vec\sigma_k(M)$. 
       Then  $\vec\sigma_k(M)$  is  a  subsequence  of  $\vec\tau_{k+1}(M)$.  
       Since   $\vec{\mathcal{L}} (M)$  is  an         independence  hyperdigraph,  
       it  follows  that   $\vec\tau_{k+1}(M)\in\vec{\mathcal{L}}_{k+1} (M)$. 
       Thus  
        $(\partial_{k+1}^i)^{-1}(\vec\sigma_k(M))  \subseteq \vec{\mathcal{L}}_{k+1}(M)$.  
       We  obtain  (\ref{eq-260125-b2}).  
              \end{proof}

  A  {\it  $k$-uniform  hypergraph}  ${\mathcal{H}}_k(M)$  on  $M$  is  a  
submanifold   of ${\rm  Conf}_k(M)/\Sigma_k$, 
which   is   
a  collection  of  certain  $k$-hyperedges  on  $M$.  
 Let  $ {\mathcal{H}}_0(M)$   be  $\emptyset$    (the  empty set  containing  
no    hyperedges)    or  $\{\emptyset\}$   (the  set  containing  $\emptyset$
as  the  single     $0$-hyperedge).  
A    {\it    hypergraph}    ${\mathcal{H}} (M)$  on  $M$  is  a  union 
$\bigcup_{k\geq  0}  {\mathcal{H}}_k(M)$.  
 Let   ${\mathcal{H}} (M)$  be  a    hypegraph  on  $M$.  
We  call    ${\mathcal{H}} (M)$  a   {\it       simplicial  complex}  on  $M$
  if  
 for  any   hyperedge  $\sigma(M)\in {\mathcal{H}} (M)$ 
 and  any    non-empty  subset  $\tau(M)$  of   $\sigma(M)$,  
 it  holds  that  $\tau(M)\in {\mathcal{H}} (M)$.     
 We  call    ${\mathcal{H}} (M)$  an   {\it    independence  hypergraph}  on  $M$
     if  
 for  any  hyperedge  $\sigma(M)\in {\mathcal{H}} (M)$ 
 and  any      hyperedge  $\theta(M)\in  {\rm   Ind}(M)$  
 such  that  $\sigma(M)$  is  a  subset  of  $\theta(M)$,  
 it  holds  that  $\theta(M)\in {\mathcal{H}} (M)$.   
  For differentiation,  we  denote  a   simplicial  complex  on  $M$  as  
      ${\mathcal{K}} (M)$  
      and   denote  an      independence   hypergraph  on  $M$  as 
       ${\mathcal{L}} (M)$.   
  For  any     hyperdigraph  $\vec{\mathcal{H}}(M)$  on  $M$,
      we  call   $\pi (\vec{\mathcal{H}}(M))$    the   {\it  underlying     hypergraph}
      of   $\vec{\mathcal{H}}(M)$ 
        on  $M$.

       \begin{proposition}\label{le-2.1-26-3}
For  any   differentiable  manifold  $M$,  
\begin{enumerate}[(1)]
\item
   $\overrightarrow{\rm  Ind}(M)$  
is  a   directed  simplicial  complex  on  $M$  and   ${\rm  Ind}(M)$
is  its   underlying    simplicial  complex  on  $M$. 
Moreover,  
any     directed  simplicial  complex  $\vec{\mathcal{K}}(M)$  on  $M$  
is  a  directed  simplicial  sub-complex  of  $\overrightarrow{\rm  Ind}(M)$ 
and  any      simplicial  complex  $\mathcal{K}(M)$  on  $M$  is  a  
simplicial  sub-complex  of  
${\rm  Ind}(M)$;
\item
   $\overrightarrow{\rm  Ind}(M)$  
is  an     independence  hyperdigraph  on  $M$  and   ${\rm  Ind}(M)$
is  its   underlying    independence  hypergraph   on  $M$. 
Moreover,  
any       independence  hyperdigraph  $\vec{\mathcal{L}}(M)$  on  $M$  
is  an      independence  sub-hyperdigraph  of  $\overrightarrow{\rm  Ind}(M)$ 
and  any       hypergraph   $\mathcal{L}(M)$  on  $M$  is  an      
 independence  sub-hypergraph   of  
${\rm  Ind}(M)$.  
\end{enumerate} 
\end{proposition}

\begin{proof}
For  any  $\vec\sigma(M)\in  
\overrightarrow  {\rm  Ind}(M)$  and  any  subsequence  $\vec\tau(M)$  of  $\vec\sigma(M)$,  
 it  holds  that  $\vec\tau(M)\in  \overrightarrow  {\rm  Ind}(M)$.  
 Thus  
  $\overrightarrow{\rm  Ind}(M)$  
is  a       directed  simplicial  complex  on  $M$. 
With  the  help  of  (\ref{eq-2601-a})   and  (\ref{eq-2601-aa}),  
 ${\rm  Ind}(M)$
is  the    underlying   simplicial  complex     on  $M$. 
 For     any  finite  supersequence  
$\vec\theta(M)$  of  $\vec\sigma(M)$,  
 it  holds  that  $\vec\theta(M)\in  \overrightarrow  {\rm  Ind}(M)$.  
 Thus  
  $\overrightarrow{\rm  Ind}(M)$  
is  an     independence  hyperdigraph  on  $M$   whose  underlying      hypergraph    
is   the  independence  hypergraph 
 ${\rm  Ind}(M)$
        on  $M$.

Since  any    hyperdigraph   $\vec{\mathcal{H}}(M)$
  on  $M$  is  a  graded  submanifold  of  
$\overrightarrow{\rm  Ind}(M)$ 
such that  the  underlying    hypergraph  $\mathcal{H}(M)$ 
is  a  graded  submanifold  of  ${\rm  Ind}(M)$,  
we  obtain  the  following:     
any     directed  simplicial  complex  $\vec{\mathcal{K}}(M)$  on  $M$  
is  a      directed  simplicial  sub-complex  of  $\overrightarrow{\rm  Ind}(M)$ 
 whose   underlying     simplicial  complex  $\mathcal{K}(M)$     is  a      
simplicial  sub-complex  of  
${\rm  Ind}(M)$;  
and  
any      independence  hyperdigraph  $\vec{\mathcal{L}}(M)$  on  $M$  
is  an      independence  sub-hyperdigraph  of  $\overrightarrow{\rm  Ind}(M)$ 
 whose   underlying       hypergraph   $\mathcal{L}(M)$     is  an      
 independence  sub-hypergraph   of  
${\rm  Ind}(M)$.  
               \end{proof}

\begin{proposition}\label{pr-26.5.25.3}
Let  $ \overrightarrow{\mathcal{H}}(M)$  be  a   hyperdigraph  
on  $M$  and  let  ${\mathcal{H}}(M)$  be  the underlying  hypergraph  on 
$M$  with  the  canonical  projection  
$\pi:   \overrightarrow{\mathcal{H}}(M)\longrightarrow  
 {\mathcal{H}}(M)$. 
Let    ${\bf  Sub}(M)$  be  the  full  sub-category  of  ${\bf  M}$
     whose  objects are  submanifolds  of  $M$.   
Then  
\begin{enumerate}[(1)]
\item
there  is  a  subcategory  $ \overrightarrow {\bf  HI}$  of   $\overrightarrow{\bf  FI}$ 
and  a  subcategory  $  {\bf  HI}$  of   $ {\bf  FI}$ with  
$\pi( \overrightarrow {\bf  HI})= {\bf  HI}$  such  that  
$  \overrightarrow{\mathcal{H}}(M)$  is  a  covariant  functor 
 from   $\overrightarrow{\bf  HI}$
  to    $ {\bf  Sub}(M)$ 
     and 
      $ {\mathcal{H}}(M)$  is  a    covariant  functor  from  ${\bf  HI}$  to  $ {\bf  Sub}(M)$ 
     making   the  diagram  commute 
     \begin{eqnarray*}
     \xymatrix{
     \overrightarrow{\bf  HI}\ar[rr]^-{\pi}  \ar[rd]_-{ \overrightarrow {\mathcal{H}}(M)}  
     &&{\bf  HI}\ar[ld]^-{ {\mathcal{H}}(M)}  \\
     &  {\bf  Sub}(M); &
     }
     \end{eqnarray*}

     \item
     $ \overrightarrow{\mathcal{H}}(M)$  is  a  graded  submanifold  of  the  
      $\Delta$-manifold   $ \overrightarrow{\rm  Ind}(M)$. 
      Moreover,  $ \overrightarrow{\mathcal{H}}(M)$  is  a  sub-$\Delta$-manifold  
      of   $ \overrightarrow{\rm  Ind}(M)$  
      if  and  only  if   $ \overrightarrow{\mathcal{H}}(M)$    is  
      a  directed  simplicial  complex  on  $M$;  
     \item   
      ${\mathcal{H}}(M)$  is  a  graded  submanifold  of  the  
      $\Delta$-manifold   ${\rm  Ind}(M)$.    
      Moreover,  ${\mathcal{H}}(M)$  is  a  sub-$\Delta$-manifold  
      of   ${\rm  Ind}(M)$  
      if  and  only  if   ${\mathcal{H}}(M)$    is  
      a  simplicial  complex  on  $M$.   
     \end{enumerate}  
\end{proposition}

\begin{proof}
The  proof  of   (1)   follows from  a  straightforward  verification. 
The  proofs  of  (2)  and   (3)  follow  from     
Lemma~\ref{le-24.5.24.conf1}~(2)  and  
 Proposition~\ref{le-26-1-25-1}~(1).  
\end{proof}

Let  $\vec{\mathcal{H}}(M)$  be  a  hyperdigraph  on  $M$. 
Let   $\mathcal{H}(M)=\pi(\vec{\mathcal{H}}(M))$  be  the  
  underlying     hypergraph  of  $\vec{\mathcal{H}}(M)$.   
For  any   $k\geq  1$   and  any   $\vec\sigma_k(M)\in  \vec{\mathcal{H}}_k(M)$,  
let  
$\Sigma_k(\vec\sigma_k(M)) $  be  the  $\Sigma_k$-orbit  of  
$\vec\sigma_k(M)$  in   ${\rm  Conf}_k(M)$.  
The  smallest  symmetric  
  $k$-uniform  hyperdigraph  containing    $\vec{\mathcal{H}}_k(M)$  is  
\begin{eqnarray*}
{\rm  Sym}(\vec{\mathcal{H}}_k(M))= 
\bigcup_{\vec\sigma_k(M)\in  \vec{\mathcal{H}}_k(M)}
\Sigma_k(\vec\sigma_k(M)).    
\end{eqnarray*}
Let   ${\rm  Sym}(\vec{\mathcal{H}}_0(M))
=\vec{\mathcal{H}}_0(M)$.   
The  smallest  symmetric  
   hyperdigraph  containing    $\vec{\mathcal{H}} (M)$  is  
\begin{eqnarray*}\label{eq-26.7.27}
{\rm  Sym}(\vec{\mathcal{H}}(M))= 
\bigcup_{k\geq  0}
{\rm  Sym}(\vec{\mathcal{H}}_k(M)).    
\end{eqnarray*}
It  follows  directly  that  
\begin{eqnarray}\label{eq-26-3-11-1}
{\rm  Sym}(\vec{\mathcal{H}}(M))=\pi^{-1}(\mathcal{H}(M)).  
\end{eqnarray}

For  any  $\vec\sigma(M)\in  \vec{\mathcal{H}}(M)$,  
let  
$\Delta\vec\sigma(M)$  be  the  collection  of  all  the  subsequences  of  
$\vec\sigma(M)$
and  let  $\bar\Delta\vec\sigma(M)$  be  the  collection  of  all  the  supersequences  of  
$\vec\sigma(M)$  in   $\overrightarrow{\rm  Ind}(M)$.   
By  \cite[Section~6,  (6.1)-(6.4)]{jgp},  
the  {\it  associated    directed  simplicial  complex}
  \begin{eqnarray*}
  \Delta  \vec{\mathcal{H}}(M)=\bigcup_{\vec\sigma(M)\in  \vec{\mathcal{H}}(M)}
  \Delta\vec\sigma(M)
  \end{eqnarray*}
  is  the  smallest    directed  simplicial  complex
  on  $M$  containing  $\vec{\mathcal{H}}(M)$;    
  the  {\it  lower-associated   directed  simplicial  complex}
  \begin{eqnarray*}
  \delta  \vec{\mathcal{H}}(M)=\bigcup_{\Delta\vec\sigma(M)
  \subseteq  \vec{\mathcal{H}}(M)}
  \Delta\vec\sigma(M)
  \end{eqnarray*}
  is  the  largest    directed  simplicial  complex
  on  $M$  contained  in  $\vec{\mathcal{H}}(M)$;    
the  {\it  associated     independence  hyperdigraph}
  \begin{eqnarray*}
  \bar\Delta  \vec{\mathcal{H}}(M)=\bigcup_{\vec\sigma(M)\in  \vec{\mathcal{H}}(M)}
 \bar \Delta\vec\sigma(M)
  \end{eqnarray*}
  is  the  smallest     independence  hyperdigraph
  on  $M$  containing  $\vec{\mathcal{H}}(M)$;
  and   
  the  {\it  lower-associated      independence  hyperdigraph}
  \begin{eqnarray*}
  \bar\delta  \vec{\mathcal{H}}(M)=\bigcup_{\bar\Delta\vec\sigma(M)
  \subseteq  \vec{\mathcal{H}}(M)}
  \bar\Delta\vec\sigma(M)
  \end{eqnarray*}
  is  the  largest     independence  hyperdigraph
  on  $M$  contained  in  $\vec{\mathcal{H}}(M)$.  
    For  any  $ \sigma(M)\in   {\mathcal{H}}(M)$,  
let  
$\Delta \sigma(M)$  be  the  collection  of  all  the  subsets  of  
$ \sigma(M)$
and  let  $\bar\Delta \sigma(M)$  be  the  collection  of  all  the  supersets  of  
$ \sigma(M)$  in   $ {\rm  Ind}(M)$.   
By  \cite[Section~6,  (6.5)-(6.8)]{jgp},  
the  {\it  associated      simplicial  complex}
  \begin{eqnarray*}
  \Delta  {\mathcal{H}}(M)=\bigcup_{ \sigma(M)\in  {\mathcal{H}}(M)}
  \Delta \sigma(M)
  \end{eqnarray*}
  is  the  smallest    simplicial  complex
  on  $M$  containing  $ {\mathcal{H}}(M)$;    
  the  {\it  lower-associated    simplicial  complex}
  \begin{eqnarray*}
  \delta   {\mathcal{H}}(M)=\bigcup_{\Delta \sigma(M)
  \subseteq   {\mathcal{H}}(M)}
  \Delta \sigma(M)
  \end{eqnarray*}
  is  the  largest   simplicial  complex
  on  $M$  contained  in  $ {\mathcal{H}}(M)$;    
the  {\it  associated     independence  hypergraph}
  \begin{eqnarray*}
  \bar\Delta   {\mathcal{H}}(M)=\bigcup_{ \sigma(M)\in   {\mathcal{H}}(M)}
 \bar \Delta \sigma(M)
  \end{eqnarray*}
  is  the  smallest   independence  hypergraph
  on  $M$  containing  ${\mathcal{H}}(M)$; 
  and  
  the  {\it  lower-associated    independence  hypergraph}
  \begin{eqnarray*}
  \bar\delta  {\mathcal{H}}(M)=\bigcup_{\bar\Delta \sigma(M)
  \subseteq   {\mathcal{H}}(M)}
  \bar\Delta \sigma(M)
  \end{eqnarray*}
  is  the  largest    independence  hypergraph
  on  $M$  contained  in  $ {\mathcal{H}}(M)$. 
    By  a  straightforward  verification,    
  \begin{eqnarray}
  \Delta{\rm  Sym}(\vec{\mathcal{H}}(M)) &=&  {\rm  Sym}(\Delta\vec{\mathcal{H}}(M)),
  \label{eq26-3-6-s1} \\
\delta{\rm  Sym}(\vec{\mathcal{H}}(M)) &=&  {\rm  Sym}(\delta\vec{\mathcal{H}}(M)),
  \label{eq26-3-6-s3}\\
    \bar \Delta{\rm  Sym}(\vec{\mathcal{H}}(M))&= & 
   {\rm  Sym}(\bar\Delta\vec{\mathcal{H}}(M)), 
     \label{eq26-3-6-s2}\\
\bar\delta{\rm  Sym}(\vec{\mathcal{H}}(M)) &=&  
{\rm  Sym}(\bar\delta\vec{\mathcal{H}}(M)).  
  \label{eq26-3-6-s4}
  \end{eqnarray}
  
  \begin{proposition}\label{le-26-3-3-1}
  For  any     hyperdigraph  $\vec{\mathcal{H}}(M)$  on  $M$,  
  we  have    commutative  diagrams  
  \begin{eqnarray}\label{26-3-3-diag1}
  \xymatrix{
 \delta \vec{\mathcal{H}}(M) \ar[d]_-{\pi}
 \ar[r]
 & \vec{\mathcal{H}}(M)  \ar[d]_-{\pi}
 \ar[r]
 & \Delta \vec{\mathcal{H}}(M) \ar[d]^-{\pi}
 \ar[r] 
 &\overrightarrow {\rm  Ind}(M)\ar[d]^-{\pi}
 \\
  \delta {\mathcal{H}}(M) 
 \ar[r]
 & {\mathcal{H}}(M)  
 \ar[r]
 & \Delta{\mathcal{H}}(M)
 \ar[r]
 &{\rm  Ind}(M)
  }
  \end{eqnarray}
  and 
    \begin{eqnarray}\label{26-3-3-diag2}
  \xymatrix{
 \bar\delta \vec{\mathcal{H}}(M) \ar[d]_-{\pi}
 \ar[r]
 & \vec{\mathcal{H}}(M)  \ar[d]_-{\pi}
 \ar[r]
 & \bar\Delta \vec{\mathcal{H}}(M) \ar[d]^-{\pi}
 \ar[r] 
 &\overrightarrow {\rm  Ind}(M)\ar[d]^-{\pi}
 \\
  \bar\delta {\mathcal{H}}(M) 
 \ar[r]
 & {\mathcal{H}}(M)  
 \ar[r]
 &\bar \Delta{\mathcal{H}}(M)
 \ar[r]
 &{\rm  Ind}(M)  
  }
  \end{eqnarray}
  such  that  all  the  horizontal  maps  
  are  canonical  inclusions  
  and  all  the  vertical  maps
    are  canonical  projections. 
  \end{proposition}
  
  \begin{proof}
    By   \cite[Section~6]{jgp}, 
    \begin{eqnarray*}
&  \pi(\Delta\vec{\mathcal{H}}(M))=\Delta{\mathcal{H}}(M), ~~~~~~
   \pi(\delta\vec{\mathcal{H}}(M))=\delta{\mathcal{H}}(M),\\
&    \pi(\bar\Delta\vec{\mathcal{H}}(M))=\bar\Delta{\mathcal{H}}(M), ~~~~~~
   \pi(\bar\delta\vec{\mathcal{H}}(M))=\bar\delta{\mathcal{H}}(M). 
  \end{eqnarray*}
Therefore,  with  the  help  of  Proposition~\ref{le-2.1-26-3},
 the  restrictions  of   (\ref{eq-2601-aa})  induce    projections 
\begin{eqnarray*}
 \pi\mid_{\Delta\vec{\mathcal{H}}(M)}: &&
  \Delta\vec{\mathcal{H}}(M) 
\longrightarrow \Delta{\mathcal{H}}(M),
\\
\pi\mid_{\delta\vec{\mathcal{H}}(M)}:  &&
  \delta\vec{\mathcal{H}}(M) 
\longrightarrow \delta{\mathcal{H}}(M),
\\
 \pi\mid_{\bar\Delta\vec{\mathcal{H}}(M)}: && 
  \bar\Delta\vec{\mathcal{H}}(M))
 \longrightarrow\bar\Delta{\mathcal{H}}(M),
 \\
\pi\mid_{\bar\delta\vec{\mathcal{H}}(M)}:  &&
\bar\delta\vec{\mathcal{H}}(M))\longrightarrow\bar\delta{\mathcal{H}}(M)   
\end{eqnarray*}
  such  that  the  diagrams  (\ref{26-3-3-diag1})  and  
    (\ref{26-3-3-diag2})  
  commute. 
  \end{proof}
  
  \begin{corollary}\label{co-26-3-a1}
  For  any   hyperdigraph  $\vec{\mathcal{H}}(M)$  on  $M$,  
  \begin{eqnarray}
  \pi(\Delta{\rm  Sym}(\vec{\mathcal{H}}(M))) &=& \Delta{\mathcal{H}}(M),
  \label{eq-26-3-6-p1}\\
    \pi(\delta{\rm  Sym}(\vec{\mathcal{H}}(M))) &=& \delta{\mathcal{H}}(M),
        \label{eq-26-3-6-p3}\\
    \pi(\bar\Delta{\rm  Sym}(\vec{\mathcal{H}}(M))) &=& \bar\Delta{\mathcal{H}}(M),
      \label{eq-26-3-6-p2}\\
    \pi(\bar\delta{\rm  Sym}(\vec{\mathcal{H}}(M))) &=& \bar\delta{\mathcal{H}}(M).
      \label{eq-26-3-6-p4} 
  \end{eqnarray}  
  \end{corollary}
  
  \begin{proof}
  By    (\ref{eq26-3-6-s1})  and  (\ref{26-3-3-diag1}),
  we  obtain  (\ref{eq-26-3-6-p1}).  
   By    (\ref{eq26-3-6-s3})  and  (\ref{26-3-3-diag1}),
  we  obtain  (\ref{eq-26-3-6-p3}).    
   By    (\ref{eq26-3-6-s2})  and  (\ref{26-3-3-diag2}),
  we  obtain  (\ref{eq-26-3-6-p2}). 
   By    (\ref{eq26-3-6-s4})  and  (\ref{26-3-3-diag2}),
  we  obtain  (\ref{eq-26-3-6-p4}).   
  \end{proof}

\begin{theorem}[Main  Result  I]
\label{th-6.13.1}
Let   $M$  be  a   differentiable  manifold.  
Let   $\vec{\mathcal{H}}(M)$  be  a  hyperdigraph  on  $M$ 
 with  its  underlying  hypergraph  $\mathcal{H}(M)$.  
Then  we  have    a  
sequence  of  double  homology  groups 
\begin{eqnarray}\label{eq-6.27.h1}
&\xymatrix{
H (\delta  \vec{\mathcal{H}}(M);R) \ar[r] 
& H (\vec{\mathcal{H}}(M);R) \ar[r] 
& H (\Delta \vec{\mathcal{H}}(M);R) \ar[r] 
&H (\overrightarrow{\rm  Ind}(M);R).  
  }
\end{eqnarray}
Moreover,      if   $M$  is  a  submanifold  of  $\mathbb{R}$,  
then 
 we  have  a  commutative  diagram    of double    homology  groups
\begin{eqnarray}\label{eq-6.10.h1}
&\xymatrix{
H (\delta  \vec{\mathcal{H}}(M);R) \ar[r]\ar[d]_-{\pi_*}
& H (\vec{\mathcal{H}}(M);R) \ar[r]\ar[d]_-{\pi_*}
& H (\Delta \vec{\mathcal{H}}(M);R) \ar[r]\ar[d]^-{\pi_*}
&H (\overrightarrow{\rm  Ind}(M);R) \ar[d]^-{\pi_*} \\
H (\delta {\mathcal{H}}(M);R) \ar[r] 
& H  ( {\mathcal{H}}(M);R) \ar[r] 
& H (\Delta {\mathcal{H}}(M);R) \ar[r] 
&H ( {\rm  Ind}(M);R).    
}
\end{eqnarray}
\end{theorem}

\begin{proof}
Apply Proposition~\ref{pr-26.5.25.3}~(2)   to  Theorem~\ref{pr-5.22.9}.  
By  (\ref{eq-5.21.h1}),  
we  obtain  (\ref{eq-6.27.h1}).  
 Suppose  in  addition  that 
 $M$  is  a  submanifold  of  $\mathbb{R}$.   
  Then 
   there  is  a  continuous  total  order  on  the  face  maps  
of   the   $\Delta$-manifold  ${\rm   Inf}(M)$.   
   By  
  applying  Proposition~\ref{pr-26.5.25.3}~(2)  and  (3)  to  Theorem~\ref{pr-5.22.9}, 
  we  obtain   the  
sequence  of  double  homology  groups 
 in  the  second  row  of   (\ref{eq-6.10.h1}).  
 By  an   argument  analogous   to  that  of  
 \cite[Lemma~3.8 and  Theorem~3.9]{hdg} 
 and  
 \cite[Theorem~5.1]{reg},  
 the  diagram  commutes 
 \begin{eqnarray}\label{diag-26-6-27-1}
 \xymatrix{
 R(\delta  \vec{\mathcal{H}}(M))\ar[r]  \ar[d]_-{\pi_\#}&
   {\rm  Inf}( \vec{\mathcal{H}}(M))\ar[d]_-{\pi_\#} \ar[r]&
   {\rm  Sup}( \vec{\mathcal{H}}(M))\ar[d]_-{\pi_\#}\ar[r]
    & R(\Delta  \vec{\mathcal{H}}(M))\ar[d]^-{\pi_\#}\ar[r] 
    &R(\overrightarrow{\rm  Ind}(M))\ar[d]^-{\pi_\#}\\
   R(\delta  {\mathcal{H}}(M))\ar[r]  &
   {\rm  Inf}(  {\mathcal{H}}(M))  \ar[r]&
   {\rm  Sup}( {\mathcal{H}}(M))\ar[r]
    & R(\Delta  {\mathcal{H}}(M))\ar[r] 
    &R( {\rm  Ind}(M))    
 }
 \end{eqnarray} 
where  the  horizontal  maps  are  inclusions  of  double-chain  complexes 
and  the  vertical  maps  are  double-chain  maps  induced  by  $\pi$.  
 Applying  the  homology  functor  to   
 (\ref{diag-26-6-27-1}),  
 we  obtain  (\ref{eq-6.10.h1}).  
\end{proof}

  \subsection{Hypergraphs  and  independence  complexes  on  disjoint  unions  of  manifolds}\label{26.5.25-ss3.5}

  Let  $M$  and  $M'$  be  two  disjoint  manifolds.  
 Then  for  any  $k\geq  1$,  a  directed  hyperedge  $\vec\sigma_k(M\sqcup  M')$  
 in    $ {\rm  Conf}_k(M\sqcup  M')$  
 is  of  the  form  
 \begin{eqnarray*}
 \vec\sigma_k(M\sqcup  M')=( \vec\sigma_{k_1}(M), \vec\sigma'_{l_1}(M'), \ldots, 
  \vec\sigma_{k_n}(M), \vec\sigma'_{l_n}(M'))  
  \end{eqnarray*}
  such  that  $n\geq  1$,   $k_1,l_1,\ldots, k_n, l_n\geq  0$,  
  $
  \sum_{i=1}^n  (k_i +l_i)= k   
$,     and   $ \vec\sigma_{k_i}(M) \in {\rm  Conf}_{k_i}(M ) $  
  and   $ \vec\sigma'_{l_i}(M') \in {\rm  Conf}_{l_i}(M' ) $
  for  each  $1\leq  i\leq  n$.  
 Thus 
 \begin{eqnarray*}
 {\rm  Conf}_k(M\sqcup  M') =  \bigsqcup_{~n\geq 1~}
 \bigsqcup_{ ~k_1,l_1,\ldots, k_n, l_n\geq  0 ~\atop
  ~\sum_{i=1}^n  (k_i +l_i) =k~} 
  \prod_{i=1}^n   ({\rm  Conf}_{k_i}(M ) \times  {\rm  Conf}_{l_i}(M' ) ).         
 \end{eqnarray*}
 Note  that   allowing  $k_i, l_i= 0$  may  yield  multiple 
 parameterizations  for  the  same  sequence  via  the  insertion  of  
 empty  tuples;  however,  these redundant  representations 
 correspond  to  the  same  underlying  sequence   without  ambiguity.  
 Consequently,  
it  follows  that  
 \begin{eqnarray}\label{eq-26-march-02-1}
 \overrightarrow{\rm  Ind}(M\sqcup  M')= \mathcal{F}(
  \overrightarrow{\rm  Ind}(M ),     \overrightarrow{\rm  Ind}(   M')),
 \end{eqnarray}
which  is      the  free  product  of  $ \overrightarrow{\rm  Ind}(M )$  
 and      $ \overrightarrow{\rm  Ind}(   M')$ 
 consisting  of  all   the  finite  words  generated by  the  elements  of   
$ \overrightarrow{\rm  Ind}(M )$  
 and      $ \overrightarrow{\rm  Ind}(   M')$. 
 Moreover,
 a    hyperedge  $ \sigma_k(M\sqcup  M')$  
 in    $ {\rm  Conf}_k(M\sqcup  M')/\Sigma_k$  
 is  of  the  form  
 \begin{eqnarray*}
 \sigma_k(M\sqcup  M')= \sigma_{g }(M) \sqcup  \sigma'_{h}(M') 
     \end{eqnarray*}
  such  that  $g,h\geq  0$  and 
     $g+h= k   $.    
 Thus 
  \begin{eqnarray*}
 {\rm  Conf}_k(M\sqcup  M')/\Sigma_k =  
 \bigsqcup_{ g,h\geq  0 \atop
   g+h =k} 
   {\rm  Conf}_{g}(M )/\Sigma_g  *  {\rm  Conf}_{h}(M' )/\Sigma_h,          
 \end{eqnarray*}
 which  is      a  disjoint  union  of  joins  of  the  $g$-uniform  hypergraphs 
$ {\rm  Conf}_{g}(M )/\Sigma_g $  and  the  $h$-uniform  hypergraphs
 $  {\rm  Conf}_{h}(M' )/\Sigma_h$.    
 It  follows  that  
 \begin{eqnarray}\label{eq-26-march-02-2}
 {\rm  Ind}(M\sqcup  M')=  
 {\rm  Ind}(M )  *      {\rm  Ind}(   M'),
 \end{eqnarray}
 which  is  the  join  of  $ {\rm  Ind}(M )$  and      $ {\rm  Ind}(   M')$.

 Let  $\vec{\mathcal{H}}(M)$  and  $\vec{\mathcal{H}}'(M')$  
 be   hyperdigraphs  on  $M$  and  $M'$  respectively. 
 Let  $\mathcal{H}(M)$  be  the  underlying    hypergraph  of  
 $\vec{\mathcal{H}}(M)$
 and  let  $\mathcal{H}'(M)$  be  the  underlying   hypergraph  of  
 $\vec{\mathcal{H}}'(M)$.  
 Then   the  free  product  
 \begin{eqnarray}\label{eq-26-march-02-5}
 \mathcal{F}(\vec{\mathcal{H}}(M),  \vec{\mathcal{H}}'(M'))
 \end{eqnarray} 
 consisting  of  all   the  finite  words  generated by  the  elements  in 
$\vec{\mathcal{H}}(M)$  and   $  \vec{\mathcal{H}}'(M')$ 
is  a     sub-hyperdigraph  of  (\ref{eq-26-march-02-1}).
  The  join
 \begin{eqnarray}\label{eq-26-march-02-3}
 {\mathcal{H}}(M)  *   {\mathcal{H}}'(M') 
 \end{eqnarray} 
is  a  sub-hypergraph  of  (\ref{eq-26-march-02-2})  such  that  
  (\ref{eq-26-march-02-3})  is  the  underlying   hypergraph  of  
  (\ref{eq-26-march-02-5}).  

\begin{proposition}\label{le-1.5}
For  any    hyperdigraphs
$\vec{\mathcal{H}}(M)$  and  $\vec{\mathcal{H}}'(M')$
 on  $M$  and  $M'$  respectively, 
 we  have  commutative  diagrams 
 \begin{eqnarray}\label{26-3-3-diag3}
 \xymatrix{
  \mathcal{F}(\delta\vec{\mathcal{H}}(M), \delta\vec{\mathcal{H}}'(M')) 
 \ar[r]\ar[d]
 &\mathcal{F}(\vec{\mathcal{H}}(M), \vec{\mathcal{H}}'(M'))
  \ar[r]\ar[d]
  & \mathcal{F}(\Delta\vec{\mathcal{H}}(M), \Delta\vec{\mathcal{H}}'(M')) 
  \ar[d]\\
   (\delta   \mathcal{H}(M)) *(\delta\mathcal{H}'(M') )
 \ar[r] 
 &  \mathcal{H}(M) *\mathcal{H}'(M') 
  \ar[r] 
  &(\Delta  \mathcal{H}(M)) * ( \Delta\mathcal{H}'(M')) 
 }
 \end{eqnarray}
 and
 \begin{eqnarray}\label{26-3-3-diag5}
 \xymatrix{
 \mathcal{F}(\bar\delta\vec{\mathcal{H}}(M), \bar\delta\vec{\mathcal{H}}'(M'))
 \ar[r]\ar[d]
 &\mathcal{F}(\vec{\mathcal{H}}(M), \vec{\mathcal{H}}'(M'))
  \ar[r]\ar[d]
  &\mathcal{F}(\bar\Delta\vec{\mathcal{H}}(M),\bar\Delta \vec{\mathcal{H}}'(M'))
  \ar[d]\\
   (\bar\delta   \mathcal{H}(M)) *(\bar\delta \mathcal{H}'(M') )
 \ar[r] 
 &  \mathcal{H}(M) *\mathcal{H}'(M') 
  \ar[r] 
  &(\bar\Delta  \mathcal{H}(M)) *(\bar\Delta  \mathcal{H}'(M')) 
 }
 \end{eqnarray}
 such  that  
 \begin{enumerate}[(1)]
 \item
 if  both  $\vec{\mathcal{H}}(M)$  and  $\vec{\mathcal{H}}'(M')$
 are   directed simplicial  complexes  (resp.     independence  hyperdigraphs),  
 then   all  the  horizontal  maps  in  (\ref{26-3-3-diag3})   (resp.  
 all  the  horizontal  maps  in  (\ref{26-3-3-diag5}))
are  identity;  
 
 \item
  if  both  ${\mathcal{H}}(M)$  and  ${\mathcal{H}}'(M')$
 are      simplicial  complexes
   (resp.    independence  hypergraphs),  
 then    all  the  horizontal  maps  in  the   second  row  of  (\ref{26-3-3-diag3})
  (resp.  all  the  horizontal  maps   in  the  second  row  of    (\ref{26-3-3-diag5}))  
are  identity.
  \end{enumerate}
\end{proposition}

\begin{proof}
It  follows from    
Proposition~\ref{le-26-3-3-1}      that 
 \begin{eqnarray}\label{26-3-3-diag7}
 \xymatrix{
 \delta(\mathcal{F}(\vec{\mathcal{H}}(M), \vec{\mathcal{H}}'(M')))
 \ar[r]\ar[d]
 &\mathcal{F}(\vec{\mathcal{H}}(M), \vec{\mathcal{H}}'(M'))
  \ar[r]\ar[d]
  &\Delta(\mathcal{F}(\vec{\mathcal{H}}(M), \vec{\mathcal{H}}'(M')))
  \ar[d]\\
   \delta(  \mathcal{H}(M) *\mathcal{H}'(M') )
 \ar[r] 
 &  \mathcal{H}(M) *\mathcal{H}'(M') 
  \ar[r] 
  &\Delta( \mathcal{H}(M) * \mathcal{H}'(M')) 
 }
 \end{eqnarray}
 and
 \begin{eqnarray}\label{26-3-3-diag8}
 \xymatrix{
 \bar\delta(\mathcal{F}(\vec{\mathcal{H}}(M), \vec{\mathcal{H}}'(M')))
 \ar[r]\ar[d]
 &\mathcal{F}(\vec{\mathcal{H}}(M), \vec{\mathcal{H}}'(M'))
  \ar[r]\ar[d]
  &\bar\Delta(\mathcal{F}(\vec{\mathcal{H}}(M), \vec{\mathcal{H}}'(M')))
  \ar[d]\\
   \bar\delta(  \mathcal{H}(M) *\mathcal{H}'(M') )
 \ar[r] 
 &  \mathcal{H}(M) *\mathcal{H}'(M') 
  \ar[r] 
  &\bar\Delta( \mathcal{H}(M) * \mathcal{H}'(M'))    
 }
 \end{eqnarray}
 where  all  the vertical  maps  are  canonical  projections 
 induced  by  taking  the  underlying  hypergraphs.   
On  the  other  hand,  similar  to   \cite[Subsection~2.1,  (i)'-(iv)']{jktr2023}, 
 we  have  
\begin{eqnarray}
\label{eq-26-3-3-51}
 \delta(\mathcal{F}(\vec{\mathcal{H}}(M), \vec{\mathcal{H}}'(M')))
 &=&
  \mathcal{F}(\delta\vec{\mathcal{H}}(M), \delta\vec{\mathcal{H}}'(M')),\\
  \label{eq-26-3-3-52}
 \Delta(\mathcal{F}(\vec{\mathcal{H}}(M), \vec{\mathcal{H}}'(M')))
 &=&
  \mathcal{F}(\Delta\vec{\mathcal{H}}(M), \Delta\vec{\mathcal{H}}'(M')),  \\
  \label{eq-26-3-3-53}
  \bar\delta(\mathcal{F}(\vec{\mathcal{H}}(M), \vec{\mathcal{H}}'(M')))
 &=&
  \mathcal{F}(\bar\delta\vec{\mathcal{H}}(M), \bar\delta\vec{\mathcal{H}}'(M')),\\
  \label{eq-26-3-3-54}
 \bar\Delta(\mathcal{F}(\vec{\mathcal{H}}(M), \vec{\mathcal{H}}'(M')))
 &=&
  \mathcal{F}(\bar\Delta\vec{\mathcal{H}}(M), \bar\Delta\vec{\mathcal{H}}'(M'))
\end{eqnarray}
and 
\begin{eqnarray}
\label{eq-26-3-3-55}
 \delta({\mathcal{H}}(M) * {\mathcal{H}}'(M'))
 &=&
   (\delta{\mathcal{H}}(M))* ( \delta{\mathcal{H}}'(M')),\\
  \label{eq-26-3-3-56}
 \Delta({\mathcal{H}}(M)*  {\mathcal{H}}'(M'))
 &=&
 (\Delta{\mathcal{H}}(M)) *( \Delta{\mathcal{H}}'(M')),  \\
  \label{eq-26-3-3-57}
  \bar\delta({\mathcal{H}}(M)* {\mathcal{H}}'(M'))
 &=&
  (\bar\delta{\mathcal{H}}(M))*( \bar\delta{\mathcal{H}}'(M')),\\
  \label{eq-26-3-3-58}
 \bar\Delta({\mathcal{H}}(M)* {\mathcal{H}}'(M'))
 &=&
  (\bar\Delta{\mathcal{H}}(M))* (\bar\Delta{\mathcal{H}}'(M')).  
\end{eqnarray}
By  (\ref{26-3-3-diag7}),  
(\ref{eq-26-3-3-51}),  
(\ref{eq-26-3-3-52}),  
 (\ref{eq-26-3-3-55})  and    
 (\ref{eq-26-3-3-56}),  
 we  obtain  (\ref{26-3-3-diag3}). 
 By  (\ref{26-3-3-diag8}),  
(\ref{eq-26-3-3-53}),  
(\ref{eq-26-3-3-54}),  
 (\ref{eq-26-3-3-57})  and    
 (\ref{eq-26-3-3-58}),  
 we  obtain  (\ref{26-3-3-diag5}).

(1)   
 {\sc  Case~1}.  
   Suppose 
 both  $\vec{\mathcal{H}}(M)$  and  $\vec{\mathcal{H}}'(M')$
 are    directed simplicial  complexes. 
 Then   
\begin{eqnarray}\label{eq-26-3-5-1}
\delta\vec{\mathcal{H}}(M)=\vec{\mathcal{H}}(M)=\Delta\vec{\mathcal{H}}(M),
~~~~~~
\delta\vec{\mathcal{H}}'(M')=\vec{\mathcal{H}}'(M')=\Delta\vec{\mathcal{H}}'(M').
\end{eqnarray}  
 Moreover,    
 both  ${\mathcal{H}}(M)$  and  ${\mathcal{H}}'(M')$
 are     simplicial  complexes.
 Thus  
 \begin{eqnarray}\label{eq-26-3-5-2}
 \delta\mathcal{H}(M)={\mathcal{H}}(M)=\Delta\mathcal{H}(M),
 ~~~~~~
  \delta\mathcal{H}'(M')={\mathcal{H}}'(M')=\Delta\mathcal{H}'(M').
 \end{eqnarray}   
By  (\ref{eq-26-3-5-1}), 
the 
 horizontal  maps  in  the   first  row  of  (\ref{26-3-3-diag3})
 are  identity  and  
 by  (\ref{eq-26-3-5-2}), 
the 
 horizontal  maps  in  the  second  row  of  (\ref{26-3-3-diag3})
 are  identity.

{\sc  Case~2}.  
   Suppose 
 both  $\vec{\mathcal{H}}(M)$  and  $\vec{\mathcal{H}}'(M')$
 are     independence  hyperdigraphs.  
  Then   
\begin{eqnarray}\label{eq-26-3-5-3}
\bar\delta\vec{\mathcal{H}}(M)=\vec{\mathcal{H}}(M)=\bar\Delta\vec{\mathcal{H}}(M),
~~~~~~
\bar\delta\vec{\mathcal{H}}'(M')=\vec{\mathcal{H}}'(M')=\bar\Delta\vec{\mathcal{H}}'(M').
\end{eqnarray}  
Moreover,  
 both  ${\mathcal{H}}(M)$  and  ${\mathcal{H}}'(M')$
 are     independence  hypergraphs. 
 Thus  
 \begin{eqnarray}\label{eq-26-3-5-4}
 \bar\delta\mathcal{H}(M)={\mathcal{H}}(M)=\bar\Delta\mathcal{H}(M),
 ~~~~~~
  \bar\delta\mathcal{H}'(M')={\mathcal{H}}'(M')=\bar\Delta\mathcal{H}'(M').
 \end{eqnarray}   
By  (\ref{eq-26-3-5-3}), 
the 
 horizontal  maps  in  the   first  row  of  (\ref{26-3-3-diag5})
 are  identity  and  
 by  (\ref{eq-26-3-5-4}), 
the 
 horizontal  maps  in  the  second  row  of  (\ref{26-3-3-diag5})
 are  identity.

 (2)  
 The  proof  is  similar to  (1).   
\end{proof}

\section{The  hypergraphs  of  packings  and  coverings}\label{s.27.7.28.4}

In  this  section,  
we   subsequently  discuss     the  constrained  configuration spaces, 
the  space  of  packings, 
the  space  of 
packings   with  hypergraph  constraints, the  space  of  coverings, 
   and  the  space  of  coverings   with  hypergraph  constraints
    as  
     examples for  hypergraphs   on  manifolds.

\subsection{Constrained  configuration  spaces  on  Riemannian  manifolds}\label{ss4.1-26apr}

Let   $g$  be   a  Riemannian  metric  on  $M$.  
Let  $d_M$  be  the  geodesic  distance  on  $M$   induced  by   $g$.  
Then  $d_M$  is  a  continuous  map  $d_M:  M\times  M\longrightarrow  [0,+\infty]$.  
 In  addition,  if  $M$  is  connected,  then  $d_M:  M\times  M\longrightarrow  [0,+\infty)$. 
 Let  $\mathbb{I}\subseteq  [0,+\infty]$  be  any  subset  of  $ [0,+\infty]$.   
The     $k$-th  {\it   ordered   configuration  space     on  $M$   with  constraint  
 $\mathbb{I}$}  is  
 \begin{eqnarray*}
 {\rm  Conf}_k(M,\mathbb{I}) =  \{(p_1,  p_2,  \ldots,  p_k)\in  {\rm  Conf}_k(M) \mid   d_M(p_i,p_j)\in \mathbb{I}    {\rm~for~all~}  i\neq  j\}.  
 \end{eqnarray*}
 The  $k$-th  {\it   unordered   configuration  space     on  $M$   with  constraint    $\mathbb{I}$}   
 is 
the  $\Sigma_k$-orbit  space 
 \begin{eqnarray*}
 {\rm  Conf}_k(M,  \mathbb{I})/\Sigma_k  =  \{\{p_1,  p_2,  \ldots,  p_k\}\in  {\rm  Conf}_k(M)/\Sigma_k\mid  d_M(p_i,p_j)\in \mathbb{I} 
 {\rm~for~all~}  i\neq  j\}.    
 \end{eqnarray*}
 The    {\it  directed  independence  complex   on  $M$   with  constraint  
 $\mathbb{I}$}  is  
\begin{eqnarray}\label{eq-260306-ind1}
\overrightarrow {\rm  Ind}(M, \mathbb{I})=\bigcup_{k\geq  0}
 {\rm  Conf}_k(M,\mathbb{I}) 
\end{eqnarray}
 and 
 the   {\it   independence  complex   on  $M$   with  constraint 
  $\mathbb{I}$}   is   
 \begin{eqnarray}\label{eq-260306-ind2}
 {\rm  Ind}(M, \mathbb{I})=\bigcup_{k\geq  0}
 {\rm  Conf}_k(M,\mathbb{I}) /\Sigma_k. 
\end{eqnarray}
Then   
   (\ref{eq-260306-ind1})  is  a  symmetric 
   directed  simplicial  sub-complex   
of  $\overrightarrow {\rm  Ind}(M)$  and  
(\ref{eq-260306-ind2})  is     the  underlying     simplicial complex of 
 (\ref{eq-260306-ind1}),  which  is  an
    simplicial  sub-complex   
of  ${\rm  Ind}(M)$.

 In  particular,  
 let  $\Omega=\{(t,s)\mid  0\leq   t<s  \leq  +\infty\}$  and  
 let  $\mathbb{I}$  be  an  interval  of  the  form  
 $(t,s)$,  $(t,s]$,       
 $[t,s)$  or   $[t,s]$  for  $(t,s)\in \Omega$. 
 Without loss  of  generality,  we  take  $\mathbb{I}=(t,s]$  for $(t,s)\in \Omega$.     
   We  have  a  $\Sigma_k$-equivariant  double  filtration  
 \begin{eqnarray}\label{eq-filt-1}
  {\rm  Conf}_k(M,(-,-])=\{ {\rm  Conf}_k(M,(t,s])\mid  (t,s)\in \Omega  \}
 \end{eqnarray}
 as  well  as  an  induced  double  filtration 
  \begin{eqnarray}\label{eq-filt-2}
  {\rm  Conf}_k(M,(-,-])/\Sigma_k=\{ {\rm  Conf}_k(M,(t,s])/\Sigma_k\mid  
    (t,s)\in \Omega  \}
 \end{eqnarray}
 such  that  
 \begin{eqnarray*}
 &{\rm  Conf}_k(M,(t_2,s])\subseteq  {\rm  Conf}_k(M,(t_1,s]), \\
& {\rm  Conf}_k(M,(t_2,s])/\Sigma_k\subseteq  {\rm  Conf}_k(M,(t_1,s])/\Sigma_k
 \end{eqnarray*}  
 for  any  $0\leq  t_1\leq  t_2  <s\leq  +\infty$   and   
 \begin{eqnarray*}
 &{\rm  Conf}_k(M,(t,s_1])\subseteq  {\rm  Conf}_k(M,(t,s_2]),
 \\
 & {\rm  Conf}_k(M,(t,s_1])/\Sigma_k\subseteq  {\rm  Conf}_k(M,(t,s_2])/\Sigma_k
 \end{eqnarray*}  
 for  any  $0\leq  t    <s_1\leq  s_2  \leq  +\infty$.

 \begin{definition}\label{le-26-10-21.1}
 For  any   $(t,s) \in \Omega$,     
 \begin{enumerate}[(1)]
 \item
we  say  that  $(t,s)$   is  a    {\it  homotopic   critical  point}   of  (\ref{eq-filt-1})   if     
for  any  open   neighborhood  
$
 U 
$
    of  $(t,s)$   \footnote[5]{  In   the  case  where   $s=+\infty$,  
    an   open  neighborhood  $U$  of  
    $(t,+\infty)\in \Omega$  is  given  by  
    $W\times  (x,+\infty]$,  where  $W$  is  an  open  neighborhood 
    of  $t\in [0,+\infty)$  and   $0\leq  x<+\infty$.  
    Here  $U$   is    $2$-dimensional   and  $W$  is   $1$-dimensional.  
    },    
    there  exist 
    $(t',s'),  (t'',s'')  \in   U$  with  $t''\leq   t'<s'\leq   s''  $    such  that  
    the   canonical  inclusion  of   
    $ {\rm  Conf}_k(M,(t',s'])$  into   $ {\rm  Conf}_k(M,(t'',s''])$ 
    is  not  a  homotopy  equivalence;
    \item
    we  say  that  $(t,s)$   is  a    {\it  homotopic   critical  point}   of  (\ref{eq-filt-2})   if    
for  any  open   neighborhood  
$
 U 
$
    of  $(t,s)$,  
    there  exist 
    $(t',s'),  (t'',s'')  \in   U$  with  $t''\leq   t'<s'\leq   s''  $    such  that  
    the   canonical  inclusion   of  
    $ {\rm  Conf}_k(M,(t',s'])/\Sigma_k$  into    $ {\rm  Conf}_k(M,(t'',s''])/\Sigma_k$ 
    is  not  a  homotopy  equivalence.  
    \end{enumerate}  
 \end{definition}

  We  denote  the  sets  of  all  the  homotopic  critical  points   of  
  ${\rm  Conf}_k(M,(-,-])$  and  ${\rm  Conf}_k(M,(-,-])/\Sigma_k$    respectively,
   as  
  \begin{eqnarray*}
  {\rm   Crit}^{\simeq} ({\rm  Conf}_k(M,(-,-])),~~~~~~ {\rm  Crit}^{\simeq}({\rm  Conf}_k(M,(-,-])/\Sigma_k).     
  \end{eqnarray*}   
  Note  that these  sets  
    are  called   {\it     spectra}   and  are  
    denoted   by   ${\rm Spec}$    in  \cite[Sect.~4]{gromov1}.

 \begin{definition}\label{le-26-10-21.2}
 For  any   $(t,s) \in \Omega$,     
 \begin{enumerate}[(1)]
 \item
we  say  that   $(t,s)$   is  a  {\it    homological   critical  point}   of  (\ref{eq-filt-1})   if    
for  any  open   neighborhood  
$
 U 
$
    of  $(t,s)$,  
    there  exist 
    $(t',s'),  (t'',s'')  \in   U$  with  $t''\leq   t'<s'\leq   s''  $    such  that  
    the    homomorphism  from  
    $H_\bullet( {\rm  Conf}_k(M,(t',s']))$  to   $ H_\bullet( {\rm  Conf}_k(M,(t'',s'']))$ 
    induced  by  the  canonical  inclusion   
    is  not  an  isomorphism;
    \item
    we  say  that  $(t,s)$   is  a   {\it   homological   critical  point}   of  (\ref{eq-filt-2})   if    
for  any  open   neighborhood  
$
 U 
$
    of  $(t,s)$,  
    there  exist 
    $(t',s'),  (t'',s'')  \in   U$  with  $t''\leq   t'<s'\leq   s''  $    such  that  
     the     homomorphism   from  
    $H_\bullet( {\rm  Conf}_k(M,(t',s'])/\Sigma_k)$ 
    to  
    $  H_\bullet( {\rm  Conf}_k(M,(t'',s''])/\Sigma_k)$ 
    induced  by  the  canonical  inclusion  
    is  not  an  isomorphism.  
    \end{enumerate}  
 \end{definition}

     We  denote  the  sets  of  all  the  homological  critical  points   of  
  ${\rm  Conf}_k(M,(-,-])$  and  ${\rm  Conf}_k(M,(-,-])/\Sigma_k$   respectively,   as  
  \begin{eqnarray}\label{eq-26-feb-27-1}
  {\rm   Crit}^{H} ({\rm  Conf}_k(M,(-,-])),~~~~~~ {\rm  Crit}^{H}({\rm  Conf}_k(M,(-,-])/\Sigma_k).      
  \end{eqnarray}

     \begin{theorem}
     \label{pr-4.5.1}
     For  any  Riemannian  manifold  $M$,   
     \begin{eqnarray}\label{eq-26-feb-27-2}
  {\rm   Crit}^{\simeq} ({\rm  Conf}_k(M,(-,-])) =  
   {\rm   Crit}^{\simeq}({\rm  Conf}_k(M,(-,-])/\Sigma_k),      
  \end{eqnarray} 
  which   contains   a  subset
  \begin{eqnarray}\label{eq-6.12.u1}
  {\rm  Crit}^H({\rm  Conf}_k(M,(-,-] )) \cup  
 {\rm  Crit}^H({\rm  Conf}_k(M,(-,-])/\Sigma_k ). 
  \end{eqnarray} 
   \end{theorem}

   In  order  to  prove  Theorem~\ref{pr-4.5.1}, 
   we  need  the  next  lemma.  
   
   \begin{lemma}
   \label{le-26.6.12.1}
   Let  $X$  be  a  CW-complex.  
   Let  $G$  be  a  group  acting  on  $X$  freely  and  properly  discontinuously.  
   Let  $Y$  be  a  $G$-invariant  sub-CW-complex  of  $X$.  
Then  the  canonical  inclusion  $Y\subseteq  X$  is a  homotopy  equivalence 
if  and  only  if  the  canonical  inclusion  $Y/G\subseteq  X/G$  is a  
homotopy  equivalence.  
   \end{lemma}
   
   \begin{proof}
       We  have  a  commutative  diagram
   \begin{eqnarray*}
   \xymatrix{
   Y\ar[d]_-{\pi}\ar[r]^-{i}  &X  \ar[d]^-{\pi}\\
   Y/G\ar[r]^-{i/G}   &X/G, 
   }
   \end{eqnarray*}
   where  the  horizontal  maps  are  canonical  inclusions 
   and  the vertical  maps are  normal coverings 
    induced  by  the  $G$-actions  
     whose   deck transformation group  is  $G$.

    {\sc  Case~1}.   Both  $X$  and  $Y$  are  connected.

   Then  both  $X/G$  and  $Y/G$  are  connected  as  well.  
  Thus  
we  have a  commutative  diagram  of  fundamental  groups  
   \begin{eqnarray}\label{diag-6.12.1}
   \xymatrix{
   1\ar[r]  &\pi_1(Y)\ar[d]_-{i_*}\ar[r]  & \pi_1(Y/G)\ar[d]^-{(i/G)_*}\ar[r]   &G \ar@{=}[d]\ar[r] &1\\
   1\ar[r]  & \pi_1(X)\ar[r]  & \pi_1(X/G) \ar[r]   &G \ar[r] &1
   }
   \end{eqnarray}
   where  each  row  is a  short  exact  sequence.  
 By the  Five  Lemma  and  diagram  chasing,  
   it  follows  that  in  (\ref{diag-6.12.1}), 
   $i_*$  is  an  isomorphism  if  and  only  if  
   $(i/G)_*$  is  an  isomorphism.  
    Moreover,  
   for  any  $n\geq  2$,  
   we  have  a  commutative  diagram  of  homotopy  groups 
   \begin{eqnarray}\label{diag-6.12.2}
   \xymatrix{
  \pi_n(Y)\ar[d]_-{\pi_*}^-{\cong }\ar[r]^-{i_*}   & \pi_n(X) \ar[d]^-{\pi_*}_-{\cong }\\
    \pi_n(Y/G)\ar[r]^-{(i/G)_*}  & \pi_n(X/G)
   }
   \end{eqnarray}
   where  the  vertical  maps  are  isomorphisms.  
   Hence   in   (\ref{diag-6.12.2}), 
   $i_*$  is  an  isomorphism  if  and  only  if  
   $(i/G)_*$  is  an  isomorphism.  
   Therefore,   
   by  (\ref{diag-6.12.1})  and  (\ref{diag-6.12.2}),   
   $i$  is  a  weak  homotopy  equivalence  if   and  only  if  
$i/G$  is a  weak  homotopy  equivalence.     
   Since  $X$  and  $Y$  
   (resp.  $X/G$  and  $Y/G$)  are  $CW$-complexes, 
   by  the  Whitehead  Theorem, 
  $i$  (resp.  $i/G$)  is  a   homotopy equivalence 
   if  and  only  if    it  is  a  weak  homotopy  equivalence.      
  Consequently,  
     $i$  is  a     homotopy  equivalence  if   and  only  if  
$i/G$  is a    homotopy  equivalence.

{\sc  Case~2}.   Either  $X$  or  $Y$  is  not  connected.

Let  $\{X_\alpha\}_{\alpha\in  A}$  and  
$\{Y_\beta\}_{\beta\in  B}$  be  the  connected  component decompositions  of 
$X$ and  $Y$  respectively.  
 Then  $G$  acts  on  these  connected  components.
 For  each  $\alpha\in  A$,  
   the  stabilizer  of  $G$  on  $X_\alpha$  is  a  subgroup  of  $G$  given by
   \begin{eqnarray*}
   {\rm  Stab}_G(X_\alpha)=\{g\in  G\mid  gX_\alpha =X_\alpha\}. 
   \end{eqnarray*}
   Then   $ {\rm  Stab}_G(X_\alpha)$
   acts  on  $X_\alpha$  freely  and  properly  discontinuously.  
   Moreover,  
   $X_\alpha /{\rm  Stab}_G(X_\alpha)$  is  a  connected  component  of  $X/G$
   such  that  $\pi^X_\alpha: X_\alpha\longrightarrow  X_\alpha /{\rm  Stab}_G(X_\alpha)$
   is  a  normal  covering  with  deck  transformation  group  ${\rm  Stab}_G(X_\alpha)$.  
   Similarly,  for  each  $\beta\in  B$,  
   the  stabilizer  of  $G$  on  $Y_\beta$  is  a  subgroup  of  $G$  given by
   \begin{eqnarray*}
   {\rm  Stab}_G(Y_\beta)=\{g\in  G\mid  gY_\beta =Y_\beta\}. 
   \end{eqnarray*}
   Then    
   $Y_\beta /{\rm  Stab}_G(Y_\beta)$  is  a  connected  component  of  $Y/G$
   such  that  $\pi^Y_\beta: Y_\beta\longrightarrow  Y_\beta /{\rm  Stab}_G(Y_\beta)$
   is  a  normal  covering  with  deck  transformation  group  ${\rm  Stab}_G(Y_\beta)$. 
  We  have  a  commutative  diagram
   \begin{eqnarray*}
   \xymatrix{
   Y_\beta\ar[d]_-{\pi^Y_\beta}\ar[rrr]^-{i\mid_{Y_\beta}}  &&&X_\alpha \ar[d]^-{\pi^X_\alpha}\\
   Y_\beta /{\rm  Stab}_G(Y_\beta)\ar[rrr]^-{(i/G)\mid_{Y_\beta /{\rm  Stab}_G(Y_\beta)}}   &&&X_\alpha /{\rm  Stab}_G(X_\alpha)  
   }
   \end{eqnarray*}
where,   by   {\sc  Case~1},  
  $i\mid_{Y_\beta}$  is  a  homotopy  equivalence  if  and  only  if  
  $(i/G)\mid_{Y_\beta /{\rm  Stab}_G(Y_\beta)}$  is  a  homotopy  equivalence.

   The  canonical  inclusion  $i:  Y\longrightarrow  X$  is  a  homotopy  equivalence 
  if  and  only  if  
  there  is  a  bijection  between  $A$  and  $B$ 
  such  that   for  each  $\beta\in  B$,
     $i\mid_{Y_\beta}$ 
     is  
  a  homotopy  equivalence  to   the  corresponding  $X_\alpha$.  
  Similarly, 
  the  canonical  inclusion  $i/G:  Y/G\longrightarrow  X/G$ 
   is  a  homotopy  equivalence 
  if  and  only  if  $Y/G$  and  $X/G$  have  
  the  same  number  of  connected  components 
    and  each   $(i/G)\mid_{Y_\beta /{\rm  Stab}_G(Y_\beta)}$
     is  
  a  homotopy  equivalence.  
Therefore,
$i$  is  a     homotopy  equivalence  if   and  only  if  
$i/G$  is a    homotopy  equivalence. 
      \end{proof}

   \begin{proof}[Proof  of  Theorem~\ref{pr-4.5.1}]
  For   any  $0\leq  t'\leq  t\leq  s\leq  s'\leq  +\infty$,  
  let  $X$  be  $ {\rm  Conf}_k(M,(t',s']) $  and  
  let  $Y$  be  ${\rm  Conf}_k(M,(t,s])$   in    Lemma~\ref{le-26.6.12.1}.  
  Then  
  by  Lemma~\ref{le-26.6.12.1},  
    the  canonical  inclusion       
   \begin{eqnarray}\label{eq-26-feb-27-5}
    {\rm  Conf}_k(M,(t,s])\subseteq    {\rm  Conf}_k(M,(t',s']) 
    \end{eqnarray}
   is  a  homotopy  equivalence  if  and  only  if     
     the  induced  inclusion  
       \begin{eqnarray}\label{eq-26-feb-27-3}
          {\rm  Conf}_k(M,(t,s])/\Sigma_k \subseteq  {\rm  Conf}_k(M,(t',s'])/\Sigma_k 
          \end{eqnarray}
   is  a  homotopy  equivalence.  
  The  proof  of  (\ref{eq-26-feb-27-2})  follows.

     Since  the  homology  is  invariant  under  homotopy  equivalence,  
   both      sets  defined  in  (\ref{eq-26-feb-27-1})  are  subsets  of   those  in  (\ref{eq-26-feb-27-2}).  
   Precisely,  for   any  $0\leq  t'\leq  t\leq  s\leq  s'\leq  +\infty$,  
   suppose  the  canonical  inlclusion  (\ref{eq-26-feb-27-5}) 
   is  a  homotopy  equivalence. 
   Then  (\ref{eq-26-feb-27-5})  induces  an   isomorphism  of  homology
   \begin{eqnarray*}
    H_\bullet(  {\rm  Conf}_k(M,(t,s]) )\cong  H_\bullet({\rm  Conf}_k(M,(t',s']) ). 
   \end{eqnarray*}
   Hence  
   \begin{eqnarray*}
  \Omega\setminus 
    {\rm   Crit}^{\simeq}({\rm  Conf}_k(M,(-,-]))
     \subseteq 
     \Omega\setminus  {\rm  Crit}^H({\rm  Conf}_k(M,(-,-]) ),  
   \end{eqnarray*}
   which  implies 
    \begin{eqnarray}\label{eq-26-apr-5-2}
 {\rm  Crit}^H({\rm  Conf}_k(M,(-,-]) )   \subseteq  
  {\rm   Crit}^{\simeq}({\rm  Conf}_k(M,(-,-])).   
    \end{eqnarray} 
    Similarly,  
        \begin{eqnarray} 
  {\rm  Crit}^H({\rm  Conf}_k(M,(-,-])/\Sigma_k )
     \subseteq    {\rm   Crit}^{\simeq}({\rm  Conf}_k(M,(-,-])/\Sigma_k).   
    \label{eq-26-apr-5-22}
  \end{eqnarray} 
  By  (\ref{eq-26-feb-27-2}),  
    (\ref{eq-26-apr-5-2})  and (\ref{eq-26-apr-5-22}),  
    we  obtain  that  (\ref{eq-6.12.u1})  is  a  subset  of  
     (\ref{eq-26-feb-27-2}).  
 \end{proof}

    Suppose  the  coefficients  of  the   homology     are  taken   from  integers.  
    Then   the  next  corollary  follows  from  Theorem~\ref{pr-4.5.1}.  
      
     \begin{corollary}\label{co-26.4.5.1}
     For  any
         $(t_0,s_0)\in \Omega$,   let  
   $\Omega(t_0,s_0)=\{(t,s)\in\Omega\mid  t_0\leq   t<s  \leq s_0\}$.  
   \begin{enumerate}[(1)]
   \item
      Suppose        
   $  {\rm  Conf}_k(M,(t,s])$  is  connected  and  simply-connected
    for  any  $(t,s)\in\Omega(t_0,s_0)$.  Then  $\Omega(t_0,s_0)$  
    does  not  intersect  with
    $
    {\rm  Crit}^H( {\rm  Conf}_k(M,(-,-]))  
    $
    if  and  only  if   $\Omega(t_0,s_0)$  
    does  not  intersect  with
    $
    {\rm  Crit}^{\simeq}({\rm  Conf}_k(M,(-,-]))   
    $;  
    \item
     Suppose        
   $  {\rm  Conf}_k(M,(t,s])/\Sigma_k$  is  connected  and  simply-connected
    for  any  $(t,s)\in\Omega(t_0,s_0)$.  Then 
    $\Omega(t_0,s_0)$  
    does  not  intersect  with
    $
    {\rm  Crit}^H( {\rm  Conf}_k(M,(-,-])/\Sigma_k)
   $
    if  and  only  if 
    $\Omega(t_0,s_0)$  
    does  not  intersect  with
    $
    {\rm   Crit}^{\simeq}({\rm  Conf}_k(M,(-,-])/\Sigma_k)      
    $.  
    \end{enumerate}
   \end{corollary}
   
   \begin{proof}
   (1)  
       ($\Longrightarrow$):   
   Suppose   $\Omega(t_0,s_0)$  
    does  not  intersect  with
    ${\rm  Crit}^H({\rm  Conf}_k(M,(-,-]))   $. 
 Then  for  any   $t_0\leq  t'\leq  t\leq  s\leq  s'\leq  s_0$, 
   the  canonical  inclusion  
   (\ref{eq-26-feb-27-5})  induces  an  isomorphism  of  homology groups  in  all  dimensions. 
   Thus    by  the  Hurewicz  Theorem  and  the  Whitehead  Theorem,  
  (\ref{eq-26-feb-27-5})  is  a  homotopy  equivalence.  
    Therefore,    $\Omega(t_0,s_0)$  
    does  not  intersect  with  $  {\rm   Crit}^{\simeq}({\rm  Conf}_k(M,(-,-]))   $.

        ($\Longleftarrow$):  
       The  proof  follows  from  (\ref{eq-26-apr-5-2}).

       (2)  The  proof  is  analogous  to  that  of   (1).  
   \end{proof}

 \subsection{Packings  of  closed  $r$-neighborhoods  in  Riemannian  manifolds}
 \label{ss-26.7.27.1}

Let  $p\in  M$.   Let  $r> 0$.    
The  {\it  open  $r$-neighborhood}  of  $p$  in  $M$   is  
\begin{eqnarray*}
N(r,p,M) = \{q\in  M\mid  d_M(p,q)<   r\}  
\end{eqnarray*}  
and the  {\it  closed  $r$-neighborhood}  of  $p$  in  $M$   is  
\begin{eqnarray*}
\overline { N(r,p,M) }= \{q\in  M\mid  d_M(p,q)\leq     r\}.    
\end{eqnarray*} 
If  we  fix  $p\in  M$,  then when  
$r$  is  small  enough,     the  exponential  map   gives  a  
homeomorphism  from  the  open  $r$-ball  centered at  the  origin  in  the  Euclidean  space
to the  open  $r$-neighborhood     of   $p$  in  $M$.

An   {\it    ordered   $k$-packing  of  closed  $r$-neighborhoods}    in   $M$   is   
 an  ordered  sequence  $(p_1,  p_2,  \ldots,  p_k)$  of  $k$  distinct  points  in  $M$  
 such  that  the   sets  $\overline{N(r,p_i,M)}$  for   $1\leq  i\leq  k$    are  mutually  non-intersecting.  
 Equivalently,  it   
 is   a  directed  $k$-hyperedge  $\vec{\sigma}(M)=(p_1,  p_2,  \ldots,  p_k)$  
 such  that  $d_M(p_i,p_j)>2r$  for  any  $1\leq  i<j\leq  k$.  
 The  space   of  all  the  possible  ordered    
    $k$-packings  of  closed  $r$-neighborhoods  in  $M$  
 is 
 \begin{eqnarray}\label{eq-26-02-27-9}
 {\rm  Conf}_k(M, (2r, +\infty]) 
 \end{eqnarray}
 and  the  collection  of     spaces  defined  by   (\ref{eq-26-02-27-9})    for  all  $k\geq  0$  gives  
 a    directed  simplicial  complex  on  $M$
 \begin{eqnarray*}
\overrightarrow  {\rm  Ind}(M,  (2r, +\infty])=\bigcup_{k\geq  0}  {\rm  Conf}_k(M, (2r, +\infty]). 
 \end{eqnarray*}

 An   {\it    unordered   $k$-packing  of  closed  $r$-neighborhoods}    in   $M$   is   
 a  set  $\{p_1,  p_2,  \ldots,  p_k\}$  of  $k$ distinct  points  in  $M$  
 such  that   the  sets   $\overline{N(r,p_i,M)}$  for    $1\leq  i\leq  k$   are  mutually  non-intersecting.   
 Equivalently,   it  is   a  $k$-hyperedge 
   $\sigma(M)=\{p_1,  p_2,  \ldots,  p_k\}$     
 such  that  $d_M(p_i,p_j)>2r$  for  any  $1\leq  i<j\leq  k$.  
 The  space   of  all  the  possible  unordered    
    $k$-packings  of  closed  $r$-neighborhoods  in  $M$  
 is 
 \begin{eqnarray}\label{eq-26-02-27-10}
 {\rm  Conf}_k(M, (2r, +\infty])/\Sigma_k  
 \end{eqnarray}
  and  the  collection  of   spaces  defined  by  (\ref{eq-26-02-27-10})    for  all  $k\geq  0$  gives  
 a     simplicial  complex  on  $M$
 \begin{eqnarray*}
  {\rm  Ind}(M,  (2r, +\infty])=\bigcup_{k\geq  0}  {\rm  Conf}_k(M, (2r, +\infty])/\Sigma_k. 
 \end{eqnarray*}

 For  any  $0<r_1<r_2$,  
 the  following     diagram  commutes such  that  each  square  is   a  pull-back  of  $k!$-sheeted   covering  maps 
 \begin{eqnarray*}
 \xymatrix{
  {\rm  Conf}_k(M, (2r_2, +\infty]) \ar[r]  \ar[d]_-{\pi}
  &  {\rm  Conf}_k(M, (2r_1, +\infty]) \ar[r]  \ar[d]_-{\pi}
  &  {\rm  Conf}_k(M) \ar[d]^-{\pi}\\
   {\rm  Conf}_k(M, (2r_2, +\infty])/\Sigma_k \ar[r]  
  &  {\rm  Conf}_k(M, (2r_1, +\infty])/\Sigma_k \ar[r]  
  &  {\rm  Conf}_k(M)/\Sigma_k. 
 }
 \end{eqnarray*}
 Thus  
  \begin{eqnarray}\label{eq-26.4.7.1}
\overrightarrow  {\rm  Ind}(M,  (-, +\infty])=\{\overrightarrow  {\rm  Ind}(M,  (2r, +\infty]) \mid  r>0\}
 \end{eqnarray}
gives  a  filtration  of  $\overrightarrow  {\rm  Ind}(M)$ 
and  
  \begin{eqnarray}\label{eq-26.4.7.2}
  {\rm  Ind}(M,  (-, +\infty])=\{  {\rm  Ind}(M,  (2r, +\infty]) \mid  r>0\}
 \end{eqnarray}
gives  a  filtration  of  $  {\rm  Ind}(M)$  such  that  the  projection  $\pi$  
preserves  the  filtrations 
\begin{eqnarray*}
\pi(\overrightarrow  {\rm  Ind}(M,  (-, +\infty]))= {\rm  Ind}(M,  (-, +\infty]). 
\end{eqnarray*}
We  say that  (\ref{eq-26.4.7.1})  is   the  {\it  persistent   directed
 independence  complex}  of  
   ordered  packings  of  closed  neighborhoods    on  $M$
and   (\ref{eq-26.4.7.2})  is   the  {\it  persistent  independence  complex}  of  
  unordered  packings  of  closed  neighborhoods   on  $M$.  

\begin{corollary}\label{co-26-4-7-1}
Let  $M$  and  $M'$  be  disjoint  Riemannian  manifolds.  
Then  the  persistent   directed
 independence  complex   of  
   ordered  packings  of  closed  neighborhoods    on  $M\sqcup  M'$ 
   is 
the  free  product  
 \begin{eqnarray}\label{eq-26-apr-03-5}
 \overrightarrow  {\rm  Ind}(M\sqcup  M',  (-, +\infty])=
 \mathcal{F}(\overrightarrow  {\rm  Ind}(M,  (-, +\infty]),  
 \overrightarrow  {\rm  Ind}(M',  (-, +\infty]))
 \end{eqnarray} 
and  the  persistent  
 independence  complex   of  
   unordered  packings  of  closed  neighborhoods    on  $M\sqcup  M'$ 
   is 
the  join  
 \begin{eqnarray}\label{eq-26-apr-03-3}
{\rm  Ind}(M\sqcup  M',  (-, +\infty])=
   {\rm  Ind}(M,  (-, +\infty])  *  
   {\rm  Ind}(M',  (-, +\infty])    
 \end{eqnarray} 
 such  that  $\pi$  preserves  the  filtrations  and  
 maps  (\ref{eq-26-apr-03-5})  to  (\ref{eq-26-apr-03-3}).  
\end{corollary}

\begin{proof}
The  proof   follows from  
(\ref{eq-26-march-02-1})  and  (\ref{eq-26-march-02-2}).   
\end{proof}

\subsection{Packings      with  hypergraph  constraints}
\label{ss-26.7.27-2}

Let  $\vec{\mathcal{H}}(M)$  be  a  hyperdigraph  on  $M$
 with  its  underlying    hypergraph  $\mathcal{H}(M)$.   
  For  any  subset  $\mathbb{I}\subseteq  [0,+\infty]$, 
 let  
  \begin{eqnarray*}
  \vec{\mathcal{H}}(M, \mathbb{I})=\vec{\mathcal{H}}(M)\cap  \overrightarrow{\rm  Ind}(M, \mathbb{I})   
 \end{eqnarray*}
 with  
 its   underlying  
  hypergraph  
  \begin{eqnarray*}
 {\mathcal{H}}(M, \mathbb{I})
 =   {\mathcal{H}}(M)\cap  {\rm  Ind}(M, \mathbb{I}).  
 \end{eqnarray*}  
Similar  to  \cite[Subsection~2.1,  (ii) and  (iii)]{jktr2023},  
 \begin{eqnarray*}
 \Delta   \vec{\mathcal{H}}(M, \mathbb{I})&\subseteq  & 
 (\Delta  \vec{\mathcal{H}}(M))\cap  \overrightarrow{\rm  Ind}(M, \mathbb{I}),  \\
  \delta   \vec{\mathcal{H}}(M, \mathbb{I})&=  & 
 (\delta  \vec{\mathcal{H}}(M))\cap  \overrightarrow{\rm  Ind}(M, \mathbb{I}),\\
 \Delta   {\mathcal{H}}(M, \mathbb{I})&\subseteq  & 
 (\Delta   {\mathcal{H}}(M))\cap   {\rm  Ind}(M, \mathbb{I}),  \\
  \delta   {\mathcal{H}}(M, \mathbb{I})&=  & 
 (\delta   {\mathcal{H}}(M))\cap  {\rm  Ind}(M, \mathbb{I}). 
  \end{eqnarray*}
For  any  $r>0$,  let  
 $\mathbb{I}=(2r,+\infty]$.  
 The   
   hyperdigraph $\vec{{\mathcal{H}}}(M, (2r,+\infty])$ 
   is  
 the   space  of  all  the  
 possible  ordered  packings  of  closed  $r$-neighborhoods  in  $M$ 
 with  the  constraint  that  each  ordered  packing  is  given  by
  the  sequence  of   vertices   of  a   certain  directed  hyperedge 
 in  $\vec{\mathcal{H}}(M)$. 
 The    hypergraph  ${{\mathcal{H}}}(M, (2r,+\infty])$
    is  
 the   space  of  the  possible  unordered  packings 
  of  closed  $r$-neighborhoods  in  $M$  with  the     constraint
  that  each  unordered  packing  is  given  by
  the  set  of   vertices   of  a  certain     hyperedge 
 in  ${\mathcal{H}}(M)$. 
 Let  $r$  run  over  all  positive  numbers.  
 We  obtain  a  commutative  diagram   
 \begin{eqnarray*}
 \xymatrix{
 \delta\vec{{\mathcal{H}}}(M, (-,+\infty])\ar[r]\ar[d]_-{\pi}
 &\vec{{\mathcal{H}}}(M, (-,+\infty])\ar[r]\ar[d]_-{\pi}
 &\Delta\vec{{\mathcal{H}}}(M, (-,+\infty])\ar[r]\ar[d]^-{\pi}
 &\overrightarrow{\rm  Ind}(M,  (-,+\infty])\ar[d]^-{\pi}\\
 \delta{{\mathcal{H}}}(M, (-,+\infty])\ar[r]
 &{{\mathcal{H}}}(M, (-,+\infty])\ar[r]
 &\Delta{{\mathcal{H}}}(M, (-,+\infty])\ar[r]
 &{\rm  Ind}(M,  (-,+\infty]) 
 }
 \end{eqnarray*}
 where  each  term  is  a  filtration  of    hyper(di)graphs  on  $M$,
 the   horizontal  maps  are  canonical  inclusions  and  
 the  vertical  maps  are  canonical  projections.  
   
\begin{corollary}
Let  $M$  and  $M'$  be  disjoint  Riemannian  manifolds.  
Then  the  persistent   hyperdigraph   of  
   ordered  packings  of  closed  neighborhoods    on  $M\sqcup  M'$ 
   is 
the  free  product  
 \begin{eqnarray}\label{eq-26-apr-04-5}
 \mathcal{F}(\vec{{\mathcal{H}}}(M, (-,+\infty]),  
\vec{{\mathcal{H}}}(M', (-,+\infty]))
 \end{eqnarray} 
and  the  persistent  
 hypergraph   of  
   unordered  packings  of  closed  neighborhoods    on  $M\sqcup  M'$ 
   is 
the  join  
 \begin{eqnarray}\label{eq-26-apr-04-3}
 {{\mathcal{H}}}(M, (-,+\infty])*  
   {{\mathcal{H}}}(M', (-,+\infty])    
 \end{eqnarray} 
 such  that  the  diagrams  (\ref{26-3-3-diag3})  
 and  (\ref{26-3-3-diag5})  are  satisfied  for  
 $\vec{{\mathcal{H}}}(M, (-,+\infty])$  and    
  $\vec{{\mathcal{H}}}(M', (-,+\infty])$.    
\end{corollary}

\begin{proof}
The  proof  follows  from  Proposition~\ref{le-1.5}
and  
Corollary~\ref{co-26-4-7-1}.  
\end{proof}

\subsection{Coverings  of  Riemannian  manifolds  by  open $r$-neighborhoods}
\label{ss-4.4-cover}
 
An   {\it    ordered   $k$-covering  of  open   $r$-neighborhoods}    in   $M$   is   
 an  ordered  sequence  $(p_1,  p_2,  \ldots,  p_k)$  of  $k$ distinct points  in  $M$  
 such  that  ${N(r,p_i,M)}$,  $1\leq  i\leq  k$,  form  an  open  cover  of  $M$.  
 Equivalently,  it   
 is   a  directed  $k$-hyperedge  $\vec{\sigma}(M)=(p_1,  p_2,  \ldots,  p_k)$  
 such  that  for  any  $q\in  M$,  
 $d_M(p_i,q)<r$  for  some   $1\leq  i \leq  k$.  
 The  space  of   all  the  possible  ordered    
    $k$-coverings  of  open  $r$-neighborhoods  in  $M$  is   
\begin{eqnarray}\label{eq-26feb-27-17}
{\rm  Cover}_k(M, r)&=&\{(p_1,  p_2,  \ldots,  p_k)\in  {\rm Conf}_k(M)
\mid   
{\rm ~for~any~}q\in M, \nonumber \\
&&{\rm~there ~exists~} 1\leq  i\leq  k{\rm~such ~that~}
d_M(p_i,q)<r\}. 
\end{eqnarray}
    The  collection  of  spaces  defined  by  (\ref{eq-26feb-27-17})    for  all  $k\geq  0$  gives  
 an      independence  hyperdigraph  (recall  that  an 
 independence  hyperdigraph    is  a  hyperdigraph  that  is  closed  under  adding  vertices  to  its  directed  
 hyperedges)
 \begin{eqnarray}\label{eq-26feb-27-18}
\overrightarrow  {\rm  Cv}(M, r)=\bigcup_{k\geq  0}  {\rm  Cover}_k(M, r).   
 \end{eqnarray}
The  space  of  all  the  possible  ordered  $k$-coverings  of  open  neighborhoods 
in  $M$  is  
   \begin{eqnarray*} 
{\rm  Cover}_k(M) = \bigcup_{r>0}  {\rm  Cover}_k(M,r)  
\end{eqnarray*}
and  the  space  of  all  the  possible  ordered   finite  coverings  of  open  neighborhoods 
in  $M$  is 
 an       independence  hyperdigraph   
  \begin{eqnarray}\label{eq-26feb-27-20}
 \overrightarrow {\rm  Cv}(M)=\bigcup_{r>0}   \overrightarrow {\rm  Cv}(M,r)=\bigcup_{k\geq  0}   
  {\rm  Cover}_k(M).     
 \end{eqnarray}

 An   {\it    unordered   $k$-covering  of  open   $r$-neighborhoods}    in   $M$   is   
 a  set  $\{p_1,  p_2,  \ldots,  p_k\}$  of  $k$ distinct points  in  $M$  
 such  that  ${N(r,p_i,M)}$,  $1\leq  i\leq  k$,  form   an  open  cover  of  $M$.   
 Equivalently,   it  is   a  $k$-hyperedge 
   $\sigma(M)=\{p_1,  p_2,  \ldots,  p_k\}$     
 such  that   for  any  $q\in  M$,  
 $d_M(p_i,q)<r$  for  some   $1\leq  i \leq  k$.  
   The  space  of   all  the  possible  unordered    
    $k$-coverings  of  open  $r$-neighborhoods  in  $M$
    is  the  orbit  space 
    \begin{eqnarray}\label{eq-26feb-27-15}
    {\rm  Cover}_k(M, r)/\Sigma_k. 
    \end{eqnarray}
 The  collection  of    (\ref{eq-26feb-27-15})    for  all  $k\geq  0$  gives  
 an     independence  hypergraph  
 (recall  that  an 
 independence  hypergraph    is  a  hypergraph  that  is  closed  under  adding  vertices  to  its   
 hyperedges)
 \begin{eqnarray}\label{eq-26feb-27-19}
  {\rm  Cv}(M, r)=\bigcup_{k\geq  0}  {\rm  Cover}_k(M, r)  /\Sigma_k   
 \end{eqnarray}
 such  that  (\ref{eq-26feb-27-19})  is  the underlying  
    independence  hypergraph    of   (\ref{eq-26feb-27-18}).   
 The  space  of  all  the  possible  unordered   finite  coverings  of  open  neighborhoods 
in  $M$  is 
 an      independence  hypergraph   
  \begin{eqnarray}\label{eq-26feb-27-21}
  {\rm  Cv}(M)=\bigcup_{r>0}  {\rm  Cv}(M,r) =\bigcup_{k\geq  0}  
  {\rm  Cover}_k(M)/\Sigma_k.      
 \end{eqnarray}
Then  the  persistent       independence  hyperdigraph   
   \begin{eqnarray*}
\overrightarrow  {\rm  Cv}(M,  -)=\{\overrightarrow {\rm  Cv}(M, r) \mid  r>0\}
 \end{eqnarray*}
gives  a  filtration  of (\ref{eq-26feb-27-20}) 
and   the   persistent      independence  hypergraph   
  \begin{eqnarray*}
  {\rm  Cv}(M,  -)=\{  {\rm  Cv}(M, r) \mid  r>0\}
 \end{eqnarray*}
gives  a  filtration  of (\ref{eq-26feb-27-21})   such  that  
  $\pi$  
preserves  the  filtrations 
\begin{eqnarray*}
\pi(\overrightarrow  {\rm  Cv}(M,  -))= {\rm  Cv}(M,  -). 
\end{eqnarray*}
 
\begin{corollary}\label{co-26-4-7-5}
Let  $M$  and  $M'$  be  disjoint  Riemannian  manifolds.  
Then  the  persistent       independence  hyperdigraph   of  
   ordered  coverings  of  open  neighborhoods    on  $M\sqcup  M'$ 
   is 
the  free  product  
 \begin{eqnarray}\label{eq-26-apr-05-5}
 \overrightarrow  {\rm  Cv}(M\sqcup  M',  -)=
 \mathcal{F}(\overrightarrow  {\rm  Cv}(M,  -),  
 \overrightarrow  {\rm  Cv}(M',  -))
 \end{eqnarray} 
and  the  persistent  
 independence  hypergraph   of  
   unordered  coverings  of  open  neighborhoods    on  $M\sqcup  M'$ 
   is 
the  join  
 \begin{eqnarray}\label{eq-26-apr-05-3}
{\rm  Cv}(M\sqcup  M', -)=
   {\rm  Cv}(M, -)  *  
   {\rm  Cv}(M',  -)    
 \end{eqnarray}  
 such  that  $\pi$  preserves  the  filtrations  
  and  
 sends  (\ref{eq-26-apr-05-5})  to  (\ref{eq-26-apr-05-3}). 
  \end{corollary}

\begin{proof}
The  proof    is  an  analogue  of  
(\ref{eq-26-march-02-1})  and  (\ref{eq-26-march-02-2}).   
\end{proof}

\subsection{Coverings    with  hypergraph  constraints}
\label{ss-26.7.27.3}
   
The  space  of  ordered 
coverings  of   $M$  by open  $r$-neighborhoods  with  the  constraint  that  
each  ordered  covering   is  given  by
  the  sequence  of   vertices  of    a   certain  directed  hyperedge 
 in  $\vec{\mathcal{H}}(M)$   is  the    hyperdigraph 
\begin{eqnarray*}
\vec{\mathcal{H}}(M)\cap  \overrightarrow  {\rm  Cv}(M, r).  
\end{eqnarray*}
The  space  of  unordered 
coverings  of   $M$  by open  $r$-neighborhoods  with  the  constraint  that  
each  unordered  covering   is  given  by
  the  set  of   vertices  of   a  certain     hyperedge 
 in  $\mathcal{H}(M)$   is  the     hypergraph 
\begin{eqnarray*}
{\mathcal{H}}(M)\cap     {\rm  Cv}(M, r).  
\end{eqnarray*}
Similar  to  \cite[Subsection~2.1,  (iv) and  (v)]{jktr2023},  
 \begin{eqnarray*}
 \bar\Delta   (\vec{\mathcal{H}}(M)\cap  \overrightarrow  {\rm  Cv}(M, r))
 &\subseteq  & 
 (\bar\Delta  \vec{\mathcal{H}}(M))\cap  \overrightarrow  {\rm  Cv}(M, r),  \\
  \bar\delta   (\vec{\mathcal{H}}(M)\cap  \overrightarrow  {\rm  Cv}(M, r))&=  & 
 (\bar\delta  \vec{\mathcal{H}}(M))\cap   \overrightarrow  {\rm  Cv}(M, r)
  \end{eqnarray*}
 and    
   \begin{eqnarray*}
 \bar\Delta   ({\mathcal{H}}(M)\cap  {\rm  Cv}(M, r))
 &\subseteq  & 
 (\bar\Delta  {\mathcal{H}}(M))\cap   {\rm  Cv}(M, r),  \\
  \bar\delta   ({\mathcal{H}}(M)\cap    {\rm  Cv}(M, r))&=  & 
 (\bar\delta {\mathcal{H}}(M))\cap   {\rm  Cv}(M, r).  
  \end{eqnarray*}
Let  $r$  run  over  all  positive  numbers.  
 We  obtain  a  commutative  diagram   
 \begin{eqnarray*}
 \xymatrix{
 \bar\delta(\vec{\mathcal{H}}(M)\cap  \overrightarrow  {\rm  Cv}(M, -))\ar[r]\ar[d]_-{\pi}
 &\vec{\mathcal{H}}(M)\cap  \overrightarrow  {\rm  Cv}(M, -)\ar[r]\ar[d]_-{\pi}
 & \bar\Delta(\vec{\mathcal{H}}(M)\cap  \overrightarrow  {\rm  Cv}(M, -))\ar[r]\ar[d]^-{\pi}
 & \overrightarrow  {\rm  Cv}(M, -)\ar[d]^-{\pi}\\
 \bar\delta({\mathcal{H}}(M)\cap  {\rm  Cv}(M, -))\ar[r] 
 &{\mathcal{H}}(M)\cap  {\rm  Cv}(M, -)\ar[r] 
 & \bar\Delta({\mathcal{H}}(M)\cap   {\rm  Cv}(M, -))\ar[r] 
 &   {\rm  Cv}(M, -)
 }
 \end{eqnarray*}
 where  each  term  is  a  filtration  of     hyper(di)graphs  on  $M$,
 the   horizontal  maps  are  canonical  inclusions  and  
 the  vertical  maps  are  canonical  projections.

\begin{corollary}
Let  $M$  and  $M'$  be  disjoint  Riemannian  manifolds.  
Then  the  persistent   hyperdigraph   of  
   ordered  coverings  of  open neighborhoods    on  $M\sqcup  M'$ 
   is 
the  free  product  
 \begin{eqnarray}\label{eq-26-apr-07-5}
 \mathcal{F}(\vec{\mathcal{H}}(M)\cap  \overrightarrow  {\rm  Cv}(M, -),  
\vec{\mathcal{H}}(M')\cap  \overrightarrow  {\rm  Cv}(M', -))
 \end{eqnarray} 
and  the  persistent  
 hypergraph   of  
   unordered  coverings  of  open  neighborhoods    on  $M\sqcup  M'$ 
   is 
the  join  
 \begin{eqnarray}\label{eq-26-apr-07-3}
 ({\mathcal{H}}(M)\cap {\rm  Cv}(M, -)   )
 * ({\mathcal{H}}(M')\cap   {\rm  Cv}(M', -))
 \end{eqnarray} 
 such  that  the  diagrams  (\ref{26-3-3-diag3})  
 and  (\ref{26-3-3-diag5})  are  satisfied  for  
 $\vec{\mathcal{H}}(M)\cap  \overrightarrow  {\rm  Cv}(M, -)$  and    
  $\vec{\mathcal{H}}(M')\cap  \overrightarrow  {\rm  Cv}(M', -)$.    
\end{corollary}

\begin{proof}
The  proof  follows  from  Proposition~\ref{le-1.5}
and  
Corollary~\ref{co-26-4-7-5}.  
\end{proof}

\section{Morphisms  of  hypergraphs  induced  by   maps  between  manifolds}
\label{s-27.7.28.5}

In  this  section, 
we  study  morphisms  between  hypergraphs and  independence  complexes 
induced  by  maps  between  manifolds.  
In  Subsection~\ref{ss6.1}, 
 we  construct  certain  simplicial  embeddings  between  
the  independence  complexes  from  embeddings  of  manifolds.    
 In  Subsection~\ref{ss6.2},  we   construct  certain  morphisms  of  the 
independence  hypergraphs  of  the  coverings  
from  surjections   of  manifolds.  
In  Subsection~\ref{26.6.3-ss7.3}, 
we  discuss   the   morphisms  of  hyper(di)graphs  induced  by 
embeddings  of  manifolds.  We   give  some  
commutative  diagrams  of  hyper(di)graphs  which  are  
functorial  with respect to 
the  canonical  projections.   
 This  section  provides  preliminary  groundwork 
   for  the regular  maps  in  Section~\ref{s-27.7.28.8}.

\subsection{Simplicial  maps   induced  by   embeddings  of   manifolds}\label{ss6.1}

Let  $M$  and  $M'$  be  differentiable  manifolds.  
Let  ${\rm  Emb}(M,M')$   be   the  space  of  all  
differentiable  embeddings  of   $M$  into  $M'$.  
Let 
\begin{eqnarray*}
{\rm  Inj}(\overrightarrow{\rm   Ind}(M),\overrightarrow{\rm   Ind}(M')),
~~~~~~{\rm   Inj}({\rm   Ind}(M),{\rm   Ind}(M')) 
\end{eqnarray*}  
be   the  space  of  all  
injective  morphisms  
from   $\overrightarrow{\rm   Ind}(M)$  into  
$\overrightarrow{\rm   Ind}(M')$ and   the  space  of  all  differentiable 
simplicial  embeddings  of   ${\rm   Ind}(M)$  into  ${\rm   Ind}(M')$  respectively.
 The  next   proposition   follows.  
 
\begin{proposition}\label{pr-4.15-1}
For  any  differentiable  manifolds $M$  and  $M'$,  
 we  have  a  commutative  diagram 
\begin{eqnarray}\label{diag-260415-1}
\xymatrix{
&{\rm  Inj}(\overrightarrow{\rm   Ind}(M),\overrightarrow{\rm   Ind}(M')) \ar[dd]^-{\pi}\\
{\rm  Emb}(M,M') \ar[ru]^-{\overrightarrow{\rm   Ind}(-)}  \ar[rd]_-{{\rm   Ind}(-)} &\\
& {\rm   Inj}({\rm   Ind}(M),{\rm   Ind}(M')). 
} 
\end{eqnarray}
\end{proposition}

\begin{proof} 
Let  $f:  M\longrightarrow  M'$  be  a  differentiable  embedding.  
Then  $f$  induces  a  differentiable  map 
 \begin{eqnarray}\label{eq-26-3-11-c1}
 {\rm   Conf}_k(f): {\rm   Conf}_k(M)\longrightarrow{\rm   Conf}_k(M')
\end{eqnarray}
sending  an  ordered   $k$-tuple    $(x_1,\ldots,x_k)$  of   distinct  points
in  $M$ 
to  an  ordered   $k$-tuple  $(f(x_1),\ldots,f(x_k))$  of   distinct  points 
in  $M'$. 
Since  (\ref{eq-26-3-11-c1})  is  $\Sigma_k$-equivariant,  
we  have  an  induced   differentiable  map  
 \begin{eqnarray}\label{eq-26-3-11-c2}
 {\rm   Conf}_k(f)/\Sigma_k: {\rm   Conf}_k(M)/\Sigma_k
 \longrightarrow{\rm   Conf}_k(M')/\Sigma_k
\end{eqnarray}
sending  an  unordered   $k$-tuple    $\{x_1,\ldots,x_k\}$  of   distinct  points
in  $M$ 
to  an  unordered   $k$-tuple  $\{f(x_1),\ldots,f(x_k)\}$  of   distinct  points 
in  $M'$. 
Let  $k$  run  over  nonnegative  integers.
The  map  in  (\ref{eq-26-3-11-c1})
 yields         a   morphism  of  directed  simplicial  complexes 
\begin{eqnarray*}
\overrightarrow{\rm   Ind}(f): 
 \overrightarrow{\rm   Ind}(M)\longrightarrow 
  \overrightarrow{\rm   Ind}(M')
  \end{eqnarray*}
    given  by  
\begin{eqnarray*}
\overrightarrow{\rm   Ind}(f)(\vec{\sigma}(M))=(f(x_1),\ldots, f(x_k))
\end{eqnarray*}
   for  any  directed  simplex  $\vec{\sigma}(M)=(x_1,\ldots,x_k)$  of  
$\overrightarrow{\rm   Ind}(M)$;      
and  the  map  in   (\ref{eq-26-3-11-c2})  
 yields  a  
    simplicial  map  
\begin{eqnarray*}
{\rm   Ind}(f): {\rm   Ind}(M)\longrightarrow
{\rm   Ind}(M')
\end{eqnarray*} 
given  by 
\begin{eqnarray*}
{\rm   Ind}(f)(\vec{\sigma}(M))=\{f(x_1),\ldots, f(x_k)\} 
\end{eqnarray*}
for  any  simplex  ${\sigma}(M)=\{x_1,\ldots,x_k\}$  of  
${\rm   Ind}(M)$.  
 The  following  diagram  commutes 
\begin{eqnarray*}
\xymatrix{
 \overrightarrow{\rm   Ind}(M)\ar[r]^-{\overrightarrow{\rm   Ind}(f)} \ar[d]_-{\pi}
 &\overrightarrow{\rm   Ind}(M')\ar[d]^-{\pi}\\
 {\rm   Ind}(M)\ar[r]^-{ {\rm   Ind}(f)} 
 & {\rm   Ind}(M'). 
} 
\end{eqnarray*}
Consequently,  
\begin{eqnarray*}
  \overrightarrow{\rm   Ind}(f)\in  
  {\rm  Inj}(\overrightarrow{\rm   Ind}(M),\overrightarrow{\rm   Ind}(M')),  
 ~~~~~~ {\rm   Ind}(f)\in
{\rm  Inj}({\rm   Ind}(M),{\rm   Ind}(M'))
\end{eqnarray*}
such  that  
$\pi  (\overrightarrow{\rm   Ind}(f))=  {\rm   Ind}(f)$,
 where  $\pi$  on  the  left-hand  side  denotes  the   induced   map  on  
 the  morphism  spaces.  
Therefore,    (\ref{diag-260415-1})  commutes.  
\end{proof}
Let  ${\rm   Aut}(M)$   be   the  space  of  all  
the  self-diffeomorphisms  of   $M$.  
Let  ${\rm  Aut}(\overrightarrow{\rm   Ind}(M))$  
be   the  space  of  all  
the  differentiable  isomorphisms   of  $\overrightarrow{\rm   Ind}(M)$.    
Let  ${\rm  Aut}( {\rm   Ind}(M))$  
be   the  space  of  all  
the  differentiable  isomorphisms   of  $ {\rm   Ind}(M)$.

\begin{corollary}
\label{co-26.4.15-1}
For  any  differentiable  manifold  $M$,  
 we  have  a  commutative  diagram 
\begin{eqnarray}\label{diag-260415-2}
\xymatrix{
&{\rm  Aut}(\overrightarrow{\rm   Ind}(M))\ar[dd]^-{\pi}\\
{\rm   Aut}(M)\ar[ru]^-{\overrightarrow{\rm   Ind}(-)}  \ar[rd]_-{{\rm   Ind}(-)} &\\
&{\rm  Aut}(\overrightarrow{\rm   Ind}(M)). 
} 
\end{eqnarray}
\end{corollary}

\begin{proof}
Let  $M=M'$  in  Proposition~\ref{pr-4.15-1}. 
The  corollary  follows.  
\end{proof}

Suppose  in  addition  that  $M$  and  $M'$  have     Riemannian  metrics  $g$.   
The  induced  distance  functions  on  $M$  and  $M'$  are  
$d_M$  and  $d_{M'}$  respectively.  
Let  ${\rm  IsomE}(M, M')$  be  the  space  of  
 all  
differentiable  embeddings   $f: M\longrightarrow  M'$   such  that 
$f$  preserves  the  Riemannian  metric  $g$,   
i.e.  
$f^* (g_{M'}) =  g_M$.  
Note  that  
 the   maps  in   ${\rm  IsomE}(M, M')$  are  $1$-Lipschitz  but  may  not    preserve 
 the  extended  distances  globally.
  For  any  $r>0$,  
  Let 
\begin{eqnarray*}
{\rm  Inj}(\overrightarrow{\rm   Ind}(M, (0, 2r]),\overrightarrow{\rm   Ind}(M', (0,2r])),  
~~~~~~{\rm   Inj}({\rm   Ind}(M,(0, 2r]),{\rm   Ind}(M',(0, 2r]))
\end{eqnarray*}  
be   the  space  of  all  
injective  morphisms  
from   $\overrightarrow{\rm   Ind}(M, (0, 2r])$  into  
$\overrightarrow{\rm   Ind}(M',(0, 2r])$  and    
    the  space  of  all  differentiable 
simplicial  embeddings  of   ${\rm   Ind}(M,(0, 2r])$  into  ${\rm   Ind}(M',(0, 2r])$
respectively. 

\begin{proposition}\label{pr-4.16-1}
For  any  Riemannian   manifolds $M$  and  $M'$,  
 we  have  a  family  of   commutative  diagrams 
\begin{eqnarray}\label{diag-260416-1}
\xymatrix{
&{\rm  Inj}(\overrightarrow{\rm   Ind}(M, (0, 2r]),\overrightarrow{\rm   Ind}(M', (0, 2r])) \ar[dd]^-{\pi}\\
{\rm  IsomE}(M,M') \ar[ru]^-{\overrightarrow{\rm   Ind}(-)}  \ar[rd]_-{{\rm   Ind}(-)} &\\
& {\rm   Inj}({\rm   Ind}(M, (0, 2r]),{\rm   Ind}(M', (0, 2r]))  
} 
\end{eqnarray}
for   any  $r>0$,  where  the  vertical map  $\pi$  in (\ref{diag-260416-1})
   gives  a  filtration  of   the  vertical map  $\pi$  in  (\ref{diag-260415-1}).  
\end{proposition}

\begin{proof}
Let  $f\in  {\rm  IsomE}(M,M')$.  
Then    
\begin{eqnarray}\label{eq-24.4.16-ineq1}
d_M(x,y)\geq   d_{M'}(f(x), f(y))
\end{eqnarray}
  for  any  $x,y\in  M$. 
Hence  
$f$  induces  a  differentiable  embedding  
 \begin{eqnarray}\label{eq-26-4-16-c1}
 {\rm   Conf}_k(f): {\rm   Conf}_k(M, (0, 2r])\longrightarrow{\rm   Conf}_k(M',(0, 2r])
\end{eqnarray}
preserving  the  Riemannian  metric     on  the  
$k$-th  ordered  configuration  spaces,  i.e.  
\begin{eqnarray*}
 {\rm   Conf}_k(f)^*(g_{M'}^{\oplus  k}) = g_{M}^{\oplus  k}.  
\end{eqnarray*}
Let  $k$  run  over  all  nonnegative  integers.  
Then  (\ref{eq-26-4-16-c1})  induces   a   morphism  of  directed  simplicial  complexes 
\begin{eqnarray*}
\overrightarrow{\rm   Ind}(f): 
 \overrightarrow{\rm   Ind}(M, (0, 2r])\longrightarrow 
  \overrightarrow{\rm   Ind}(M', (0, 2r]).  
  \end{eqnarray*}
The  $\Sigma_k$-actions   on  ${\rm   Conf}_k(M, (0, 2r])$  
and  ${\rm   Conf}_k(M',(0, 2r])$  preserve  the  Riemannian  metric     as  well,  i.e.  for  any  $s\in \Sigma_k$,  
\begin{eqnarray*}
s^*(g_{M}^{\oplus  k}) = g_{M}^{\oplus  k}, 
~~~~~~s^*(g_{M'}^{\oplus  k}) = g_{M'}^{\oplus  k}.   
\end{eqnarray*}
Thus  (\ref{eq-26-4-16-c1})  induces   a  differentiable  embedding 
 \begin{eqnarray}\label{eq-26-4-16-c2}
 {\rm   Conf}_k(f)/\Sigma_k: {\rm   Conf}_k(M, (0, 2r])/\Sigma_k\longrightarrow{\rm   Conf}_k(M',(0, 2r])/\Sigma_k
\end{eqnarray}
preserving  the  Riemannian  metric   on  the  
$k$-th  unordered  configuration  spaces,  i.e.  
\begin{eqnarray*}
( {\rm   Conf}_k(f)/\Sigma_k)^*(g_{M'}^{\oplus  k}) = g_{M}^{\oplus  k}.  
\end{eqnarray*}
Let  $k$  run  over  all  nonnegative  integers.  
Then  (\ref{eq-26-4-16-c2})  induces  a  simplicial  map  
\begin{eqnarray*}
 {\rm   Ind}(f): 
 {\rm   Ind}(M, (0, 2r])\longrightarrow 
 {\rm   Ind}(M', (0, 2r])   
  \end{eqnarray*}
such  that  the  diagram  commutes 
\begin{eqnarray*}
\xymatrix{
 \overrightarrow{\rm   Ind}(M, (0, 2r])\ar[r]^-{  \overrightarrow{\rm   Ind}(f)}\ar[d]_-{\pi}
 & \overrightarrow{\rm   Ind}(M', (0, 2r])\ar[d]^-{\pi}\\
  {\rm   Ind}(M, (0, 2r])\ar[r]^-{{\rm   Ind}(f)}
 & {\rm   Ind}(M', (0, 2r]). 
}
\end{eqnarray*}
We  obtain  the  commutative  diagram  (\ref{diag-260416-1}).  
\end{proof}

Let  ${\rm   Isom}(M)$   be   the  space  of  all  
the  self-diffeomorphisms  of   $M$  preserving  $g$.  
Let  ${\rm  Aut}(\overrightarrow{\rm   Ind}(M,(0,2r]))$  
be   the  space  of  all  
the  differentiable  isomorphisms   of  $\overrightarrow{\rm   Ind}(M)$.       
Let  ${\rm  Aut}( {\rm   Ind}(M,(0,2r]))$  
be   the  space  of  all  
the  differentiable  isomorphisms   of  $ {\rm   Ind}(M,(0,2r])$.

\begin{corollary}
\label{co-26.4.16-1}
For  any  Riemannian     manifold  $M$,  
 we  have  a  commutative  diagram 
\begin{eqnarray}\label{diag-260416-2}
\xymatrix{
&{\rm  Aut}(\overrightarrow{\rm   Ind}(M,,(0,2r]))\ar[dd]^-{\pi}\\
{\rm   Isom}(M)\ar[ru]^-{\overrightarrow{\rm   Ind}(-)}  \ar[rd]_-{{\rm   Ind}(-)} &\\
&{\rm  Aut}(\overrightarrow{\rm   Ind}(M,,(0,2r]))  
} 
\end{eqnarray}
for   any  $r>0$,  where  the  vertical map  $\pi$  in (\ref{diag-260416-2})
   gives  a  filtration  of   the  vertical map  $\pi$  in  (\ref{diag-260415-2}).  
\end{corollary}

\begin{proof}
Let  $M=M'$  in  Proposition~\ref{pr-4.16-1}. 
The  corollary  follows.  
\end{proof}

\subsection{Morphisms  of  independence  hypergraphs   induced  by   
surjections   of   manifolds} \label{ss6.2}

Let    $M$  and  $M'$  be  manifolds  with      Riemannian  metrics  $g$.    
Let  ${\rm  IsomS}(M, M')$   be   the  space  of  all  the 
differentiable  surjections  of   $M$  onto  $M'$  preserving  
 $g$.        
 Note  that  
 the   maps  in   ${\rm  IsomS}(M, M')$  are  $1$-Lipschitz  but  may  not    preserve 
 the  extended  distances  globally.
Let 
\begin{eqnarray*}
{\rm  Mor}({\rm   Cv}(M,r),{\rm   Cv}(M',r))  
\end{eqnarray*}  
be   the  space  of  all  
  morphisms  of  independence  hypergraphs  
from   ${\rm   Cv}(M,r)$   to  
${\rm   Cv}(M',r)$.  
The  next  proposition  follows.

\begin{proposition}\label{pr-4.17-cv1}
For  any  $r>0$  and  
any  Riemannian     manifolds $M$  and  $M'$,  
 we  have  a  map  
\begin{eqnarray}\label{diag-260417-cv1}
{\rm   Cv} (-):
{\rm  IsomS}(M,M') \longrightarrow  
{\rm  Mor}({\rm   Cv}(M,r),{\rm   Cv}(M',r)).    
\end{eqnarray}
\end{proposition}

\begin{proof}
 Let $f\in  {\rm  IsomS}(M,M')$. 
  Then   
 $f:  M\longrightarrow  M'$  is  a  differentiable  surjection  such  that
$f^* (g_{M'}) =  g_M$.       
Consequently,  (\ref{eq-24.4.16-ineq1})  is satisfied.  
  For  any  $r>0$,  
with  the  help  of  (\ref{eq-26feb-27-19})  and  (\ref{eq-24.4.16-ineq1}),
 $f$  induces  a  morphism  of  independence  hypergraphs 
 \begin{eqnarray}\label{eq-26-4-16-cv1}
 {\rm   Cv} (f): {\rm   Cv}(M,r)\longrightarrow{\rm   Cv}(M',r)
\end{eqnarray}
sending  an  unordered   $k$-tuple    $\{x_1,\ldots,x_k\}\in  {\rm   Cv}(M,r)$  of   distinct  points
in  $M$ 
to  an  unordered   $l$-tuple  $\{f(x_1),\ldots,f(x_l)\}\in {\rm   Cv}(M',r)$  of   distinct  points 
in  $M'$,  where   $l\leq  k$  and  
$f(x_1),\ldots,f(x_l)$  are  obtained  from  $f(x_1),\ldots,f(x_k)$  
by  removing  the  repeated  points.  
Consequently,  we  obtain  (\ref{diag-260417-cv1}). 
\end{proof}

Let  
${\rm  Aut}(  {\rm   Cv}(M,r))$  
be   the  space  of  all  
the  differentiable  isomorphisms   of   
the  independence  hypergraph $  {\rm   Cv}(M,r)$
on  $M$.

\begin{corollary}
\label{co-26.4.17-cv1}
For  any  $r>0$  and  any  Riemannian     manifold   $M$,  
 we  have  a  map  
\begin{eqnarray}\label{diag-260417-cv9}
{\rm   Cv} (-):
{\rm  Isom}(M) \longrightarrow  
{\rm  Aut}({\rm   Cv}(M,r)).    
\end{eqnarray}
\end{corollary}

\begin{proof}
Let  $M=M'$  in  Proposition~\ref{pr-4.17-cv1}.    
The  corollary  follows.  
\end{proof}

\subsection{Morphisms  of  hypergraphs  induced by  embeddings  of  manifolds}
\label{26.6.3-ss7.3}

Let   $f:  M\longrightarrow  M'$  be  an  embedding  of  manifolds.  
Let  $\vec{\mathcal{H}}(M)$  and   $\vec{\mathcal{H}}'(M')$  be    hyperdigraphs
  on  $M$  and  $M'$  respectively.  
  We  say  that  $f$  is  a   {\it  morphism}  of  hyperdigraphs from  
  $\vec{\mathcal{H}} (M)$  to  $\vec{\mathcal{H}}'(M')$ 
  and  write   $f:  \vec{\mathcal{H}}(M)\longrightarrow  \vec{\mathcal{H}}'(M')$  if  
    for  any  directed  hyperedge  $\vec\sigma(M)=(v_0,v_1,\ldots, v_k)$
  of  $\vec{\mathcal{H}}(M)$, 
  its  image  
  $f(\vec\sigma(M))=(f(v_0),f(v_1),\ldots, f(v_k))$
  is  a  directed  hyperedge   of  $\vec{\mathcal{H}}'(M')$.  
    Let  ${\mathcal{H}}(M)$  and  ${\mathcal{H}}'(M')$   be    hypergraphs
  on  $M$  and  $M'$  respectively.  
  We  say  that 
    $f$  is  a   {\it  morphism}  of  hypergraphs from  
  ${\mathcal{H}} (M)$  to  ${\mathcal{H}}'(M')$ 
  and  write   $f:  {\mathcal{H}}(M)\longrightarrow  {\mathcal{H}}'(M')$  if  
    for  any    hyperedge  $\sigma(M)=\{v_0,v_1,\ldots, v_k\}$
  of  ${\mathcal{H}}(M)$, 
  its  image  
  $f(\sigma(M))=\{f(v_0),f(v_1),\ldots, f(v_k)\}$
  is  a   hyperedge   of  ${\mathcal{H}}'(M')$.

  \begin{lemma}
  \label{pr-26.6.1.1}
  Let  $f:  \vec{\mathcal{H}}(M)\longrightarrow  \vec{\mathcal{H}}'(M')$  
  be  a  morphism  of  hyperdigraphs  given  by  an  embedding  
  $f:  M\longrightarrow  M'$.  
  Then  
  $f$  induces       commutative  diagrams   
    \begin{eqnarray}\label{26-6-2-diag1}
  \xymatrix{
 \delta \vec{\mathcal{H}}(M) \ar[d]_-{\delta  f}
 \ar[r]
 & \vec{\mathcal{H}}(M)  \ar[d]_-{f}
 \ar[r]
 & \Delta \vec{\mathcal{H}}(M) \ar[d]^-{\Delta  f}
 \ar[r]
 &\overrightarrow{\rm  Ind}(M)\ar[d]^-{\overrightarrow {\rm  Ind}(f)}
 \\
  \delta\vec{\mathcal{H}}'(M') 
 \ar[r]
 &\vec{\mathcal{H}}'(M')  
 \ar[r]
 & \Delta\vec{\mathcal{H}}'(M')
  \ar[r]
 &\overrightarrow{\rm  Ind}(M'),
  }\\
\label{26-6-2-diag8}
     \xymatrix{
\bar \delta \vec{\mathcal{H}}(M) \ar[d]_-{\bar\delta  f}
 \ar[r]
 & \vec{\mathcal{H}}(M)  \ar[d]_-{  f}
 \ar[r]
 &\bar \Delta \vec{\mathcal{H}}(M) \ar[d]^-{\bar\Delta  f}
 \ar[r]
 &\overrightarrow{\rm  Ind}(M)\ar[d]^-{\overrightarrow {\rm  Ind}(f)}
 \\
  \bar\delta\vec{\mathcal{H}}'(M') 
 \ar[r]
 &\vec{\mathcal{H}}'(M')  
 \ar[r]
 & \bar\Delta\vec{\mathcal{H}}'(M')
  \ar[r]
 &\overrightarrow{\rm  Ind}(M').
 }
\end{eqnarray}
    \end{lemma}
  
  \begin{proof}
  The  lemma  follows  from  a  straightforward  verification.  
  The  argument  is  a  directed  version  of  an  analogue  of  \cite[Theorem~7.3]{ca26}.  
  \end{proof}
  
   \begin{lemma}
  \label{pr-26.6.1.2}
  Let  $f:   {\mathcal{H}}(M)\longrightarrow   {\mathcal{H}}'(M)$  
  be  a  morphism  of  hypergraphs  given  by  an  embedding  
  $f:  M\longrightarrow  M'$.   
  Then  
  $f$  induces    a   commutative  diagram  
    \begin{eqnarray}\label{26-6-2-diag3}
  \xymatrix{
 \delta {\mathcal{H}}(M) \ar[d]_-{\delta  f}
 \ar[r]
 & {\mathcal{H}}(M)  \ar[d]_-{ f}
 \ar[r]
 & \Delta {\mathcal{H}}(M) \ar[d]^-{ \Delta  f}
 \ar[r]
 & {\rm  Ind}(M)\ar[d]^-{{\rm   Ind} ( f)}
 \\
  \delta {\mathcal{H}}'(M') 
 \ar[r]
 & {\mathcal{H}}'(M')  
 \ar[r]
 & \Delta {\mathcal{H}}'(M')
  \ar[r]
 & {\rm  Ind}(M'), 
  }\\
  \label{26-6-2-diag10}
     \xymatrix{
\bar \delta  {\mathcal{H}}(M) \ar[d]_-{\bar\delta  f}
 \ar[r]
 & {\mathcal{H}}(M)  \ar[d]_-{   f}
 \ar[r]
 &\bar \Delta {\mathcal{H}}(M) \ar[d]_-{\bar\Delta  f}
 \ar[r]
 & {\rm  Ind}(M)\ar[d]^-{{\rm   Ind} ( f)}
 \\
  \bar\delta {\mathcal{H}}'(M') 
 \ar[r]
 & {\mathcal{H}}'(M')  
 \ar[r]
 & \bar\Delta {\mathcal{H}}'(M')
  \ar[r]
 & {\rm  Ind}(M').
  } 
  \end{eqnarray}
    \end{lemma}
  
  \begin{proof}
    The  lemma  follows  from  a  straightforward  verification.  
      The  argument  is  an  analogue  of  \cite[Theorem~7.3]{ca26}.  
  \end{proof}

 \begin{proposition}\label{pr-26.6.2-3}
  Let  $f:  \vec{\mathcal{H}}(M)\longrightarrow  \vec{\mathcal{H}}'(M')$  
  be  a  morphism  of  hyperdigraphs  given  by  an  embedding  
  $f:  M\longrightarrow  M'$. 
  Let  $\mathcal{H}(M)$  and  $\mathcal{H}'(M')$  be 
  the  underlying  hypergraphs  of  $\vec{\mathcal{H}}(M)$
  and    $\vec{\mathcal{H}}'(M')$  respectively.  
  Let  $\pi$  be  the  canonical  projection 
  sending  an  ordered  configuration  to  the  unordered  configuration.   
  Then  $f$  induces   four    commutative  diagrams  (\ref{26-6-2-diag1}) --
     (\ref{26-6-2-diag10})  such  that  
 $\pi$  gives  a  projective  morphism  from  
 (\ref{26-6-2-diag1})
 to  (\ref{26-6-2-diag3})
 as  well  as  a  projective  morphism  from  
 (\ref{26-6-2-diag8})
 to  (\ref{26-6-2-diag10}).  
  \end{proposition}
  
  \begin{proof}
  By  Lemma~\ref{pr-26.6.1.1}, 
  $f:  \vec{\mathcal{H}}(M)\longrightarrow  \vec{\mathcal{H}}'(M')$   induces  
     commutative  diagrams   (\ref{26-6-2-diag1})  and  (\ref{26-6-2-diag8}).  
  On  the  other  hand,  
   $f$
     induces  a  canonical   morphism  of  the  underlying  hypergraphs  
   $f:   {\mathcal{H}}(M)\longrightarrow   {\mathcal{H}}'(M')$
  such  that  the   diagram  commutes  
  \begin{eqnarray}\label{eq-26.6.4.1}
  \xymatrix{
  \vec{\mathcal{H}}(M)\ar[r]^-{f}\ar[d]_-{\pi} 
  &\vec{\mathcal{H}}'(M')\ar[d]^-{\pi} \\
    {\mathcal{H}}(M)\ar[r]^-{f}
  & {\mathcal{H}}'(M').
  }
  \end{eqnarray}
   By  Lemma~\ref{pr-26.6.1.2}, 
  $f:  {\mathcal{H}}(M)\longrightarrow   {\mathcal{H}}'(M')$   induces  
    commutative  diagrams   (\ref{26-6-2-diag3})   and  (\ref{26-6-2-diag10}). 
  By  (\ref{eq-26.6.4.1}),  
  \begin{enumerate}[(1)]
\item
   $\pi$  maps  any  entry  $X$   in   (\ref{26-6-2-diag1})  onto    
    the  corresponding  entry  $\pi(X)$  in   (\ref{26-6-2-diag3}); 
 \item
 any  map   $X\longrightarrow  Y $  in   (\ref{26-6-2-diag1})
 and the  corresponding  map 
 $\pi(X)\longrightarrow  \pi(Y)$   in   (\ref{26-6-2-diag3})
 give  a  commutative  diagram 
 \begin{eqnarray*}
 \xymatrix{
 X\ar[r]\ar[d]_-{\pi}  &Y\ar[d]^-{\pi}\\
 \pi(X)\ar[r]  &\pi(Y). 
 }
 \end{eqnarray*} 
 \end{enumerate}
  It  follows  from    (1)  and  (2)  that    
   $\pi$  gives   a  projective  morphism  from  
 (\ref{26-6-2-diag1})
 to  (\ref{26-6-2-diag3}).  
 Similarly,   
    $\pi$  gives   a  projective  morphism  from  
 (\ref{26-6-2-diag8})
 to  (\ref{26-6-2-diag10}).  
   \end{proof}

 \section{Hypergraphs   on  surfaces  and  their  fundamental  groups}
 \label{s.26.7.28.6}

 In  this  section,  
 we  study  the fundamental  groups  (resp.  fundamental  groupoids)  of  
 $k$-uniform  hyper(di)graphs  on  surfaces  as  well as  their 
   connections  with  the      braid groups 
 (resp.      braid groupoids)  on  surfaces.  
 As  some  examples,  
 we  consider  the  persistent  fundamental groupoids  of  packings  in  
 Subsection~\ref{26.6.7-ss8.1}
    and  consider   the  persistent  fundamental groupoids  of  coverings   
    in  Subsection~\ref{26.6.7-ss8.2}.   
Throughout  this  section,  we  assume  that   $M$  is    a   connected  surface.  
 For  any  $k\geq  1$,  
the  {\it   pure  braid  group}  on  $M$
is  
\begin{eqnarray*}
 P_k(M)=\pi_1({\rm  Conf}_k(M))
\end{eqnarray*}
and  the  {\it      braid  group}  on  $M$
is  
\begin{eqnarray*}
 B_k(M)=\pi_1({\rm  Conf}_k(M)/\Sigma_k), 
\end{eqnarray*}
where  $\pi_1$  is  the  fundamental  group.  
Similarly,      the  {\it   pure  braid  groupoid}  on  $M$  may  be  defined  
as  
\begin{eqnarray*}
 \mathcal{P}_k(M)=\Pi_1({\rm  Conf}_k(M))
\end{eqnarray*}
and     the  {\it     braid  groupoid}  on  $M$  may  be  defined  
as  
\begin{eqnarray*}
 \mathcal{B}_k(M)=\Pi_1({\rm  Conf}_k(M)/\Sigma_k), 
\end{eqnarray*}
where  $\Pi_1$  is  the  fundamental  groupoid. 
The  canonical  projection  
(\ref{eq-2601-a})  induces  
  a  group  homomorphism  
\begin{eqnarray*}
\pi_*:   P_k(M)\longrightarrow 
B _{k}(M)  
\end{eqnarray*}
and  a  homomorphism  of  groupoids 
\begin{eqnarray}\label{eq-26.4.24.0}
\pi_*:   \mathcal{P}_k(M)\longrightarrow 
\mathcal{B}_{k}(M).    
\end{eqnarray} 
For   any     $1\leq  i\leq  k$,   
(\ref{eq-2.1a})  induces  
    a  group  homomorphism  
\begin{eqnarray*}
(\partial^i_k)_*:   P_k(M)\longrightarrow 
P _{k-1}(M)  
\end{eqnarray*}
and  a  homomorphism  of  groupoids 
\begin{eqnarray*}
(\partial^i_k)_*:   \mathcal{P}_k(M)\longrightarrow 
 \mathcal{P}_{k-1}(M).    
\end{eqnarray*}
For  any  $s\in   \Sigma_k$,   (\ref{eq-26.4.23.1})  induces   a  group  homorphism  
\begin{eqnarray*}
s_*:  P_k(M) \longrightarrow   P_k(M) 
\end{eqnarray*}
and  a   homomorphism  of  groupoids
\begin{eqnarray*}
s_*:  \mathcal{P}_k(M) \longrightarrow   \mathcal{P}_k(M). 
\end{eqnarray*}
It  follows from  (\ref{eq-4.23.2})  that  
the  diagrams    commute  
\begin{eqnarray*}
\xymatrix{
P_k(M)\ar[r]^-{(\partial^i_k)_*} \ar[d]_-{(\varphi_k^i(\Sigma_{k-1}))_*}   
&P_{k-1}(M)\ar[d]^-{(\Sigma_{k-1})_*}\\
P_k(M) \ar[r]^-{(\partial^i_k)_*}  &P_{k-1}(M),  
}~~~~~~
\xymatrix{
\mathcal{P}_k(M)\ar[r]^-{(\partial^i_k)_*} \ar[d]_-{(\varphi_k^i(\Sigma_{k-1}))_*}   
&\mathcal{P}_{k-1}(M)\ar[d]^-{(\Sigma_{k-1})_*}\\
\mathcal{P}_k(M) \ar[r]^-{(\partial^i_k)_*}  &\mathcal{P}_{k-1}(M).   
}
\end{eqnarray*}
Let  $\vec{\mathcal{H}}(M)=\bigcup_{k\geq  0}\vec{\mathcal{H}}_k(M)$
be  a  hyperdigraph  on  $M$  with   its  underlying  hypergraph    
 $\mathcal{H}(M)=\bigcup_{k\geq  0}\mathcal{H}_k(M)$.   

\begin{proposition}\label{pr-4.23.1}
 There are   graded  homomorphisms 
 \begin{eqnarray*}
 \Pi_1(\vec{\mathcal{H}}_\bullet(M))\longrightarrow  \mathcal{P}_\bullet(M),
 ~~~~~~\Pi_1({\mathcal{H}}_\bullet(M) )\longrightarrow  \mathcal{B}_\bullet(M)
 \end{eqnarray*}
 such that  the  diagram  commutes 
 \begin{eqnarray}\label{diag-26-4-23-fgpd}
\xymatrix{
\Pi_1(\vec{\mathcal{H}}_\bullet(M)) \ar[r] \ar[d]_-{\pi_*}  &
 \mathcal{P}_\bullet(M)  \ar[d]^-{\pi_*} \\
\Pi_1( {\mathcal{H}}_\bullet(M)) \ar[r]  &
 \mathcal{B}_\bullet(M).  
}
\end{eqnarray}
\end{proposition}

\begin{proof}
For  any  $k\geq  1$,  
we  have  a  commutative  diagram
 \begin{eqnarray} \label{diag-26-4-20}
  \xymatrix{
     \vec{\mathcal{H}}_k(M)  \ar[d]_-{\pi}
 \ar[r]
 &{\rm  Conf}_k(M)\ar[d]^-{\pi}
 \\
   {\mathcal{H}}_k(M)  
 \ar[r]
 &{\rm  Conf}_k(M)/\Sigma_k   
  }
  \end{eqnarray}
  where  the  horizontal  maps  are     inclusions  and  the vertical  maps  
  are   projections.  Applying  the fundamental  groupoid  functor  to 
  (\ref{diag-26-4-20}),  
  we  obtain  (\ref{diag-26-4-23-fgpd}).  
\end{proof}

\begin{proposition}\label{pr-4.23.2}
Suppose  for  each  $k\geq  1$,   $\vec{\mathcal{H}}_k(M)$
 is  a   connected  manifold.  
 Then  there are   graded  homomorphisms 
 \begin{eqnarray*}
 \pi_1(\vec{\mathcal{H}}_\bullet(M))\longrightarrow   {P}_\bullet(M),
 ~~~~~~\pi_1({\mathcal{H}}_\bullet(M) )\longrightarrow  {B}_\bullet(M)
 \end{eqnarray*}
 such that  the  diagram  commutes 
 \begin{eqnarray}\label{diag-26-4-23-fgp}
\xymatrix{
\pi_1(\vec{\mathcal{H}}_\bullet(M)) \ar[r] \ar[d]_-{\pi_*}  &
  {P}_\bullet(M)  \ar[d]^-{\pi_*} \\
\pi_1( {\mathcal{H}}_\bullet(M)) \ar[r]  &
  {B}_\bullet(M).  
}
\end{eqnarray}
\end{proposition}

\begin{proof}
For  each  $k\geq  1$,   the  connectivity  of  $\vec{\mathcal{H}}_k(M)$
  implies the  connectivity  of  
   $\mathcal{H}_k(M)$. 
 Applying  the fundamental  group   functor  to 
  (\ref{diag-26-4-20}),  
  we  obtain  (\ref{diag-26-4-23-fgp}).
\end{proof}

In  the  remaining  part  of  this  section,  we  
suppose  $M$  has  a  Riemannian  metric  $g$.

\subsection{Persistent  fundamental  groupoids  of  packings}\label{26.6.7-ss8.1}

  For  any  $k\geq  1$  and  any  $\mathbb{I}\subseteq  [0,+\infty]$,  
 following  the  notations in  Subsection~\ref{ss4.1-26apr},  
 we  let  
 \begin{eqnarray*}
 \mathcal{P}_k(M,\mathbb{I}) = \Pi_1( {\rm  Conf}_k(M, \mathbb{I})), ~~~~~~
  \mathcal{B}_k(M,\mathbb{I}) = \Pi_1( {\rm  Conf}_k(M, \mathbb{I})/\Sigma_k). 
 \end{eqnarray*}
 The  canonical  projection  $\pi$     induces  a   homomorphism  of  groupoids 
 \begin{eqnarray}\label{eq-26-4-23-15}
 \pi_*:   \mathcal{P}_k(M,\mathbb{I})\longrightarrow   \mathcal{B}_k(M,\mathbb{I}).  
 \end{eqnarray}
 The  fundamental  groupoid   of     (\ref{eq-filt-1})  consists  of   a  
 family  of  groupoids 
 \begin{eqnarray}\label{eq-pers-fgpd1}
  \mathcal{P}_k(M,(-,-])=\{\mathcal{P}_k(M,(t,s])\mid  
  (t,s)\in\Omega\}
 \end{eqnarray}
 and  the  fundamental  groupoid  of     (\ref{eq-filt-2})   consists  of   a  
 family  of  groupoids 
 \begin{eqnarray}\label{eq-pers-fgpd2}
  \mathcal{B}_k(M,(-,-])=\{\mathcal{B}_k(M,(t,s])\mid  
  (t,s)\in\Omega\}. 
 \end{eqnarray}
 For   any  $0\leq  t'\leq  t\leq  s\leq  s'\leq  +\infty$,  
    the  canonical  inclusion      
   (\ref{eq-26-feb-27-5}) 
   induces  
  a  homomorphism  from   
 $\mathcal{P}_k(M,(t,s])$  to  $ \mathcal{P}_k(M,(t',s'])$
 and  the  canonical  inclusion   (\ref{eq-26-feb-27-3}) 
 induces  a  homomorphism  from  
 $\mathcal{B}_k(M,(t,s])$  to   $ \mathcal{B}_k(M,(t',s'])$.  
 Thus  (\ref{eq-pers-fgpd1}) 
 and  (\ref{eq-pers-fgpd2})  are    double-persistent  groupoids.

  \begin{proposition}
  The  homomorphism  
 (\ref{eq-26-4-23-15})  gives  a   persistent  morphism   of   
 double-persistent    groupoids 
 \begin{eqnarray}\label{eq-pers-fgpd3}
 \pi_*: \mathcal{P}_k(M,(-,-])\longrightarrow
  \mathcal{B}_k(M,(-,-]).  
 \end{eqnarray}
 \end{proposition}
 
 \begin{proof}
  For   any  $0\leq  t'\leq  t\leq  s\leq  s'\leq  +\infty$,  
  there  is  a  pull-back   diagram  
 \begin{eqnarray*}
 \xymatrix{
 {\rm  Conf}_k(M, (t, s])  \ar[r]\ar[d]_-{\pi}
 &{\rm  Conf}_k(M, (t',s'])\ar[d]^-{\pi}\\
  {\rm  Conf}_k(M, (t, s])/\Sigma_k  \ar[r] 
 &{\rm  Conf}_k(M, (t', s'])/\Sigma_k     
 }
 \end{eqnarray*}
 where  the  horizontal  maps  are  canoical  inclusions.  
 Applying  the  fundamental  groupoid  functor,  
 we  have  an  induced  commutative  diagram  
 \begin{eqnarray*}
 \xymatrix{
 \mathcal{P}_k(M, (t, s])  \ar[r]\ar[d]_-{\pi_*}
 &\mathcal{P}_k(M, (t',s'])\ar[d]^-{\pi_*}\\
  \mathcal{B}_k(M, (t, s])   \ar[r] 
 &\mathcal{B}_k(M, (t', s']).        
 }
 \end{eqnarray*}
 Therefore,  
 (\ref{eq-pers-fgpd3})  is  a  persistent  morphism  
 of  double-persistent    groupoids.   
 \end{proof}

\subsection{Persistent  fundamental  groupoids  of  coverings}
\label{26.6.7-ss8.2}

 For  any  $k\geq  1$  and  any  $r>0$,  following    
 the  notations  in  Subsection~\ref{ss-4.4-cover},  we  consider  the  
fundamental  groupoids  
 \begin{eqnarray}\label{eq-fgrd-cv-1}
 \Pi_1({\rm  Cover}_k(M, r)),~~~~~~  \Pi_1({\rm  Cover}_k(M, r)/\Sigma_k). 
 \end{eqnarray}
 The  canonical  projection  $\pi$     induces  a   homomorphism  of  groupoids 
 \begin{eqnarray}\label{eq-26-4-25-1}
 \pi_*:   \Pi_1({\rm  Cover}_k(M, r))\longrightarrow  \Pi_1({\rm  Cover}_k(M, r)/\Sigma_k).  
 \end{eqnarray}
Let  $r$  run  over  all  positive  numbers.  
Then     (\ref{eq-fgrd-cv-1})  gives   
 families   of  groupoids 
 \begin{eqnarray}\label{eq-pers-fgpd-cv-2}
 \Pi_1({\rm  Cover}_k(M, -))&=&\{\Pi_1({\rm  Cover}_k(M, r))\mid  
  r>0\},\\
   \Pi_1({\rm  Cover}_k(M, -)/\Sigma_k)&=&\{\Pi_1({\rm  Cover}_k(M, r)/\Sigma_k)\mid  
  r>0\}. 
  \label{eq-pers-fgpd-cv-25}
 \end{eqnarray}
 For  any  $0<r\leq  r'<+\infty$,  
  the  canonical  inclusion   
  $
  {\rm  Cover}_k(M, r)  \subseteq   
   {\rm  Cover}_k(M, r')
  $
    induces    a   homomorphism     
  \begin{eqnarray*}
    \Pi_1({\rm  Cover}_k(M, r)) \longrightarrow 
    \Pi_1({\rm  Cover}_k(M, r')) 
    \end{eqnarray*}
   and    the  canonical  inclusion   
  $
   {\rm  Cover}_k(M, r)/\Sigma_k \subseteq  
   {\rm  Cover}_k(M, r')/\Sigma_k 
   $
     induces  a  homomorphism     
   \begin{eqnarray*}
   \Pi_1({\rm  Cover}_k(M, r)/\Sigma_k) \longrightarrow 
   \Pi_1({\rm  Cover}_k(M, r')/\Sigma_k). 
   \end{eqnarray*} 
   Thus  (\ref{eq-pers-fgpd-cv-2})  and  (\ref{eq-pers-fgpd-cv-25})
  are  persistent  groupoids.  
  
\begin{proposition}
  The  homomorphism  
 (\ref{eq-26-4-25-1})  gives  a   persistent  morphism   of   
 persistent    groupoids   
 \begin{eqnarray}\label{eq-pers-fgpd-cv-3}
 \pi_*:\Pi_1({\rm  Cover}_k(M, -))\longrightarrow
  \Pi_1({\rm  Cover}_k(M, -)/\Sigma_k).  
 \end{eqnarray} 
 \end{proposition}
 
 \begin{proof}
   For  any  $0<r\leq  r'<+\infty$,   
  there  is  a  pull-back   diagram  
 \begin{eqnarray*}
 \xymatrix{
 {\rm  Cover}_k(M, r)  \ar[r]\ar[d]_-{\pi}
 & {\rm  Cover}_k(M, r')\ar[d]^-{\pi}\\
  {\rm  Cover}_k(M, r)/\Sigma_k  \ar[r] 
 & {\rm  Cover}_k(M, r')/\Sigma_k     
 }
 \end{eqnarray*}
 where  the  horizontal  maps  are  canoical  inclusions.  
 Applying  the  fundamental  groupoid  functor,  
 we  have  an  induced  commutative  diagram  
 \begin{eqnarray*}
 \xymatrix{
\Pi_1({\rm  Cover}_k(M, r) )\ar[r]\ar[d]_-{\pi_*}
 &\Pi_1({\rm  Cover}_k(M, r')) \ar[d]^-{\pi_*}\\
  \Pi_1({\rm  Cover}_k(M, r)/\Sigma_k   )  \ar[r] 
 &\Pi_1({\rm  Cover}_k(M, r') /\Sigma_k  ).        
 }
 \end{eqnarray*}
 Therefore,  
 (\ref{eq-pers-fgpd-cv-3})  is  a  persistent  morphism  
 of   persistent    groupoids.   
 \end{proof}

 \section{The  hypergraphs  of  frames}\label{s.26.7.28.7}
 
 In  this  section,  we  give  more  examples  of  simplicial  complexes  and  
 hypergraphs  on  manifolds.  
We  consider  the  frames  in  Euclidean  spaces 
as  well  as  the  frame  bundles  on  manifolds.  
 We  construct  simplicial  complexes  of   frames  in  Euclidean  spaces
 as    unions  of  Stiefel  manifolds  in  Subsection~\ref{ss4.1}  
 and  construct   hypergraphs 
 of   constraint  frames  in  Euclidean  spaces
 as    unions  of  submanifolds  of   Stiefel  manifolds  in  Subsection~\ref{ss4.2}. 
 We  construct  simplicial  complexes  of   frame  bundles  on  manifolds
    in  Subsection~\ref{ss4.3}  
 and  construct  hypergraphs   of  constraint  frame  bundles  on  manifolds
 in  Subsection~\ref{ss4.a}.  
 This  section  provides  preliminary  groundwork
   for  the regular  maps  in  Section~\ref{s-27.7.28.8}.

 \subsection{The    simplicial  complexes   of   frames}\label{ss4.1}
 
 Let    $\mathbb{F}$  be  the   real  numbers  $\mathbb{R}$  or  the  complex numbers  
   $\mathbb{C}$. 
   Let  $\mathbb{F}^n$  be  the  $n$-dimensional  vector  space  over  $\mathbb{F}$.  
 For  any  $1\leq  k\leq  n+1$,
  an    {\it   affine  $k$-frame}  in  $\mathbb{F}^n$  
  is an  ordered  $k$-tuple  $(a_1,a_2,\ldots, a_k)$  
  of  affinely  independent  vectors  in  $\mathbb{F}^n$.
   The  collection  of  all  the  affine  $k$-frames  in  $\mathbb{F}^n$  forms  
  a    (real  or  complex)  manifold  $A_k(\mathbb{F}^n)$,  
 which  is  called  the   {\it     (real  or  complex)  affine  Stiefel  manifold}.  
 The  {\it  affine  $0$-frame}  is  the  empty  sequence.

   For  any  $1\leq  k\leq  n$,  
  a  {\it  $k$-frame}  in  $\mathbb{F}^n$  
  is an  ordered  $k$-tuple  $(e_1,e_2,\ldots, e_k)$  
  of  linearly  independent  vectors  in  $\mathbb{F}^n$     
  and  an   {\it  orthonormal   $k$-frame}  in  $\mathbb{F}^n$  
  is an  ordered  $k$-tuple  $(\epsilon_1,\epsilon_2,\ldots, \epsilon_k)$  
  of  orthonormal  vectors  in  $\mathbb{F}^n$.   
  The  {\it  $0$-frame}  as  well  as  the  {\it  orthonormal   $0$-frame} 
   is  the  empty  sequence.  
  The  collection  of  all  the  $k$-frames  in  $\mathbb{F}^n$  forms  
  a  non-compact  (real  or  complex)   manifold  $\tilde  V_k(\mathbb{F}^n)$, 
  which  is  a  submanifold   
  of  $A_k(\mathbb{F}^n)$  
 and   is  called  the   {\it   non-compact    (real  or  complex)  Stiefel  manifold}. 
  The  collection  of  all  the  orthonormal  $k$-frames  in  $\mathbb{F}^n$  forms  
  a  compact  (real  or  complex)  manifold  $ V_k(\mathbb{F}^n)$,    which   is   a  
  submanifold  of     $\tilde  V_k(\mathbb{F}^n)$   and   is  
  called  the  
  {\it  (real  or  complex)  Stiefel  manifold}.  
  The  canonical  inclusion  of  $ V_k(\mathbb{F}^n)$ 
  in     $\tilde  V_k(\mathbb{F}^n)$  is  a  homotopy  equivalence.

    The  symmetric  group  $\Sigma_k$  acts  on 
    $A_k(\mathbb{F}^n)$,  $\tilde  V_k(\mathbb{F}^n)$
    and   $ V_k(\mathbb{F}^n)$  
  by  permuting  the  order  of  the  vectors  in  each  (affine)  $k$-frame.  
  The  canonical  inclusions  
  \begin{eqnarray*}
   V_k(\mathbb{F}^n)\longrightarrow  
   \tilde  V_k(\mathbb{F}^n)\longrightarrow 
   A_k(\mathbb{F}^n) 
  \end{eqnarray*}
  are  $\Sigma_k$-equivariant.  
    Let  
    \begin{eqnarray*}
    A_\bullet (\mathbb{F}^n)&=&\bigcup_{k=  0}^{n+1}   A_k(\mathbb{F}^n),  
    \\
    \tilde  V_\bullet (\mathbb{F}^n)&=&\bigcup_{k=  0}^n  \tilde  V_k(\mathbb{F}^n),
    \\
    V_\bullet (\mathbb{F}^n)&=&\bigcup_{k=  0}^n   V_k(\mathbb{F}^n)   
    \end{eqnarray*}
 and    
  \begin{eqnarray*}
  A_\bullet (\mathbb{F}^n)/\Sigma_\bullet
  &=&\bigcup_{k=  0}^{n+1}   A_k(\mathbb{F}^n) /\Sigma_k,
\\
  \tilde  V_\bullet (\mathbb{F}^n)/\Sigma_\bullet
  &=&\bigcup_{k=  0}^n   \tilde V_k(\mathbb{F}^n) /\Sigma_k,  
  \\
  V_\bullet (\mathbb{F}^n)/\Sigma_\bullet
  &=&\bigcup_{k=  0}^n   V_k(\mathbb{F}^n) /\Sigma_k. 
  \end{eqnarray*}
  Then  viewed  as  their  underlying  combinatorial  structures,
  \begin{enumerate}[(1)]
  \item
   $  A_\bullet (\mathbb{F}^n)$,  
  $\tilde  V_\bullet (\mathbb{F}^n)$  and   $  V_\bullet (\mathbb{F}^n)$
  are       directed  simplicial  complexes  
  on  $\mathbb{F}^n$  such that  
  $  V_\bullet (\mathbb{F}^n)\subseteq  \tilde  V_\bullet (\mathbb{F}^n)
  \subseteq  A_\bullet (\mathbb{F}^n)$;   
  
    \item
     $  A_\bullet (\mathbb{F}^n)/\Sigma_\bullet$,  
  $\tilde  V_\bullet (\mathbb{F}^n)/\Sigma_\bullet$  and   
  $  V_\bullet (\mathbb{F}^n)/\Sigma_\bullet$
  are        simplicial  complexes  
  on  $\mathbb{F}^n$  such that  
  $  V_\bullet (\mathbb{F}^n)/\Sigma_\bullet\subseteq  
  \tilde  V_\bullet (\mathbb{F}^n)/\Sigma_\bullet
  \subseteq  A_\bullet (\mathbb{F}^n)/\Sigma_\bullet$.
    \end{enumerate}

  Let  
  $\mathbb{F}P^{n-1}$  be  the  real  or  complex  projective  space
  whose  elements  are  (real  or  complex)  lines  in  $\mathbb{F}^n$
  passing  through  the origin.    
  Then   there  is  a  canonical  $\Sigma_k$-equivariant  map
  \begin{eqnarray}\label{eq-4.2.2}
  \varphi:  
 A_k(\mathbb{F}^n) \longrightarrow   {\rm  Conf}_k(\mathbb{F}P^{n-1})   
  \end{eqnarray}
  sending  an  affine  $k$-frame  in  $\mathbb{F}^n$  to  the  
  ordered  $k$-tuple  of  (real  or  complex)  lines  in  $\mathbb{F}^n$.  
  There are      induced  $\Sigma_k$-equivariant  maps   $\varphi$ 
  given  by  the  restrictions
  \begin{eqnarray*}
  \varphi\mid_{\tilde  V_k(\mathbb{F}^n)}:  &&    \tilde  V_k(\mathbb{F}^n)
    \longrightarrow 
   {\rm  Conf}_k(\mathbb{F}P^{n-1}),\\
   \varphi\mid_{ V_k(\mathbb{F}^n)}:  &&    V_k(\mathbb{F}^n)
    \longrightarrow 
   {\rm  Conf}_k(\mathbb{F}P^{n-1}).  
  \end{eqnarray*}
 Therefore,
  we  have  a   commutative  diagram  of   directed  simplicial  complexes  
   \begin{eqnarray*}
  \xymatrix{
      V_\bullet(\mathbb{F}^n)  \ar[r]  \ar[rd]_-{\varphi}
      &  \tilde  V_\bullet(\mathbb{F}^n) \ar[r]\ar[d]^-{  \varphi}
      & A_\bullet(\mathbb{F}^n)\ar[ld]^-{\varphi}\\
  &\overrightarrow {\rm  Ind}(\mathbb{F}P^{n-1})   &
  }
  \end{eqnarray*}
where  all  the  arrrows  are  morphisms  of   
directed  simplicial  complexes  and  the  horizontal  arrows  are  canonical  inclusions. 
   Consequently,  we  have    an  induced  
commutative  diagram  of        simplicial  complexes  
    \begin{eqnarray*}
  \xymatrix{
      V_\bullet(\mathbb{F}^n)/\Sigma_\bullet  \ar[r]  \ar[rd]_-{\varphi/\Sigma_\bullet}
      &  \tilde  V_\bullet(\mathbb{F}^n)/\Sigma_\bullet \ar[r]\ar[d]^-{  \varphi/\Sigma_\bullet}
      & A_\bullet(\mathbb{F}^n)/\Sigma_\bullet\ar[ld]^-{\varphi/\Sigma_\bullet}\\
   &  {\rm  Ind}(\mathbb{F}P^{n-1})   &
  }
  \end{eqnarray*}
where  all   the  arrrows  are  simplicial  maps
 and  the  horizontal  arrows  are  canonical  inclusions.

     Take  a  filtration  
 $   \mathbb{F}^1\subseteq  \mathbb{F}^2\subseteq \cdots 
 \subseteq \mathbb{F}^n\subseteq
  \mathbb{F}^{n+1}\subseteq \cdots $  
  of vector  spaces.  The  colimit  is  
  $ \mathbb{F}^\infty=\bigcup_{n\geq  1}\mathbb{F}^n$.  
  This  induces  a   filtration  
 $   \mathbb{F}P^0 \subseteq  \mathbb{F}P^1 \subseteq \cdots 
 \subseteq \mathbb{F}P^{n-1} \subseteq
  \mathbb{F}P^{n} \subseteq \cdots $  
  of projective  spaces  whose  colimit  is  
  $ \mathbb{F}P^\infty=\bigcup_{n\geq  1}\mathbb{F}P^{n-1}$.  
  For any  $k\geq  1$,  
  we  have   
  a  commutative  diagram   of  (real  or  complex)  manifolds 
    \begin{eqnarray*}
    \xymatrix{
     V_k(\mathbb{F}^k) \ar[r] \ar[d]
 &      V_k(\mathbb{F}^{k+1})  \ar[r] \ar[d]
 &\cdots \ar[r]
&    V_k(\mathbb{F}^{n})  \ar[r]\ar[d]
&      V_k(\mathbb{F}^{n+1})  \ar[r]\ar[d]
 &\cdots \\
 \tilde  V_k(\mathbb{F}^k) \ar[r] \ar[d]
 &  \tilde   V_k(\mathbb{F}^{k+1})  \ar[r] \ar[d]
 &\cdots \ar[r]
&   \tilde  V_k(\mathbb{F}^{n})  \ar[r]\ar[d]
&  \tilde   V_k(\mathbb{F}^{n+1})  \ar[r]\ar[d]
 &\cdots \\
   A_k(\mathbb{F}^k) \ar[r] \ar[d]
 &      A_k(\mathbb{F}^{k+1})  \ar[r] \ar[d]
 &\cdots \ar[r]
&    A_k(\mathbb{F}^{n})  \ar[r]\ar[d]
&      A_k(\mathbb{F}^{n+1})  \ar[r]\ar[d]
 &\cdots \\
 {\rm  Conf}_k(\mathbb{F}P^{k-1}) \ar[r] 
 &  {\rm  Conf}_k(\mathbb{F}P^{k}) \ar[r] 
 &\cdots \ar[r]
&  {\rm  Conf}_k(\mathbb{F}P^{n-1}) \ar[r]
&   {\rm  Conf}_k(\mathbb{F}P^{n}) \ar[r]
 &\cdots 
 }
  \end{eqnarray*}
  where  
  \begin{enumerate}[(1)]
    \item
  the  first   row   is  the  induced   $\Sigma_k$-equivariant  filtration   of  the  
    Stiefel  manifolds,  whose  colimit  is  denoted  as  $V_k(\mathbb{F}^\infty)$;  
  \item
   the  second   row   is  the  induced   $\Sigma_k$-equivariant  filtration   of  the  
    non-compact  Stiefel  manifolds,  
    whose  colimit  is  denoted  as  $\tilde  V_k(\mathbb{F}^\infty)$;  
     \item
  the  third  row   is  the  induced   $\Sigma_k$-equivariant  filtration   of  the  
    affine Stiefel  manifolds,  whose  colimit  is  denoted  as  $A_k(\mathbb{F}^\infty)$;
  \item
   the last  row  is  the  induced   $\Sigma_k$-equivariant  filtration of  configuration 
   spaces  over  projective  spaces, 
   whose  colimit  is  denoted  as  ${\rm  Conf}_k(\mathbb{F}P^{\infty}) $.  
     \end{enumerate}
     This  induces  a  commutative  diagram   of  
       directed  simplicial  complexes
   \begin{eqnarray*}
    \xymatrix{
    V_\bullet(\mathbb{F}^{n})  \ar[r]\ar[d]
&      V_\bullet(\mathbb{F}^{n+1})  \ar[r]\ar[d]
 &\cdots \ar[r] 
 & V_\bullet(\mathbb{F}^{\infty})\ar[d]\\
   \tilde  V_\bullet(\mathbb{F}^{n})  \ar[r]\ar[d]
&  \tilde   V_\bullet(\mathbb{F}^{n+1})  \ar[r]\ar[d]
 &\cdots\ar[r] 
 &\tilde V_\bullet(\mathbb{F}^{\infty}) \ar[d]\\
  A_\bullet(\mathbb{F}^{n})  \ar[r]\ar[d]
&      A_\bullet(\mathbb{F}^{n+1})  \ar[r]\ar[d]
 &\cdots \ar[r] 
 & A_\bullet(\mathbb{F}^{\infty})\ar[d]\\
  \overrightarrow {\rm  Ind}(\mathbb{F}P^{n-1}) \ar[r]
&   \overrightarrow {\rm  Ind}(\mathbb{F}P^{n}) \ar[r]
 &\cdots \ar[r] 
 &  \overrightarrow   {\rm  Ind}(\mathbb{F}P^{\infty}) 
 }
  \end{eqnarray*}
  denoted  as  ${\rm  Dgm(ordered~frames)}$,  
  where  all  the  maps  are  morphisms  of   
      directed  simplicial  complexes.  
 This  also  induces  
  a  commutative  diagram   of  
          simplicial  complexes
   \begin{eqnarray*}
    \xymatrix{
   V_\bullet(\mathbb{F}^{n}) /\Sigma_\bullet  \ar[r]\ar[d]
&      V_\bullet(\mathbb{F}^{n+1}) /\Sigma_\bullet  \ar[r]\ar[d]
 &\cdots \ar[r]
 & V_\bullet(\mathbb{F}^{\infty})/\Sigma_\bullet\ar[d]\\
   \tilde  V_\bullet(\mathbb{F}^{n})/\Sigma_\bullet   \ar[r]\ar[d]
&  \tilde   V_\bullet(\mathbb{F}^{n+1})/\Sigma_\bullet   \ar[r]\ar[d]
 &\cdots  \ar[r]
 & \tilde   V_\bullet(\mathbb{F}^{\infty})/\Sigma_\bullet\ar[d]\\
 A_\bullet(\mathbb{F}^{n}) /\Sigma_\bullet  \ar[r]\ar[d]
&      A_\bullet(\mathbb{F}^{n+1}) /\Sigma_\bullet  \ar[r]\ar[d]
 &\cdots \ar[r]
 & A_\bullet(\mathbb{F}^{\infty})/\Sigma_\bullet\ar[d]\\
  {\rm  Ind}(\mathbb{F}P^{n-1}) \ar[r]
&     {\rm  Ind}(\mathbb{F}P^{n}) \ar[r]
 &\cdots  \ar[r] 
 &      {\rm  Ind}(\mathbb{F}P^{\infty}) 
 }
  \end{eqnarray*}
 denoted  as  ${\rm  Dgm(unordered~frames)}$,  where  all  the  maps  are    simplicial  maps. 
 
 \begin{proposition}\label{pr-4.6.1-stief}
 The   canonical  projections
from  the  ordered  tuples  to  the  unordered  tuples      
  induce   a  morphism  between  commutative  diagrams
  \begin{eqnarray}\label{eq-diag-26.4.8}
  \pi:  {\rm  Dgm(ordered~frames)}\longrightarrow  
  {\rm  Dgm(unordered~frames)}.  
  \end{eqnarray}
  \end{proposition}
  
  \begin{proof}
  Let  $\pi$  be  the  canonical  projection.  
  We have  a  commutative  diagram
    \begin{eqnarray*}
 \xymatrix{
  A_\bullet(\mathbb{F}^{\infty}) \ar[r]^-{\pi}\ar[d]_-{\varphi}
  & 
  A_\bullet(\mathbb{F}^{\infty}) /\Sigma_\bullet\ar[d]^-{\varphi/\Sigma_\bullet}\\   
   \overrightarrow   {\rm  Ind}(\mathbb{F}P^{\infty})\ar[r]^-{\pi}
 &
 {\rm  Ind}(\mathbb{F}P^{\infty}).   
 } 
  \end{eqnarray*}
  By  restricting  $\pi:   A_\bullet(\mathbb{F}^{\infty})\longrightarrow  
  A_\bullet(\mathbb{F}^{\infty}) /\Sigma_\bullet$  to  the  submanifolds  
   $  A_\bullet (\mathbb{F}^n)$,  
  $\tilde  V_\bullet (\mathbb{F}^n)$,    $  V_\bullet (\mathbb{F}^n)$,
  $\tilde  V_\bullet (\mathbb{F}^\infty)$  and   $  V_\bullet (\mathbb{F}^\infty)$,  
   we  obtain    a  morphism 
  (\ref{eq-diag-26.4.8})  of   commutative  diagrams.  
  \end{proof}

 \subsection{The   hypergraphs   of  frames  with  constraints}\label{ss4.2}

     Let  $\vec{\mathcal{H}}(\mathbb{F}^\infty)$  be  
     a  sub-hyperdigraph  of  $ A_\bullet(\mathbb{F}^{\infty})$.  
      Let  $  \vec{\mathcal{H}}'(\mathbb{F}^\infty)= \vec{\mathcal{H}}(\mathbb{F}^\infty) \cap 
     \tilde  V_\bullet (\mathbb{F}^\infty)$  
     be  
     a  sub-hyperdigraph  of  $ \tilde  V_\bullet(\mathbb{F}^{\infty})$
     and  let  
  $\vec{\mathcal{H}}''(\mathbb{F}^\infty)= \vec{\mathcal{H}}(\mathbb{F}^\infty) \cap 
      V_\bullet (\mathbb{F}^\infty)$
       be  
     a  sub-hyperdigraph  of  $  V_\bullet(\mathbb{F}^{\infty})$.    
     For  any  finite   positive  integer  $n$,  
     \begin{eqnarray*}
     \vec{\mathcal{H}}(\mathbb{F}^n)&=&\vec{\mathcal{H}}(\mathbb{F}^\infty)\cap  
     A_\bullet(\mathbb{F}^n),\\
      \vec{\mathcal{H}}'(\mathbb{F}^n)&=&\vec{\mathcal{H}}'(\mathbb{F}^\infty)\cap  
     \tilde  V_\bullet(\mathbb{F}^n),\\
      \vec{\mathcal{H}}''(\mathbb{F}^n)&=&\vec{\mathcal{H}}''(\mathbb{F}^\infty)\cap  
     V_\bullet(\mathbb{F}^n)
     \end{eqnarray*}
       are     sub-hyperdigraphs  of 
   $A_\bullet (\mathbb{F}^n)$,  $\tilde  V_\bullet (\mathbb{F}^n)$
   and  $V_\bullet (\mathbb{F}^n)$  respectively.    
   Then
   $\vec{\mathcal{H}} (\mathbb{F}^n)$  is   the  space  of  all  
  the  ordered  affine   finite  frames  in  $\mathbb{F}^n$   with  the  constraint  defined  by  
  $\vec{\mathcal{H}}(\mathbb{F}^\infty)$,  
    $ \vec{\mathcal{H}}'(\mathbb{F}^n)$  
       is   the  space  of  all  
  the     ordered  finite  frames  in  $\mathbb{F}^n$   with  the  constraint  defined  by  
  $\vec{\mathcal{H}} (\mathbb{F}^\infty)$,  
  and  $  \vec{\mathcal{H}}''(\mathbb{F}^n)$  
       is   the  space  of  all  
  the   ordered   finite  orthonormal  frames  in  $\mathbb{F}^n$   with  the  constraint  defined  by  
  $\vec{\mathcal{H}} (\mathbb{F}^\infty)$.    
     We  have  a  commutative  diagram   of  
       hyperdigraphs
   \begin{eqnarray*}
    \xymatrix{
    \vec{\mathcal{H}}''(\mathbb{F}^n) \ar[r]\ar[d]
&     \vec{\mathcal{H}}''(\mathbb{F}^{n+1})  \ar[r]\ar[d]
 &\cdots \ar[r] 
 &   \vec{\mathcal{H}}''(\mathbb{F}^\infty)\ar[d]\\
     \vec{\mathcal{H}}'(\mathbb{F}^n) \ar[r]\ar[d]
&     \vec{\mathcal{H}}'(\mathbb{F}^{n+1})  \ar[r]\ar[d]
 &\cdots \ar[r] 
 &   \vec{\mathcal{H}}'(\mathbb{F}^\infty)\ar[d]\\
    \vec{\mathcal{H}}(\mathbb{F}^n) \ar[r]\ar[d]
&     \vec{\mathcal{H}}(\mathbb{F}^{n+1})  \ar[r]\ar[d]
 &\cdots \ar[r] 
 &   \vec{\mathcal{H}}(\mathbb{F}^\infty)\ar[d]\\
  \overrightarrow {\rm  Ind}(\mathbb{F}P^{n-1}) \ar[r]
&   \overrightarrow {\rm  Ind}(\mathbb{F}P^{n}) \ar[r]
 &\cdots \ar[r] 
 &  \overrightarrow   {\rm  Ind}(\mathbb{F}P^{\infty}) 
 }
  \end{eqnarray*}
  denoted  as  ${\rm  Dgm}(\vec{\mathcal{H}}(\mathbb{F}^\infty))$,  
  where  all  the  maps  are  morphisms  of   
      hyperdigraphs.

     Let  $\mathcal{H}(\mathbb{F}^\infty)$  
  be  the  underlying  hypergraph   of  
  $ \vec{\mathcal{H}}(\mathbb{F}^\infty)$.
   The  underlying  hypergraphs  of   $\vec{\mathcal{H}}'(\mathbb{F}^\infty)$ 
   and  $\vec{\mathcal{H}}''(\mathbb{F}^\infty) $  are  respectively  
   \begin{eqnarray*}
    {\mathcal{H}}'(\mathbb{F}^\infty) &=& {\mathcal{H}}(\mathbb{F}^\infty) \cap 
     (\tilde  V_\bullet (\mathbb{F}^\infty)/\Sigma_\bullet)  ,\\
      {\mathcal{H}}''(\mathbb{F}^\infty) &=&  {\mathcal{H}}(\mathbb{F}^\infty) \cap 
      (V_\bullet (\mathbb{F}^\infty)/\Sigma_\bullet).  
     \end{eqnarray*}
       For  any  finite   positive  integer  $n$,  
     \begin{eqnarray*}
     {\mathcal{H}}(\mathbb{F}^n)&=& {\mathcal{H}}(\mathbb{F}^\infty)\cap  
    ( A_\bullet(\mathbb{F}^n)/\Sigma_\bullet),\\
       {\mathcal{H}}'(\mathbb{F}^n)&=& {\mathcal{H}}'(\mathbb{F}^\infty)\cap  
     (\tilde  V_\bullet(\mathbb{F}^n)/\Sigma_\bullet),\\
      {\mathcal{H}}''(\mathbb{F}^n)&=& {\mathcal{H}}''(\mathbb{F}^\infty)\cap  
     (V_\bullet(\mathbb{F}^n)/\Sigma_\bullet)
     \end{eqnarray*}
       are   the  underlying  hypergraphs  of  
       $\vec{\mathcal{H}}(\mathbb{F}^n)$, 
       $\vec{\mathcal{H}}'(\mathbb{F}^n)$ 
       and  
       $\vec{\mathcal{H}}''(\mathbb{F}^n)$      respectively.   
        We  have   
  a  commutative  diagram   of  
        hypergraphs
   \begin{eqnarray*}
    \xymatrix{
     {\mathcal{H}}''(\mathbb{F}^n) \ar[r]\ar[d]
&      {\mathcal{H}}''(\mathbb{F}^{n+1})  \ar[r]\ar[d]
 &\cdots \ar[r] 
 &    {\mathcal{H}}''(\mathbb{F}^\infty)\ar[d]\\
      {\mathcal{H}}'(\mathbb{F}^n) \ar[r]\ar[d]
&      {\mathcal{H}}'(\mathbb{F}^{n+1})  \ar[r]\ar[d]
 &\cdots \ar[r] 
 &    {\mathcal{H}}'(\mathbb{F}^\infty)\ar[d]\\
     {\mathcal{H}}(\mathbb{F}^n) \ar[r]\ar[d]
&      {\mathcal{H}}(\mathbb{F}^{n+1})  \ar[r]\ar[d]
 &\cdots \ar[r] 
 &    {\mathcal{H}}(\mathbb{F}^\infty)\ar[d]\\
 {\rm  Ind}(\mathbb{F}P^{n-1}) \ar[r]
&     {\rm  Ind}(\mathbb{F}P^{n}) \ar[r]
 &\cdots \ar[r] 
 &   {\rm  Ind}(\mathbb{F}P^{\infty}) 
 }
  \end{eqnarray*}
  denoted  as  ${\rm  Dgm}( {\mathcal{H}}(\mathbb{F}^\infty))$,  
  where  all  the  maps  are  morphisms  of   
      hypergraphs.
    
    \begin{corollary}\label{co-4.11.2-stief}
 The   canonical  projections
from  the  ordered  tuples  to  the  unordered  tuples      
  induce   a  morphism  between  commutative  diagrams
  \begin{eqnarray}\label{eq-diag-26.4.11}
  \pi:  {\rm  Dgm } (\vec {\mathcal{H}}(\mathbb{F}^\infty))\longrightarrow  
  {\rm  Dgm} ({\mathcal{H}}(\mathbb{F}^\infty)).  
  \end{eqnarray}
  \end{corollary}
  
  \begin{proof}
    We  have   a  commutative  diagram   
      \begin{eqnarray}\label{diag-26.4.11.5}
      \xymatrix{
       \vec{\mathcal{H}}''(\mathbb{F}^\infty)\ar[r] \ar[d]_-{\pi}
      &\vec{\mathcal{H}}'(\mathbb{F}^\infty)\ar[r]\ar[d]_-{\pi}
      &\vec{\mathcal{H}}(\mathbb{F}^\infty)\ar[r]\ar[d]^-{\pi}
      &\overrightarrow{\rm  Ind}(\mathbb{F}P^\infty)\ar[d]^-{\pi}\\
      \mathcal{H}''(\mathbb{F}^\infty)\ar[r] 
      &\mathcal{H}'(\mathbb{F}^\infty)\ar[r]
      &\mathcal{H}(\mathbb{F}^\infty) \ar[r]
      &{\rm  Ind}(\mathbb{F}P^\infty)
      }
      \end{eqnarray}
      where  all  the  horizontal  maps  are  canonical  inclusions  
      and  all  the  vertical  maps  are  canonical  projections.  
      By  restricting  (\ref{diag-26.4.11.5})  to  the  corresponding  
      hyper(di)graphs  
      on  $\mathbb{F}^n$,
      we  still  have  a  commutative  diagram.  
Thus  
 by  restricting  (\ref{eq-diag-26.4.8})  to the  corresponding  sub-hyper(di)graphs, 
 we  obtain  (\ref{eq-diag-26.4.11}).  
  \end{proof}

   \subsection{Frame  bundles  on  manifolds}\label{ss4.3}

    Let  $M$   be  a   (real  or  complex)  
    $n$-dimensional  differentiable  manifold.  
              Let  $TM$  be  the  tangent  bundle  of  $M$.  
      For  any  $0\leq  k\leq  n+1$, 
      let   $A_k(TM)$  be  the  fibre  bundle  of  affine  $k$-frames  on  $M$. 
      Then  for  any  $p\in  M$,  
      the  fibre  of  $A_k(TM)$     at $p$  is    isomorphic  to  
        $A_k(\mathbb{F}^n)$.   
         In  particular,    $A_{n+1}(TM)$  is  a  
        principal  ${AGL}(\mathbb{F}^{n})$-bundle 
          over  $M$,  where  
           ${AGL}(\mathbb{F}^n)$  is  the  $\mathbb{F}$-affine 
             transformation  group  on  
           $\mathbb{F}^n$.  
           An     {\it  affine   (real  or  complex)  $k$-frame    field}  on  $M$  is  a  cross-section 
   of  $A_k(TM)$,  i.e.,   it  is   an  ordered  
     $k$-tuple   $(X_1,X_2,\ldots,X_k)$   of  vector  fields 
     $X_1,X_2,\ldots,X_k$  on  $M$  such  that  
     $X_1(p),X_2(p),\ldots,X_k(p)\in   T_p M$  are  affinely  independent  over  
     $\mathbb{F}$
     for  any  $p\in  M$.        
  Let  
  \begin{eqnarray*}
  U_k(M)=\Gamma(A_k(TM)) 
  \end{eqnarray*}
      be  the  space  of  all  
       the  affine     (real  or  complex)   $k$-frame     fields  on  $M$.    
        Here  $\mathbb{F}=\mathbb{R}$  if  $M$  is  a  real  manifold  
     and  $\mathbb{F}=\mathbb{C}$  if  $M$  is  a  complex  manifold.

      For  any  $0\leq   k\leq  n$,  
      let   $\tilde  V_k(TM)$  be  the  fibre  bundle  of  $k$-frames  on  $M$.  
       Then  for  any  $p\in  M$,  
      the  fibre  of   
       $\tilde  V_k(TM)$   
        at $p$   is    isomorphic  to  
        $\tilde  V_k(\mathbb{F}^n)$.  
          In  particular,  
           $\tilde   V_{n}(TM)$  is  a  
        principal  ${GL}(\mathbb{F}^{n})$-bundle 
          over  $M$,  where  
           ${GL}(\mathbb{F}^n)$  is  the   general  linear   
             transformation  group  on  
           $\mathbb{F}^n$.        
     A    {\it   (real  or  complex)  $k$-frame    field}  on  $M$  is  an  ordered  
     $k$-tuple   $(X_1,X_2,\ldots,X_k)$   of  vector  fields 
     $X_1,X_2,\ldots,X_k$  on  $M$  such  that  
     $X_1(p),X_2(p),\ldots,X_k(p)\in  T_pM$  are  $\mathbb{F}$-linearly  independent  
     for  any  $p\in  M$. 
      Let  
      \begin{eqnarray*}
      \tilde   W_k(M)=\Gamma(\tilde  V_k(TM)) 
      \end{eqnarray*}
        be  the  space  of  all  
       the  (real  or  complex)   $k$-frame    fields  on  $M$.  
        Then   $\tilde   W_k(M)$   is  a  subspace  of  $U_k(M)$.

     Suppose     $M$  has  a  Riemannian  metric  $g$  if  $M$  is  a  real  manifold 
     and  $M$ has a  Hermitian  metric  $h$  if  $M$  is  a  complex  manifold.   
     For  any  $0\leq   k\leq  n$,    
      let   $  V_k(TM)$  be  the  fibre  bundle  of  orthonormal  $k$-frames  on  $M$.
      In  particular,      
           $   V_{n}(TM)$  is  a  
        principal  ${O}(\mathbb{F}^{n})$-bundle 
          over  $M$,    where  
           ${O}(\mathbb{F}^n)$  is  the   orthonormal  group  on  
           $\mathbb{F}^n$,  i.e.,   it  is  $O(n)$  if  $\mathbb{F}=\mathbb{R}$
           and  $U(n)$  if  $\mathbb{F}=\mathbb{C}$.     
    A    {\it   (real  or  complex)  orthonormal      $k$-frame    field}  on  $M$  is  a   
      (real  or  complex)   $k$-frame    field 
       $(X_1,X_2,\ldots,X_k)$     such  that  
     $X_1(p),X_2(p),\ldots,X_k(p)\in  T_pM$  are  orthonormal  with respect to  $g$ or  $h$ 
     for  any  $p\in  M$.
       Let   
       \begin{eqnarray*}
       W_k(M)=\Gamma(  V_k(TM))
       \end{eqnarray*}
         be  the  space  of  all  
       the   (real  or  complex)   orthonormal  $k$-frame    fields  on  $M$.  
            Then  $W_k(M)$  is  a  
          subspace   of   $\tilde   W_k(M)$.

          Let   $\Gamma(TM)$  be  the  space  of  all  smooth  sections  of  $TM$.  
          We  have   directed  simplicial  complexes
          $U_\bullet(M)$,   $\tilde W_\bullet(M)$  and  
    $W_\bullet(M)$
    on    $\Gamma(TM)$  \footnote[6]{      Note  that  $\Gamma(TM)$  is  not  a  finite  dimensional  manifold.  
    Nevertheless,  we  can  still  define  hyper(di)graphs  
    and  (directed)   simplicial  complexes  on  $\Gamma(TM)$.  
    }
           given  by  
  \begin{eqnarray*}
  U_\bullet(M) &=& \bigcup_{k=  0}^{n+1} U_k(M),\\
   \tilde W_\bullet(M) & =& \bigcup_{k= 0}^n  \tilde  W_k(M),\\
    W_\bullet(M) &= &\bigcup_{k=  0}^n  W_k(M).    
    \end{eqnarray*} 
        The  canonical  inclusions    
      \begin{eqnarray*}
          W_\bullet(M) \longrightarrow    \tilde   W_\bullet(M) \longrightarrow
         U_\bullet(M) 
      \end{eqnarray*}
     are    morphisms  of   directed  simplicial  complexes.

   The  symmetric  group  $\Sigma_k$  acts  on  $U_k(M)$
   freely  and  properly  discontinuously.   The  restrictions  of  this  
   $\Sigma_k$-action  to   
   $\tilde   W_k(M)$
   and    $W_k(M)$  
   are  also  free   and  properly  discontinuous.    
   Thus  there  are       $\Sigma_k$-orbit  spaces  
     \begin{eqnarray*}
   W_k(M)/\Sigma_k\subseteq  \tilde   W_k(M)/\Sigma_k  \subseteq   U_k(M)/\Sigma_k      
   \end{eqnarray*}
   such  that  
   \begin{eqnarray*}
   U_k(M)/\Sigma_k &=&  \Gamma(A_k(TM)/\Sigma_k),\\
   \tilde   W_k(M)/\Sigma_k &=&\Gamma(\tilde  V_k(TM)/\Sigma_k),\\
   W_k(M)/\Sigma_k&=&\Gamma( V_k(TM)/\Sigma_k).  
   \end{eqnarray*}  
     At  each  $p\in  M$,  
       the   fibre  of  $  A_k(TM)/\Sigma_k$       is   $  A_k(\mathbb{F}^n)/\Sigma_k$, 
       the fibre  of   $\tilde    V_k(TM)/\Sigma_k$     is
    $\tilde  V_k(\mathbb{F}^n)/\Sigma_k$,   
    and
  the   fibre  of  $   V_k(TM)/\Sigma_k$       is   $  V_k(\mathbb{F}^n)/\Sigma_k$.
      Thus    $W_k(M)/\Sigma_k$  is  a  subspace   of
      $\tilde   W_k(M)/\Sigma_k$  and  
      $\tilde   W_k(M)/\Sigma_k$  is  a  subspace  of  $U_k(M)/\Sigma_k$.   
    We  have    simplicial  complexes
          $U_\bullet(M)/\Sigma_\bullet$,   $\tilde W_\bullet(M)/\Sigma_\bullet$  and  
    $W_\bullet(M)/\Sigma_\bullet$  
    on    $\Gamma(TM)$  
               given  by  
     \begin{eqnarray*}
  U_\bullet(M) /\Sigma_\bullet&=& \bigcup_{k=  0}^{n+1} U_k(M)/\Sigma_k,\\
   \tilde W_\bullet(M)/\Sigma_\bullet & =& \bigcup_{k= 0}^n  \tilde  W_k(M)/\Sigma_k,\\
    W_\bullet(M) /\Sigma_\bullet&= &\bigcup_{k=  0}^n  W_k(M)/\Sigma_k.    
    \end{eqnarray*}    
      The  canonical  inclusions    
      \begin{eqnarray*}
          W_\bullet(M)/\Sigma_\bullet \longrightarrow    \tilde   W_\bullet(M)/\Sigma_\bullet \longrightarrow
         U_\bullet(M) /\Sigma_\bullet
      \end{eqnarray*}
     are     simplicial  maps.

     \begin{corollary}\label{co-4.8.1-stief}
     Let  $M$  be  a  Riemannian  manifold  or  a  Hermitian  manifold.  
 Then  we  have     a     commutative  diagram  of  section  spaces  of   fibre  bundles  over  $M$
 \footnote[7]{The  fibres  in  the  fibre  bundles  of  (\ref{eq-diag-26.4.8.a})
 are  $\Delta$-manifolds,  which  are  not  merely   usual  manifolds.  }
  \begin{eqnarray}\label{eq-diag-26.4.8.a}
 \xymatrix{
 W_\bullet(M) \ar[r]\ar[d]_-{\pi}
 &    \tilde   W_\bullet(M)\ar[r]\ar[d]_-{\pi}
  &         U_\bullet(M) \ar[d]^-{\pi}\\
 W_\bullet(M)/\Sigma_\bullet \ar[r]
 &    \tilde   W_\bullet(M)/\Sigma_\bullet \ar[r]
  &         U_\bullet(M) /\Sigma_\bullet  
 }
 \end{eqnarray}
 where  the  horizontal  maps  are  canonical  inclusions  
 and  the  vertical  maps  are  projections  sending  the  ordered  tuples  
 to  the  unordered  tuples.  
  \end{corollary}
  
  \begin{proof}
   For  any  $p\in  M$,  by 
   Proposition~\ref{pr-4.6.1-stief}, 
   we  have  a  commutative  diagram  of  the  fibres 
   \begin{eqnarray}\label{eq-diag-26.4.8.b}
 \xymatrix{
 V_\bullet(\mathbb{F}^n) \ar[r]\ar[d]_-{\pi}
 &    \tilde   V_\bullet(\mathbb{F}^n)\ar[r]\ar[d]_-{\pi}
  &         A_\bullet(\mathbb{F}^n) \ar[d]^-{\pi}\\
 V_\bullet(\mathbb{F}^n)/\Sigma_\bullet \ar[r]
 &    \tilde  V_\bullet(\mathbb{F}^n)/\Sigma_\bullet \ar[r]
  &         A_\bullet(\mathbb{F}^n) /\Sigma_\bullet.    
 }
 \end{eqnarray}
 Let  $p$  run  over  $M$.  
 The  commutative  diagram 
 (\ref{eq-diag-26.4.8.b})  implies  the  commutative  diagram 
\begin{eqnarray}\label{eq-26.10.21.9}
\xymatrix{
V_\bullet(TM) \ar[r]\ar[d]_-{\pi}
 &    \tilde   V_\bullet(TM)\ar[r]\ar[d]_-{\pi}
  &         A_\bullet(TM) \ar[d]^-{\pi}\\
 V_\bullet(TM)/\Sigma_\bullet \ar[r]
 &    \tilde  V_\bullet(TM)/\Sigma_\bullet \ar[r]
  &         A_\bullet(TM) /\Sigma_\bullet
}
\end{eqnarray}
 where  all  the  arrows  are  morphisms  of  bundles  over  $M$. 
 Taking  the  cross-sections  of  the  fibre  bundles  in  (\ref{eq-26.10.21.9}),  
 we  obtain   (\ref{eq-diag-26.4.8.a}).  
  \end{proof}
  
     \subsection{Frame  bundles  with  constraints on  manifolds}\label{ss4.a}

Let   $\vec{\mathcal{H}}(\Gamma(TM))$   be   a  
sub-hyperdigraph  of  $U_\bullet(M)$. 
 Then         
 $\vec{\mathcal{H}}(\Gamma(TM))$   is 
a   hyperdigraph   on  $\Gamma(TM)$.   
  An      {\it     affine   (real  or  complex)   $k$-frame    field  on  $M$ 
    with  constraint  
    $\vec{\mathcal{H}}(\Gamma(TM))$}    is  an  ordered  
     $k$-tuple   $(X_1,X_2,\ldots,X_k)$   of  vector  fields 
     $X_1,X_2,\ldots,X_k$  on  $M$  such  that  $(X_1,X_2,\ldots,X_k)\in  \vec{\mathcal{H}}(\Gamma(TM))$.      
Let  ${\mathcal{H}}(\Gamma(TM))$  be  the 
  underlying   hypergraph   
of   $\vec{\mathcal{H}}(\Gamma(TM))$.  
Then     ${\mathcal{H}}(\Gamma(TM))$  
is  a  sub-hypergraph  of   $U_\bullet(M)/\Sigma_\bullet$.    
Let     $\vec{\mathcal{H}}'(\Gamma(TM))=\vec{\mathcal{H}}(\Gamma(TM))
\cap   \tilde W_\bullet(M)$  and  let  
$\vec{\mathcal{H}}''(\Gamma(TM))=\vec{\mathcal{H}}(\Gamma(TM))
\cap     W_\bullet(M)$. 
  For  any  $0\leq  k\leq  n $,  
  a      {\it     (real  or  complex)  $k$-frame    field  on  $M$ 
    with  constraint  
    $\vec{\mathcal{H}}'(\Gamma(TM))$}
           is  an  ordered  
     $k$-tuple   $(X_1,X_2,\ldots,X_k)$   of  vector  fields 
     $X_1,X_2,\ldots,X_k$  on  $M$  such  that  $(X_1,X_2,\ldots,X_k)\in  \vec{\mathcal{H}}'(\Gamma(TM))$,
     and  
    a      {\it     (real  or  complex)  orthonormal  $k$-frame    field  on  $M$ 
    with  constraint  $\vec{\mathcal{H}}''(\Gamma(TM))$}  
        is  an  ordered  
     $k$-tuple   $(X_1,X_2,\ldots,X_k)$   of  vector  fields 
     $X_1,X_2,\ldots,X_k$  on  $M$  such  that   
       $(X_1,X_2,\ldots,X_k)\in  \vec{\mathcal{H}}''(\Gamma(TM))$.      
Let  ${\mathcal{H}}'(\Gamma(TM))$   and  ${\mathcal{H}}''(\Gamma(TM))$   be  the 
  underlying    hypergraphs   
of   $\vec{\mathcal{H}}'(\Gamma(TM))$  and  $\vec{\mathcal{H}}''(\Gamma(TM))$ 
respectively. 
Then  
     ${\mathcal{H}}'(\Gamma(TM))$  
is  a  sub-hypergraph  of   $\tilde  W_\bullet(M)/\Sigma_\bullet$
and  ${\mathcal{H}}''(\Gamma(TM))$  
is  a  sub-hypergraph  of   $W_\bullet(M)/\Sigma_\bullet$.

   \begin{corollary}\label{co-4.12.1-stief}
     Let  $M$  be  a  Riemannian  manifold  or  a  Hermitian  manifold.  
 Then  we  have     a     commutative  diagram  of   section  spaces  of   fibre  bundles  over  $M$
 \footnote[8]{The  fibres  in  the  fibre  bundles  of  (\ref{eq-diag-26.4.12.a})
 are  graded  submanifolds  of  $\Delta$-manifolds,  which  are  not  merely   
 usual  manifolds.  }
  \begin{eqnarray}\label{eq-diag-26.4.12.a}
 \xymatrix{
\vec{\mathcal{H}}''(\Gamma(TM)) \ar[r]\ar[d]_-{\pi}
 &    \vec{\mathcal{H}}'(\Gamma(TM)) \ar[r]\ar[d]_-{\pi}
  &         \vec{\mathcal{H}}(\Gamma(TM))  \ar[d]^-{\pi}\\
 {\mathcal{H}}''(\Gamma(TM))  \ar[r]
 &    {\mathcal{H}}'(\Gamma(TM))  \ar[r]
  &       {\mathcal{H}}(\Gamma(TM)) 
 }
 \end{eqnarray}
 where  the  horizontal  maps  are  canonical  inclusions  
 and  the  vertical  maps  are  projections  sending  the  ordered  tuples  
 to  the  unordered  tuples.  
  \end{corollary}
  
  \begin{proof}
    By  restricting  (\ref{eq-diag-26.4.8.a})  to the  corresponding  sub-hyper(di)graphs, 
 we  obtain  (\ref{eq-diag-26.4.12.a}).  
  \end{proof}

\section{Regular  maps  and  geometric  realizations    of      hypergraphs
  on  manifolds}
  \label{s-27.7.28.8}

In  this  section,  
we  generalize  the  classical  geometric  realizations  
of  simplicial  complexes  on   discrete  vertices
and    
    define  the  geometric  realizations  of  simplicial  complexes 
as  well  as  hypergraphs  on  manifolds   by  using  the  affine   
regularity   of    embeddings  of  the manifolds  in  Euclidean  spaces.  
   In  Subsection~\ref{ss-26.6.7-7.1}, 
   we  generalize  the  classical  geometric  realizations  
of  simplicial  complexes  
and  define  the  geometric  realizations  
of  hypergraphs  on  discrete  vertices. 
   In  Subsection~\ref{ss-7.2-geomr}, 
   we  characterize  affinely  $k$-regular  embeddings 
   of  manifolds  in  Euclidean  spaces 
    equivalently  as  geometric  realizations  of  
   the  independence  complexes.  
   We  define  geometric  realizations  of  hypergraphs  on  manifolds  
   as  embeddings  of  the  manifolds  in  Euclidean  spaces  
   with  certain regularities.        
   For  any  
  cardinality   $c$,    
   let  $\mathbb{F}^c$  be  the  $c$-fold  direct  sum  of  $\mathbb{F}$,
   where  $\mathbb{F}=\mathbb{R}$  or  $\mathbb{C}$.

   \subsection{Geometric  realizations  of  hypergraphs   on  discrete  vertices}
   \label{ss-26.6.7-7.1}
  
   Without  loss  of  generality,  
   we  consider  the  real  case $\mathbb{F}=\mathbb{R}$  in  this  subsection.   
   Let  $V$  be  a  discrete  set  with  its  cardinality  ${\rm  card}(V)=c$.  
     Let  $\{\epsilon(v)\mid  v\in  V\}$  be  an    orthonormal  basis  of  
   $\mathbb{R}^c$ where  
   \begin{eqnarray}\label{eq-26-4-19-b}
   \epsilon:  V\longrightarrow \{\epsilon(v)\mid  v\in  V\}
   \end{eqnarray}
     is  a  bijection.   
     For  any  vertex  $v\in  V$,  its  geometric  realization   $|v|$  is  $\epsilon(v)$.  
     For  any  $k\geq  2$  and   any  $k$-hyperedge  $\sigma$  on  $V$,  
  its  geometric  realization  $|\sigma|$ is  the  interior  of  the  convex  hull  
  of  $\{\epsilon(v)\mid  v\in \sigma\}$,  which  is  an  open  subset  of  
  a  $(k-1)$-dimensional  hyperplane  in  $\mathbb{R}^c$. 
  
   \begin{lemma}\label{le-5.1}
  For  any  hyperedges  $\sigma$  and  $\tau$  on  $V$,  
  $\sigma\subseteq  \tau$  if  and  only  if  
      $|\sigma|$  is  a  subset  of  the  boundary  of  the  closure  of  $|\tau|$.  
  \end{lemma}
  
  \begin{proof}
 The  lemma  is  obvious  if  $\sigma=\{v\}$  or  $\tau=\{u\}$  for  some  $v,u\in  V$.  
  Without  loss  of  generality,  assume  
  $\sigma=\{v_0,v_1,\ldots, v_k\}$  and  
   $\tau=\{u_0,u_1,\ldots, u_l\}$  where  $k,l\geq  1$.  
   Then  
   \begin{eqnarray}\label{eq-26.4.geo.1}
   |\sigma|  =  \Big\{\sum_{i=0}^k  x_i \epsilon(v_i)\mid \sum_{i=0}^k  x_i=1,  x_0,x_1,\ldots,x_k>0 \Big\} 
     \end{eqnarray}
     and  the   boundary  of  the  closure  of  $|\tau|$  is 
     \begin{eqnarray}\label{eq-26.4.geo.2}
    {\rm  Bd}  ({\rm  Cl}(|\tau|) )= \Big\{\sum_{i=0}^l  y_i \epsilon(u_i)\mid \sum_{i=0}^l  y_i=1, 
    y_0,y_1,\ldots,y_l\geq  0  {\rm ~and~some~}   y_i=0 \Big\}. 
   \end{eqnarray}  
   Thus  by  (\ref{eq-26.4.geo.1})  and  (\ref{eq-26.4.geo.2}), 
   $\sigma\subseteq  \tau$  if  and  only  if 
   $|\sigma|$  is  a  subset  of  $ {\rm  Bd}  ({\rm  Cl}(|\tau|) )$.  
  \end{proof}

  Let  $\mathcal{H} $  be  a  hypergraph  on  $V$.  
 The  geometric  realization  of  $\mathcal{H} $
  is  the  subset  of  $\mathbb{R}^c$ given  by 
  \begin{eqnarray}\label{eq-4.2.1}
  |\mathcal{H} |= \bigcup_{\sigma \in \mathcal{H} }  |\sigma |.  
  \end{eqnarray}
  In  particular,  if  $\mathcal{H}$  is  a  simplicial  complex,
   then  (\ref{eq-4.2.1})  coincides  with the  usual  geometric  realization
   of  simplicial  complexes   (cf.  \cite[p.  103]{hatcher}).  
    Let  $\Delta[V]$  be  the  simplicial  complex  consisting 
     of   all  the  finite  subsets of  $V$.

  \begin{proposition}\label{pr-4-geom-1}
  For  any   hypergraph  $\mathcal{H} $    on  $V$,  
  \begin{enumerate}[(1)]
  \item
   $|\Delta\mathcal{H} |$  is  the  closure  of  $ |\mathcal{H} |$
   in  $\mathbb{R}^c$.  
   Thus  
   $ |\mathcal{H} |$ 
   is  a  closed  subset  
  of  $\mathbb{R}^c$  if  and  only  if  
  $\mathcal{H} $  is  
    a   simplicial  complex;
    \item
    $ |\mathcal{H} |$  is  an  open  subset   of    $|\Delta[V]|$
    if  and  only  if  $\mathcal{H}$  is  an  independence  hypergraph.  
    Thus 
    $|\bar \Delta\mathcal{H} |$  is  an  open  neighborhood  of  $ |\mathcal{H} |$
   in  $|\Delta[V]|$.  
    \end{enumerate}
    \end{proposition} 
  
  \begin{proof}
 (1) 
 It  follows  from (\ref{eq-4.2.1})    that  
 \begin{eqnarray}\label{eq-4.geom-bdh1}
| \Delta\mathcal{H}|  = |\mathcal{H}|\bigcup \Big(\bigcup
  _{\tau\subsetneq\sigma\atop 
  \sigma\in \mathcal{H}}  |\tau| \Big). 
  \end{eqnarray}
  By  Lemma~\ref{le-5.1}, 
  for  any $\sigma\in \mathcal{H}$ and  any  $\tau\subsetneq\sigma$,    
     $ |\tau|$  is  a  subset  of  $ {\rm  Bd}  ({\rm  Cl}(|\mathcal{H}|) )$.   
  Therefore,     the   right-hand   side  of  
  (\ref{eq-4.geom-bdh1})   is  the  closure  of  $ |\mathcal{H} |$.

  (2)  It  follows  from (\ref{eq-4.2.1})    that   
  $ |\mathcal{H} |$  is  an  open  subset   of    $|\Delta[V]|$
    if  and  only  if 
     $|\Delta[V] \setminus   \mathcal{H}|=|\Delta[V]|\setminus  |\mathcal{H} |$ 
     is  a  closed  subset   of    $|\Delta[V]|$.  
     With the  help  of  (1),  this  happens    
     if  and  only  if  $ \Delta[V] \setminus   \mathcal{H}$
      is  
    a   simplicial  complex,  or  equivalently,  
    $ \mathcal{H}$
      is  
    an  independence  hypergraph.     
      \end{proof}

 For  any  hyperdigraph  $\vec{\mathcal{H}}$  on  $V$ 
  with  its  underlying  hypergraph  ${\mathcal{H}}$,  
  the  {\it  dimension}  is  
  \begin{eqnarray}\label{eq-26-5-7-2}
  c(\vec{\mathcal{H}} )
  =\sup\{k \mid  \vec{\mathcal{H}}_k\neq\emptyset\}-1=
  \sup\{k \mid   {\mathcal{H}}_k\neq\emptyset\}-1=c({\mathcal{H}} ). 
  \end{eqnarray}
\begin{corollary}\label{pr-26.5.7.1}
  Let  $\mathcal{H}$  be   a  hypergraph  with  $c({\mathcal{H}} )=d$.    
  \begin{enumerate}[(1)]
  \item
  If  $d=0$,  then  $|\mathcal{H}|$   is   a  discrete  set  of  points;  
  \item
  If  $1\leq  d< +\infty$, 
  then  $|\mathcal{H}|$   contains  an   open  set  of  $\mathbb{R}^d$ 
  and  does  not  contain  any  open  set  of  $\mathbb{R}^{d+1}$;    
  \item
  If  $d=+\infty$,  
   then  $|\mathcal{H}|$   contains  an   open  set  of  $\mathbb{R}^n$
   for  any  finite  positive  integer  $n$.  
  \end{enumerate}
\end{corollary}
   
   \begin{proof}
   The  proof  follows  directly  from  (\ref{eq-4.2.1})  and  (\ref{eq-26-5-7-2}).  
   \end{proof}
   
   \subsection{(Affinely)  $k$-regular  embeddings      of  manifolds}\label{ss-7.2-geomr}

     Let  $M$  be  a  differentiable  manifold.

    \begin{definition}[cf.  \cite{cohen1,handel2,high1,high2}]
    \label{def-0}
    For  any    embedding  $f:  M\longrightarrow  \mathbb{F}^c$   
    and  any       $k\geq  2$,    
    \begin{enumerate}[(1)]
    \item
     if  for  any  distinct  $k$-points  $x_1, \ldots, x_k\in  M$,  
    their  images   $f(x_1), \ldots, f(x_k) $  are  (real  or  complex)
    affinely  independent,  
    then  we  say that  $f$   is  (real  or  complex)  {\it  affinely  $k$-regular};
    
    \item
     if  for  any  distinct  $k$-points  $x_1, \ldots, x_k\in  M$,  
    their  images   $f(x_1), \ldots, f(x_k) $  are  (real  or  complex)
    linearly  independent,  
    then  we  say that  $f$   is  (real  or  complex)  {\it    $k$-regular}. 
    \end{enumerate}
    \end{definition}

 \begin{proposition}\label{pr-4.18-1}
 Let  $f:  M\longrightarrow  \mathbb{F}^c$  be  an  embedding.  
 Let  ${\rm   Ind}(f):  {\rm  Ind}(M)\longrightarrow  {\rm  Ind}( \mathbb{F}^c)$ 
 be  the  induced  simplicial  embedding given  by  Proposition~\ref{pr-4.15-1}.  
 Then  
    \begin{enumerate}[(1)]
    \item
     $f$  is   (real  or  complex)     affinely  $k$-regular   
     if  and  only  if  
     ${\rm  Ind}(f)$  induces  a  simplicial  embedding  
     of  ${\rm  sk}^{k-1}({\rm  Ind}(M))$  into  
     ${\rm  sk}^{k-1}(  A_\bullet (\mathbb{F}^c)/\Sigma_\bullet)$; 
    \item
     $f$  is   (real  or  complex)       $k$-regular  
     if  and  only  if  
     ${\rm  Ind}(f)$  induces  a  simplicial  embedding  
     of  ${\rm  sk}^{k-1}({\rm  Ind}(M))$  into  
     ${\rm  sk}^{k-1}(  \tilde   V_\bullet (\mathbb{F}^c)/\Sigma_\bullet)$. 
   \end{enumerate}
   \end{proposition}
   
   \begin{proof}
   (1)  
   ($\Longrightarrow$): 
   Suppose  $f$  is        affinely  $k$-regular.  
   Then 
   for  any  simplex  $\{x_1,\ldots, x_l\}$   of   ${\rm  sk}^{k-1}({\rm  Ind}(M))$  
   where  $l\leq  k$,  
   $f(x_1),\ldots, f(x_l)$  are  affinely  independent  in  $\mathbb{F}^c$.   
   Thus  
   $\{f(x_1),\ldots, f(x_l)\}$  is  
    a  simplex  of  $A_l (\mathbb{F}^c)/\Sigma_l$.  
    Therefore, ${\rm  Ind}(f)$  induces  a  simplicial  embedding  
     of  ${\rm  sk}^{k-1}({\rm  Ind}(M))$  into  
     ${\rm  sk}^{k-1}(  A_\bullet (\mathbb{F}^c)/\Sigma_\bullet)$.

      ($\Longleftarrow$): 
  Suppose  ${\rm  Ind}(f)$  induces  a  simplicial  embedding  
     of  ${\rm  sk}^{k-1}({\rm  Ind}(M))$  into  
     ${\rm  sk}^{k-1}(  A_\bullet (\mathbb{F}^c)/\Sigma_\bullet)$.  
     Then  for  any  distinct  $k$-points  
     $x_1, \ldots,  x_k\in  M$,  
     since  $\{x_1,\ldots,x_k\}$  is  a  simplex  of  
     ${\rm  sk}^{k-1}({\rm  Ind}(M))$,  
     $f(x_1), \ldots,  f(x_k)$  are  affinely  independent  
     in  $\mathbb{F}^c$.  
     Therefore,  $f$  is       affinely  $k$-regular.

     (2)   The  proof    is      analogous  to  that  of    (1).  
   \end{proof}

    We  generalize  Definition~\ref{def-0}  in  the  following  
     definition.  
     
       \begin{definition}\label{def-1}
For  any     hypergraph  $\mathcal{H}(M)$  on  $M$,  
\begin{enumerate}[(1)]
\item
 a  {\it  
(real  or  complex)  affinely  regular  embedding} or  
a  {\it  (real  or  complex)  geometric   realization}    of  $\mathcal{H}(M)$
 is  an  embedding   $f:  M\longrightarrow  \mathbb{F}^c$ 
   such  that  for  any  $k\geq   1$  and  any 
    $k$-hyperedge  $\sigma(M)=\{x_1,\ldots,x_k\}$  of  
   $\mathcal{H}(M)$,  its  image  
   $\{f(x_1),\ldots,f(x_k)\}$ is  a  set  of  affinely  independent  
   vectors   in  $\mathbb{F}^c$;

    \item
   a  {\it  (real  or  complex)  regular  embedding}  of  $\mathcal{H}(M)$
 is  an  embedding   $f:  M\longrightarrow  \mathbb{F}^c$ 
   such  that  for  any  $k\geq   1$  and  any 
    $k$-hyperedge  $\sigma(M)=\{x_1,\ldots,x_k\}$  of  
   $\mathcal{H}(M)$,  
    its  image  
   $\{f(x_1),\ldots,f(x_k)\}$ is  a  set  of  linearly  independent  
   vectors   in  $\mathbb{F}^c$. 
   \end{enumerate}  
   \end{definition}
   
   \begin{proposition}\label{pr-4.19-01}
 Let  $f:  M\longrightarrow  \mathbb{F}^c$  be  an  embedding.   
 Then  for  any     hypergraph  $\mathcal{H}(M)$  on  $M$, 
    \begin{enumerate}[(1)]
    \item
     $f$  is   a  (real  or  complex)     affinely   regular   
     embedding  of     $\mathcal{H}(M)$   if  and  only  if  
     ${\rm  Ind}(f)$  induces  an  injective   morphism      
     of  $\mathcal{H}(M)$  into  
     $   A_\bullet (\mathbb{F}^c)/\Sigma_\bullet$; 
    \item
     $f$  is   (real  or  complex)       regular   
     embedding  of     $\mathcal{H}(M)$  
     if  and  only  if  
     ${\rm  Ind}(f)$  induces  an   injective   morphism   
     of  $\mathcal{H}(M)$  into  
     $ \tilde   V_\bullet (\mathbb{F}^c)/\Sigma_\bullet $. 
   \end{enumerate}
   \end{proposition}
   
   \begin{proof}
  The  proof  is  analogous  to  that  of   Proposition~\ref{pr-4.18-1}.  
     \end{proof}

   \begin{corollary}\label{co-26.4.19.1}
    For  any    embedding  $f:  M\longrightarrow  \mathbb{F}^c$   
    and  any       $k\geq  2$,  
   \begin{enumerate}[(1)]
    \item
     $f$  is   a  (real  or  complex)     affinely  $k$-regular   embedding  of  $M$ 
     if  and  only  if  
     $f$  is  a  (real  or  complex)  affinely  regular  embedding  of  ${\rm  sk}^{k-1}({\rm  Ind}(M))$; 
    \item
   $f$  is   a  (real  or  complex)       $k$-regular  embedding   of  $M$ 
     if  and  only  if  
     $f$  is   a (real  or  complex)  regular  embedding  of  ${\rm  sk}^{k-1}({\rm  Ind}(M))$.  
   \end{enumerate}
   \end{corollary}

\begin{proof}
With  the  help  of  Definition~\ref{def-0} and  Definition~\ref{def-1},
the  corollary  is  a  restatement  of  Proposition~\ref{pr-4.18-1}. 
\end{proof}

 \begin{corollary}\label{pr-260419-5}  
   Suppose  $M$  has  a  Riemannian  metric  $g$.   
   Then  for  any    embedding  $f:  M\longrightarrow  \mathbb{F}^c$   
    and  any       $k\geq  2$,  
     \begin{enumerate}[(1)]
    \item
       $f$  is  a  (real  or  complex)     affinely  $k$-regular  embedding   of  $M$  
     if  and  only  if  
     $f$  is  a  (real  or  complex)  affinely  regular   embedding  of
       ${\rm  sk}^{k-1}({\rm  Ind}(M, (2r,+\infty]))$  for  any  $r>0$; 
    \item
      $f$  is  a  (real  or  complex)       $k$-regular  embedding   of  $M$  
     if  and  only  if  
     $f$  is  a  (real  or  complex)    regular   embedding  of
       ${\rm  sk}^{k-1}({\rm  Ind}(M, (2r,+\infty]))$  for  any  $r>0$.  
   \end{enumerate}
   \end{corollary}
   
   \begin{proof}
   (1)  ($\Longrightarrow$): 
   Suppose  $f:  M\longrightarrow  \mathbb{F}^c$  is        affinely  $k$-regular. 
   Then  by  Corollary~\ref{co-26.4.19.1}~(1),  
   $f$  is  an  affinely  regular  embedding  of  ${\rm  sk}^{k-1}({\rm  Ind}(M))$. 
   Thus  $f$  induces  an affinely  regular  embedding  of  the  sub-complex   
    ${\rm  sk}^{k-1}({\rm  Ind}(M, (2r,+\infty]))$  for  any  $r>0$.

    ($\Longleftarrow$): 
    Suppose  $f$  is  an     affinely  regular   embedding  of
       ${\rm  sk}^{k-1}({\rm  Ind}(M, (2r,+\infty]))$  for  any  $r>0$. 
      Suppose  to  the  contrary  that 
       $f$  is  not  an    affinely  $k$-regular  embedding   of  $M$.   
      Then  there  exist    distinct  $k$-points  $x_1,\ldots,x_k\in  M$
      such that  $f(x_1),\ldots, f(x_k)$  are  affinely  dependent  in  $\mathbb{F}^c$.  
      Choose 
      \begin{eqnarray*}
      0<2r_0<\min \{d_M(x_i,x_j)\mid  1\leq  i<j\leq  k\}. 
      \end{eqnarray*}  
  Then  $\{x_1,\ldots,x_k\}\in  {\rm  sk}^{k-1}({\rm  Ind}(M, (2r_0,+\infty]))$. 
  This  constradicts  the   fact     that  
  $f$  is  an     affinely  regular   embedding  of
       ${\rm  sk}^{k-1}({\rm  Ind}(M, (2r_0,+\infty]))$.  
       Therefore,   $f$  is     an    affinely  $k$-regular  embedding   of  $M$.

       (2)    The  proof    is      analogous  to  that  of   (1).  
   \end{proof}

    For  any  hyperdigraph  $\vec{\mathcal{H}}(M)$  on  $M$
  with  its  underlying  hypergraph  ${\mathcal{H}}(M)$,  
  as  a  generalization  of   (\ref{eq-26-5-7-2}),  we  let 
  \begin{eqnarray}\label{eq-26-5-7-5}
  c(\vec{\mathcal{H}}(M))=\sup\{k \mid  \vec{\mathcal{H}}_k(M)\neq\emptyset\}-1
  =\sup\{k \mid   {\mathcal{H}}_k(M)\neq\emptyset\}-1=c({\mathcal{H}}(M)). 
  \end{eqnarray}
  For  example,  
  \begin{enumerate}[(1)]
   \item
   if  $\dim  M\geq  1$,  
   $ \vec{\mathcal{H}}(M)= \overrightarrow{\rm  Ind}(M) $ 
  and    $  {\mathcal{H}}(M)=  {\rm  Ind}(M)$,
  then  $c(\vec{\mathcal{H}}(M))=c({\mathcal{H}}(M))=+\infty$; 
  \item
  if  $ \vec{\mathcal{H}}(M)={\rm  sk}^{k-1}(\overrightarrow{\rm  Ind}(M))$ 
  and    $  {\mathcal{H}}(M)={\rm  sk}^{k-1}( {\rm  Ind}(M))$,
  then  
  $c(\vec{\mathcal{H}}(M))=c({\mathcal{H}}(M))= k -1$. 
  \end{enumerate}

  \begin{example}
  Let  $M$  be  a  discrete  set   $V$. 
  Let   $\mathcal{H}$  be  a  hypergraph  on  $V$.  
  \begin{enumerate}[(1)]
  \item
  An  affinely  regular  embedding  or   a   geometric  realization  
  of  $\mathcal{H}$  is   given  by  an  embedding  
  $f:  V\longrightarrow  \mathbb{F}^c$  such  that   
  any  hyperedge  
  $\sigma=\{v_1,\ldots,v_k\}$  of  $\mathcal{H}$  is  sent  to 
  an  affinely  independent  set  $\{f(v_1),   \ldots,  f(v_k)\}$ 
  in  $\mathbb{F}^c$.   
  In  particular,      (\ref{eq-26-4-19-b})  is  an 
  affinely  regular  embedding  or   a   geometric  realization  
  of  $\mathcal{H}$;    
  \item
  A   regular  embedding     
  of  $\mathcal{H}$  is   given  by  an  embedding  
  $f:  V\longrightarrow  \mathbb{F}^c$  such  that   
  any  hyperedge  
  $\sigma=\{v_1,\ldots,v_k\}$  of  $\mathcal{H}$  is  sent  to 
  a  linearly  independent  set  $\{f(v_1),   \ldots,  f(v_k)\}$ 
  in  $\mathbb{F}^c$.   
  \end{enumerate}
  \end{example}

     \begin{example}
  Let  $M=\mathbb{C}$.  
  \begin{enumerate}[(1)]
  \item
  Let  $f:  \mathbb{C}\longrightarrow  \mathbb{R}\times \mathbb{C}^{k-1}=\mathbb{R}^{2k-1}$  be  given  by 
  \begin{eqnarray*}
  f(z)=(1,z,z^2,\cdots, z^{k-1})
  \end{eqnarray*}
  for any  $z\in\mathbb{C}$.  
 By  \cite[Example~1.2]{cohen1},  $f$  is  a     $k$-regular  embedding.
 Thus      by  Corollary~\ref{co-26.4.19.1}~(2),  
 $f$  gives  a  
  simplicial  embedding  of  
 ${\rm  sk}^{k-1}({\rm  Ind}(\mathbb{C}))$  into  $ \tilde   V_\bullet (\mathbb{R}^{2k-1})/\Sigma_\bullet $;
  \item
  Let  $f:  \mathbb{C}\longrightarrow  \mathbb{C}^\infty$  be  given  by 
  \begin{eqnarray*}
  f(z)=(1,z,z^2,\cdots, z^k,\cdots)
  \end{eqnarray*}
  for any  $z\in\mathbb{C}$.  
  Then  for  any    $k\geq  1$  and  any  distinct  points  
  $z_1,\ldots, z_k\in \mathbb{C}$, 
  since  
  \begin{eqnarray*}
  \det 
  \begin{pmatrix}
  1  & \cdots   &1\\
  z_1 &\cdots  &z_k\\
   \cdots  &\cdots    &\cdots   \\
  z_1^{k-1} &\cdots & z_k^{k-1}
  \end{pmatrix}
  \neq   0,
  \end{eqnarray*}
  we  have  that  $f(z_1),\ldots,  f(z_k)$ are  (complex)  linearly  independent  in  
  $\mathbb{C}^\infty$. 
  Therefore,  $f$  is   a  (complex)  regular  embedding   of  
 $\mathbb{C}$  into  $\mathbb{C}^\infty$,  
 which  is  equivalent  to   a  simplicial  embedding  of   ${\rm  Ind}(\mathbb{C})$
 into  $ \tilde   V_\bullet (\mathbb{C}^{\infty})/\Sigma_\bullet $.    
  \end{enumerate}
  \end{example}
  
  \begin{example}
  Let   $M$  be  a  $1$-dimensional   manifold  without  boundary.  
  Then  $M$  is  a  disjoint  union  
  \begin{eqnarray}\label{eq-26.4.19.d1}
  M=\Big(\bigsqcup_{\alpha}  \mathbb{R} \Big)   \bigsqcup
  \Big( \bigsqcup_{\beta}  S^1  \Big)  
  \end{eqnarray}
  where  $\alpha$  and  $\beta$  are  finite  or  countable  indices.   
  By  (\ref{eq-26-march-02-2}), 
  \begin{eqnarray}\label{eq-26.4.19.d1.jn}
  {\rm  Ind}(M)= \big(*_\alpha {\rm  Ind}(\mathbb{R})  \big)*
  \big(*_\beta {\rm  Ind}(S^1)  \big).  
  \end{eqnarray}
  Let  $f_\alpha:  \mathbb{R}\longrightarrow  \mathbb{R}^\infty$ 
  be  given  by  
  \begin{eqnarray*}
  f_\alpha(x)= (1,x,x^2,\cdots, x^k,\cdots)
  \end{eqnarray*}
  for any  $x\in \mathbb{R}$.
  Let  $f'_\beta:  S^1\longrightarrow  \mathbb{R}^\infty$
  be  given  by 
    \begin{eqnarray}\label{eq-26.4.19.cr}
  f'_\beta(x)= (1,x,x^2,\cdots, x^k,\cdots)
  \end{eqnarray}
  for  any  $x\in  S^1\subseteq \mathbb{C}$,  where  in  the  target  
  of  (\ref{eq-26.4.19.cr}), 
  $\mathbb{C} $  is  identified  with
    $\mathbb{R}\oplus\mathbb{R}$  and  consequently   $\mathbb{C}^\infty$
    is  identified  with  $\mathbb{R}^\infty$.   
    Then  
    \begin{eqnarray*}
    {\rm  Ind}(f_\alpha):    {\rm  Ind}(\mathbb{R})\longrightarrow 
    \tilde  V_\bullet(\mathbb{R}^\infty)/\Sigma_\bullet,~~~~~~ 
    {\rm  Ind}(f'_\beta):    {\rm  Ind}(S^1)\longrightarrow 
    \tilde  V_\bullet(\mathbb{R}^\infty)/\Sigma_\bullet
    \end{eqnarray*}
    are  simplicial  embeddings. 
    With  the  help  of  (\ref{eq-26.4.19.d1.jn}),
          the   join  
    \begin{eqnarray*}
    (*_\alpha  {\rm  Ind}(f_\alpha)) *(*_\beta   {\rm  Ind}(f'_\beta)):
    {\rm  Ind}(M)\longrightarrow 
    V_\bullet(\mathbb{R}^\infty)/\Sigma_\bullet
     \end{eqnarray*} 
    is  a  simplicial  embedding.  
    Let   $f:  M\longrightarrow   \mathbb{R}^\infty$  be  given  by  
    $f(x)= (0, \ldots, 0, f_\alpha(x), 0, \ldots, 0)$  if  $x$  is  in  the  $\alpha$-th  copy  of  $\mathbb{R}$ 
    and  $f(x)=(0, \ldots, 0,  f'_\beta(x),0, \ldots, 0)$  if   $x$  is  in  the  $\beta$-th  copy  of  $S^1$, 
    for  any  $x$  in  (\ref{eq-26.4.19.d1}). 
    Then  $f$  is  a    regular  embedding  of  $M$  into  
    $\mathbb{R}^\infty$,  
    which  is  equivalent to  a  simplicial  embedding  of  ${\rm  Ind}(M)$  into  
    $ \tilde   V_\bullet (\mathbb{R}^{\infty})/\Sigma_\bullet $.  
  \end{example}

\section{Differential  forms  on       hypergraphs
  on  manifolds}\label{26.6.5-s7}
  
  In  this  section,  
  we  study  differential  forms  on  configuration  spaces  and  hypergraphs
  on  manifolds.      
  In  Subsection~\ref{26.6.5-ss7.1},  
  we  study  the  space  of   differential  forms  on  
  the  directed  independence  complexes  and 
   the  space  of   differential  forms   on 
   the  independence  
  complexes  on  manifolds.  
  In  Subsection~\ref{26.6.5-ss7.2},  
  we  study  the  space  of   differential  forms  on  
  hyperdigraphs  on  manifolds  as  pull-backs  of  the  space  of  
  differential  forms  on
  the  directed  independence  complexes
  and   study  the  space  of   differential  forms  on  
  hypergraphs  on  manifolds  as  pull-backs  of  the  space  of  
  differential  forms  on
  the    independence  complexes.  
  Under  some     hypotheses  on  compactness,  
  we  prove   a  commutative  diagram of  double  graded  vector  spaces
  of  the  differential forms  on  hyper(di)graphs,  the
  lower-associated  (directed)  simplicial  complexes,   the 
  associated  (directed)  simplicial  complexes 
  and  the  independence  complexes   on  manifolds, 
  in  Theorem~\ref{th-26.5.3.1}  (Main  Result  II).  
  
  \subsection{Differential  forms  on  configuration  spaces}
  \label{26.6.5-ss7.1}

The    collection  of  all  smooth  
differential  forms  on  $ {\rm  Conf}_k(M) $   is  a  CDGA  (commutative  
differential  graded  algebra)     
\begin{eqnarray}\label{eq-26.3.5.df1}
\Omega^\bullet {\rm  Conf}_k(M)=\bigoplus_{n\geq  0} \Omega^n {\rm  Conf}_k(M)  
\end{eqnarray}  
with  the  exterior  derivative  
\begin{eqnarray}\label{eq-26.3.5.df2}
d:  \Omega^n {\rm  Conf}_k(M)\longrightarrow \Omega^{n+1} {\rm  Conf}_k(M) 
\end{eqnarray}
  and   the  exterior  product  
  \begin{eqnarray}\label{eq-26.3.5.df3}
  \wedge:    \Omega^p {\rm  Conf}_k(M) \times   \Omega^q {\rm  Conf}_k(M)
   \longrightarrow  \Omega^{p+q} {\rm  Conf}_k(M), 
  \end{eqnarray}   
    where  $p,q\geq  0$.   
    The  $\Sigma_k$-action  on  ${\rm  Conf}_k(M)$  induces  pull-backs 
    $(\Sigma_k)^*$ of  differential  forms  given  by a  family  of  homomorphisms 
    of  CDGAs 
    \begin{eqnarray*}
    s^*:  (\Omega^\bullet {\rm  Conf}_k(M),  d,\wedge)\longrightarrow  
    (\Omega^\bullet {\rm  Conf}_k(M),  d,\wedge) 
    \end{eqnarray*}
    for  any   $s\in \Sigma_k$.

   Consider  the  double-graded  vector  space 
    \begin{eqnarray}\label{eq-26.4.28.2}
 \Omega^\bullet {\rm  Conf}_\bullet(M)=\bigoplus_{k\geq  0 } \bigoplus_{n\geq  0 } \Omega^n {\rm  Conf}_k(M). 
    \end{eqnarray}
     With  the  help  of  (\ref{eq-26-3-10-1}),  
 (\ref{eq-26.4.28.2})  
has     pull-backs  of  differential  forms   induced  by  the  face  maps 
\begin{eqnarray}\label{eq-26.4.30-1}
(\partial^i_k)^*:  \Omega^\bullet {\rm  Conf}_{k-1}(M)\longrightarrow 
 \Omega^\bullet{\rm  Conf}_{k}(M).    
\end{eqnarray}
    Consider  the  alternating  sum 
\begin{eqnarray}\label{eq-26.4.30-5}
(\partial_k)^*= \sum_{i=1}^k(-1)^{i-1} (\partial^i_k)^*.    
\end{eqnarray}

    \begin{proposition}
    \label{le-26.4.28.1}
    The  quintuple 
\begin{eqnarray*}
  \Omega  \overrightarrow {\rm  Ind}(M) =
  ( \Omega^\bullet {\rm  Conf}_\bullet(M),  d,   (\partial_\bullet)^*, \wedge,(\Sigma_\bullet)^*)
\end{eqnarray*}
is   a  double  complex  with  a  multiplication  $\wedge$  and  $\Sigma_\bullet$-actions 
 such  that  
 \begin{enumerate}[(1)]  
 \item
$(\partial _k)^*\circ d= d\circ (\partial _k)^*$,     
\item   
$
(\partial _k)^* (\omega\wedge\xi)= (\partial _k)^*(\omega)\wedge (\partial _k)^*(\xi)  
$
for  any  $\omega,\xi\in  \Omega^\bullet {\rm  Conf}_{k-1}(M)$,  
\item
  equipped  with  the  alternating  sum 
 \begin{eqnarray*}
 (\varphi_k (\Sigma_{k-1}))^*=  \sum_{i=1}^k(-1)^{i-1}  (\varphi_k ^i(\Sigma_{k-1}))^*, 
 \end{eqnarray*}
   the  diagram  commutes
  \begin{eqnarray*}
\xymatrix{
 \Omega^\bullet{\rm  Conf}_k(M)   
& \Omega^\bullet{\rm  Conf}_{k-1}(M)\ar[l]_-{(\partial_k)^*}\\
 \Omega^\bullet{\rm  Conf}_k(M)\ar[u]^-{(\varphi_k (\Sigma_{k-1}))^*} 
  & \Omega^\bullet{\rm  Conf}_{k-1}(M)\ar[l]_-{(\partial_k)^*}\ar[u]_-{(\Sigma_{k-1})^*}.     
}
\end{eqnarray*}
\end{enumerate}
    \end{proposition}
    
    \begin{proof}
    The  pull-backs  (\ref{eq-26.4.30-1})  satisfy  the  following  properties 
\begin{enumerate}[(1)']
\item 
$
(\partial^i_k)^*\circ d= d\circ (\partial^i_k)^*,   
$  
\item    
$
(\partial^i_k)^* (\omega\wedge\xi)= (\partial^i_k)^*(\omega)\wedge (\partial^i_k)^*(\xi)  
$
for  any  $\omega,\xi\in  \Omega^\bullet {\rm  Conf}_{k-1}(M)$,
\item
 the  diagram  commutes
  \begin{eqnarray*}
\xymatrix{
 \Omega^\bullet{\rm  Conf}_k(M)   
& \Omega^\bullet{\rm  Conf}_{k-1}(M)\ar[l]_-{(\partial^i_k)^*}\\
 \Omega^\bullet{\rm  Conf}_k(M)\ar[u]^-{(\varphi_k^i(\Sigma_{k-1}))^*} 
  & \Omega^\bullet{\rm  Conf}_{k-1}(M)\ar[l]_-{(\partial^i_k)^*}\ar[u]_-{(\Sigma_{k-1})^*}.  
}
\end{eqnarray*}
\end{enumerate}
Therefore, the  linear  combination  (\ref{eq-26.4.30-5})  
satisfies  (1)-(3).  
    \end{proof}

    The    collection  of  all  smooth  
differential  forms  on  $ {\rm  Conf}_k(M)/\Sigma_k $   is  a  CDGA        
\begin{eqnarray}\label{eq-26.5.1.df1}
\Omega^\bullet ({\rm  Conf}_k(M)/\Sigma_k)=\bigoplus_{n\geq  0} \Omega^n ({\rm  Conf}_k(M)  /\Sigma_k)
\end{eqnarray}  
with  the  exterior  derivative  
\begin{eqnarray}\label{eq-26.5.1.df2}
d:  \Omega^n ({\rm  Conf}_k(M)/\Sigma_k)\longrightarrow 
\Omega^{n+1} ({\rm  Conf}_k(M) /\Sigma_k)
\end{eqnarray}
  and   the  exterior  product  
  \begin{eqnarray}\label{eq-26.5.1.df3}
  \wedge:    \Omega^p( {\rm  Conf}_k(M)/\Sigma_k) \times  
   \Omega^q ({\rm  Conf}_k(M)/\Sigma_k)
   \longrightarrow  \Omega^{p+q} ({\rm  Conf}_k(M)/\Sigma_k).   
  \end{eqnarray}

\begin{lemma}\label{le-26.4.30-2}
The  covering  map  (\ref{eq-2601-a})  induces  an  injective  homomorphism   of  CDGAs
\begin{eqnarray}\label{eq-26-5.1-b}
\pi_k^*:   (\Omega^\bullet({\rm  Conf}_k(M)/\Sigma_k), d,  \wedge)
\longrightarrow   (\Omega^\bullet   {\rm  Conf}_k(M), d,  \wedge) 
\end{eqnarray}
whose  image consists   of  all  
  the  $\Sigma_k$-invariant  
  differential  forms   (i.e.  the differential  forms  invariant  under the  $\Sigma_k$-action) 
on ${\rm  Conf}_k(M)$. 
\end{lemma}

\begin{proof}
Since (\ref{eq-2601-a})  is  a  smooth  covering  map  of     differentiable  manifolds,
its  induced    pull-back  of  differential  forms  gives   an   injective  homomorphism   
(\ref{eq-26-5.1-b})  of  CDGAs.   
For  any  $s\in  \Sigma_k $, 
the  diagram  commutes 
 \begin{eqnarray}\label{eq-26-3-10-5}
 \xymatrix{
 \Omega^\bullet({\rm  Conf}_k(M)/\Sigma_k)\ar[r]^-{\pi_k^*}  \ar[rd]_-{\pi_k^*} & \Omega^\bullet   {\rm  Conf}_k(M)
 \ar[d]^-{s^*}\\
 &   \Omega^\bullet   {\rm  Conf}_k(M).  
   }
 \end{eqnarray}
Thus  the  image  of  (\ref{eq-26-5.1-b})  is    invariant  under the  $\Sigma_k$-action  
on ${\rm  Conf}_k(M)/\Sigma_k$. 
Moreover,  any  differential  form  on  ${\rm  Conf}_k(M)$  that  is 
  invariant  under the  $\Sigma_k$-action  is  the  pull-back  of  
  certain  differential  form  on  ${\rm  Conf}_k(M)/\Sigma_k$.  
  Therefore,  the  image  of  
   (\ref{eq-26-5.1-b})   is   the  space  of  all 
  the  $\Sigma_k$-invariant  
  differential  forms   
on ${\rm  Conf}_k(M)$.  
\end{proof}

Consider  the  double-graded  vector  space 
    \begin{eqnarray}\label{eq-26.5.1.2}
  \Omega^\bullet ({\rm  Conf}_\bullet(M)/\Sigma_\bullet)=\bigoplus_{k\geq  0 } \bigoplus_{n\geq  0 } \Omega^n ({\rm  Conf}_k(M)/\Sigma_k). 
    \end{eqnarray}
Take the  triple 
\begin{eqnarray*}
 \Omega    {\rm  Ind}(M) = (\Omega^\bullet({\rm  Conf}_\bullet(M)/\Sigma_\bullet), d,  \wedge). 
\end{eqnarray*}

\begin{corollary}\label{pr-26-5-1-1}
The  projection (\ref{eq-2601-aa})  induces  an  injective  homomorphism  
\begin{eqnarray*}
\pi^*:   \Omega    {\rm  Ind}(M) \longrightarrow  
 \Omega    \overrightarrow {\rm  Ind}(M) 
\end{eqnarray*}
whose  image consists   of  
  the  $\Sigma_\bullet$-invariant  
  differential  forms
such that  
\begin{enumerate}[(1)]
 \item
$\pi^*\circ d= d\circ \pi^*$,     
\item   
$
\pi^* (\omega\wedge\xi)= \pi^*(\omega)\wedge \pi^*(\xi)  
$
for  any  $\omega,\xi\in  \Omega^\bullet ({\rm  Conf}_{k}(M)/\Sigma_k)$. 
\end{enumerate}
\end{corollary}

 \begin{proof}
 The  proof  follows  from  Lemma~\ref{le-26.4.30-2}.  
 \end{proof}

  \begin{proposition}
    \label{pr-26.5.2.1}
    Suppose  $M$  is  a  submanifold  of  $\mathbb{R}$
  with  the  total  order  $\prec$.  
    Then  the  quadruple
\begin{eqnarray}\label{eq-26.5.2.25}
  \Omega  {\rm  Ind}(M) =
  ( \Omega^\bullet ({\rm  Conf}_\bullet(M)/\Sigma_\bullet),  d,   (\partial_\bullet)^*, \wedge)
\end{eqnarray}
is   a   double  complex  with  a  multiplication  $\wedge$  
 satisfying  
 \begin{enumerate}[(1)]  
 \item
$(\partial _k)^*\circ d= d\circ (\partial _k)^*$,     
\item   
$
(\partial _k)^* (\omega\wedge\xi)= (\partial _k)^*(\omega)\wedge (\partial _k)^*(\xi)  
$
for  any  $\omega,\xi\in  \Omega^\bullet ({\rm  Conf}_{k-1}(M)/\Sigma_{k-1})$. 
\end{enumerate}
    \end{proposition}
    
    \begin{proof}
    By  (\ref{eq-5.2.5}), 
    we  have an  induced  isomorphism  of  CDGAs
      \begin{eqnarray}\label{eq-26.5.iso1}
  (\theta_k)^*:  ( \Omega^\bullet {\rm  Conf}_k (M,\prec), d,\wedge)
  \overset{\cong}{\longrightarrow}
   \Omega^\bullet({\rm  Conf}_k(M)/\Sigma_k),  d,\wedge). 
  \end{eqnarray}
  By  (\ref{eq-5.1ab}),  
  we  have  
    pull-backs  of  differential  forms   induced  by  the  face  maps 
\begin{eqnarray}\label{eq-26.5.2-12}
(\partial^i_k)^*:  \Omega^\bullet {\rm  Conf}_{k-1}(M,\prec)\longrightarrow 
 \Omega^\bullet{\rm  Conf}_{k}(M,\prec)     
\end{eqnarray}
such  that  equipped  with  the  alternating  sum  (\ref{eq-26.4.30-5}),  
there  is  a  double  complex  with  a  multiplication  $\wedge$ 
\begin{eqnarray}\label{eq-26-5-2-21}
  \Omega  {\rm  Ind}(M,\prec) =
  ( \Omega^\bullet ({\rm  Conf}_\bullet(M,\prec),  d,   (\partial_\bullet)^*, \wedge).  
\end{eqnarray}
Therefore,  with  the  help  of  the  isomorphism   (\ref{eq-26.5.iso1}),  
(\ref{eq-26.5.2.25})  follows  from   (\ref{eq-26-5-2-21}).  
    \end{proof}

  \subsection{Differential  forms  on  hyper(di)graphs}
  \label{26.6.5-ss7.2}

For  any  $k$-uniform  hyperdigraph     $\vec{\mathcal{H}}_k(M)$    on  $M$,  
the    collection  of  all  smooth  
differential  forms  on  $\vec{\mathcal{H}}_k(M)$   is  a     CDGA    
$(\Omega^\bullet \vec{\mathcal{H}}_k(M), d, \wedge)$  where  
\begin{eqnarray*}
\Omega^\bullet \vec{\mathcal{H}}_k(M)=\bigoplus_{n\geq  0}
 \Omega^n \vec{\mathcal{H}} _k(M) 
\end{eqnarray*}     
 with  the  exterior  derivative  
\begin{eqnarray*}
d:  \Omega^n \vec{\mathcal{H}}_k(M)\longrightarrow 
\Omega^{n+1} \vec{\mathcal{H}}_k(M) 
\end{eqnarray*}
  and   the  exterior  product  
  \begin{eqnarray*}
  \wedge:    \Omega^p \vec{\mathcal{H}}_k(M) \times  
   \Omega^q \vec{\mathcal{H}}_k(M)
   \longrightarrow  \Omega^{p+q} \vec{\mathcal{H}}_k(M).  
  \end{eqnarray*}   
 The    collection  of  all  smooth  
differential  forms  on  the  underlying
$k$-uniform   hypergraph  ${\mathcal{H}}_k(M)$   is  a     CDGA    
$(\Omega^\bullet {\mathcal{H}}_k(M), d, \wedge)$  where  
\begin{eqnarray*}
\Omega^\bullet {\mathcal{H}}_k(M)=\bigoplus_{n\geq  0}
 \Omega^n {\mathcal{H}} _k(M) 
\end{eqnarray*}     
 with  the  exterior  derivative  
\begin{eqnarray*}
d:  \Omega^n {\mathcal{H}}_k(M)\longrightarrow 
\Omega^{n+1} {\mathcal{H}}_k(M) 
\end{eqnarray*}
  and   the  exterior  product  
  \begin{eqnarray*}
  \wedge:    \Omega^p {\mathcal{H}}_k(M) \times  
   \Omega^q  {\mathcal{H}}_k(M)
   \longrightarrow  \Omega^{p+q}  {\mathcal{H}}_k(M).  
  \end{eqnarray*}

  \begin{lemma}\label{le-5.3.1}
  If  $\vec{\mathcal{H}}_k(M)$  is  a  compact  submanifold  of  
  ${\rm  Conf}_k(M)$,  then the  canonical  inclusion  of   $\vec{\mathcal{H}}_k(M)$ 
  in  ${\rm  Conf}_k(M)$  induces  a  projection  of CDGAs 
  \begin{eqnarray}\label{eq-5.3.11}
(\Omega^\bullet {\rm  Conf}_k(M), d, \wedge)  \longrightarrow (\Omega^\bullet \vec{\mathcal{H}}_k(M), d, \wedge)   
  \end{eqnarray}
  and   the  canonical  inclusion  of   ${\mathcal{H}}_k(M)$ 
  in  ${\rm  Conf}_k(M)/\Sigma_k$  induces  a  projection  of CDGAs 
  \begin{eqnarray}\label{eq-5.3.12}
(\Omega^\bullet {\rm  Conf}_k(M)/\Sigma_k, d, \wedge)  
\longrightarrow (\Omega^\bullet  {\mathcal{H}}_k(M), d, \wedge).    
  \end{eqnarray}
    \end{lemma} 
    
    \begin{proof}
    Suppose $\vec{\mathcal{H}}_k(M)$  is  a  compact  submanifold  of  
  ${\rm  Conf}_k(M)$.  Then 
    by  a   smooth  version  of  the  Tietze Extension Theorem, 
    a  smooth  differential  form  on  $\vec{\mathcal{H}}_k(M)$ 
    can  be  extended  to    a  smooth  differential  form  on 
      ${\rm  Conf}_k(M)$.  
      Thus  the  map  (\ref{eq-5.3.11})  sending  a  smooth  differential  form
       on   ${\rm  Conf}_k(M)$  to  its  restriction  on   $\vec{\mathcal{H}}_k(M)$ 
       is     a   projection.  
       Moreover,  the  compactness  of   $\vec{\mathcal{H}}_k(M)$  in  ${\rm  Conf}_k(M)$
       implies the  compactness  of   ${\mathcal{H}}_k(M)$  in  
       ${\rm  Conf}_k(M)/\Sigma_k$.  
       Thus  a  smooth  differential  form  on  $ {\mathcal{H}}_k(M)$ 
    can  be  extended  to    a  smooth  differential  form  on 
      ${\rm  Conf}_k(M)/\Sigma_k$.
       Hence  the  map  (\ref{eq-5.3.12})  sending  a  smooth  differential  form
       on   ${\rm  Conf}_k(M)/\Sigma_k$  to  its  restriction  on   ${\mathcal{H}}_k(M)$ 
       is     a   projection.  
    \end{proof}
    
    \begin{lemma}\label{co-26-5-3-1}
    If  $\vec{\mathcal{H}}_k(M)$  is  a  compact  submanifold  of  
  ${\rm  Conf}_k(M)$,  then
  the  canonical  inclusions  
 $\vec{\mathcal{H}}_k(M) \subseteq  {\rm  Sym}(\vec{\mathcal{H}}_k(M))
 \subseteq  {\rm  Conf}_k(M)$  induce    projections  of  CDGAs  
 \begin{eqnarray}\label{eq-26-3-10-2}
 (\Omega^\bullet  {\rm  Conf}_k(M),d,\wedge)
 \longrightarrow   (\Omega^\bullet{\rm  Sym}(\vec{\mathcal{H}}_k(M)),d,\wedge)
 \longrightarrow  ( \Omega^\bullet\vec{\mathcal{H}}_k(M),d,\wedge).  
 \end{eqnarray}
    \end{lemma}
    
    \begin{proof}
  By  the  proof  of  Lemma~\ref{le-5.3.1},  
    the  compactness  of  $ \vec{\mathcal{H}}_k(M)$
  in   ${\rm  Conf}_k(M)$  implies  the  compactness  of  ${\mathcal{H}}_k(M)$
  in   ${\rm  Conf}_k(M)/\Sigma_k$.
  With  the  help  of  (\ref{eq-26-3-11-1}),  the  compactness  of  ${\mathcal{H}}_k(M)$
  in   ${\rm  Conf}_k(M)/\Sigma_k$
   implies     the  compactness  of  
  ${\rm  Sym}(\vec{\mathcal{H}}_k(M))$  in  ${\rm  Conf}_k(M)$.   
    By  a  similar  argument  of   Lemma~\ref{le-5.3.1}, 
    we obtain  the  first  projection  in  (\ref{eq-26-3-10-2}).  
     Moreover, 
     the  compactness  of  $ \vec{\mathcal{H}}_k(M)$  and 
      ${\rm  Sym}(\vec{\mathcal{H}}_k(M))$
  in   ${\rm  Conf}_k(M)$    implies  the  compactness  of 
   $ \vec{\mathcal{H}}_k(M)$
  in     ${\rm  Sym}(\vec{\mathcal{H}}_k(M))$.  
    By  an  argument    similar   to  that  of   Lemma~\ref{le-5.3.1}, 
    we obtain the  second  projection  in  (\ref{eq-26-3-10-2}). 
    \end{proof}

  \begin{proposition}\label{le-26.5.1.8}
   Let  $\vec{\mathcal{H}}(M)$  be  a  hyperdigraph  on  $M$ 
  with  its  underlying  hypergraph  ${\mathcal{H}} (M)$
  such  that  $ \vec{\mathcal{H}}_k(M)$  is   a    compact  submanifold  of   ${\rm  Conf}_k(M)$  for any  $k\geq  1$. 
    Then  we  have  linear  maps  
    of  double-graded  vector  spaces
 \begin{eqnarray}\label{eq-5.3.21}
 \Omega^\bullet({\rm  Conf}_\bullet(M)/\Sigma_\bullet) \longrightarrow  \Omega^\bullet{\mathcal{H}}_\bullet(M)
 \longrightarrow 
 \Omega^\bullet{\rm  Sym}(\vec{\mathcal{H}}_\bullet(M))
  \longrightarrow
  \Omega^\bullet\vec{\mathcal{H}}_\bullet(M)
 \end{eqnarray}
 with  the  second  map      injective, 
   and  the  first  and   third   maps    surjective,
    such  that  for  any  $k\geq  1$,  
(\ref{eq-5.3.21})  are  homomorphisms    of   CDGAs 
\begin{eqnarray}\label{eq-26-3-10-6}
&(\Omega^\bullet({\rm  Conf}_k(M)/\Sigma_k),d,\wedge) \longrightarrow 
 ( \Omega^\bullet{\mathcal{H}}_k(M),d,\wedge)
\\
&\longrightarrow 
 (\Omega^\bullet{\rm  Sym}(\vec{\mathcal{H}}_k(M)),d,\wedge)
 \longrightarrow  ( \Omega^\bullet\vec{\mathcal{H}}_k(M),d,\wedge). 
 \nonumber
 \end{eqnarray}
  \end{proposition}
  
  \begin{proof}
  Suppose   $\vec{\mathcal{H}}_k(M)$  is  a  compact  submanifold  of  
  ${\rm  Conf}_k(M)$  for  any  $k\geq  1$.   
  Then  by  (\ref{eq-5.3.12})   in  
  Lemma~\ref{le-5.3.1},  
  we  obtain  the  first  homomorphism  of  CDGAs  in 
  (\ref{eq-26-3-10-6}).  
  By  (\ref{eq-26-3-10-2})  in  Lemma~\ref{co-26-5-3-1},  
  we  obtain  the  third  homomorphism  of  CDGAs  in 
  (\ref{eq-26-3-10-6}).  
  Moreover,  
    the  commutative   diagram  
 \begin{eqnarray*}
 \xymatrix{
 {\rm  Sym}(\vec{\mathcal{H}}_k(M))\ar[r]\ar[d]_-{\pi_k}
  & {\rm  Conf}_k(M)\ar[d]^-{\pi_k}\\
 {\mathcal{H}}_k(M)\ar[r]
& {\rm  Conf}_k(M)/\Sigma_k, 
 }
 \end{eqnarray*}
 where  the  horizontal  maps  are  canonical  inclusions  
 and  the vertical  maps  are  canonical  projections,  
 induces  a  commutative   diagram  of  CDGAs  
 \begin{eqnarray}\label{diag-26-3-10-3}
 \xymatrix{
 (\Omega^\bullet({\rm  Conf}_k(M)/\Sigma_k),d,\wedge)
 \ar[r]\ar[d]_-{(\pi_k)^*}
 & ( \Omega^\bullet{\mathcal{H}}_k(M),d,\wedge)\ar[d]^-{(\pi_k)^*}\\
( \Omega^\bullet{\rm  Conf}_k(M),d,\wedge)
  \ar[r]
 &
   (\Omega^\bullet {\rm  Sym}(\vec{\mathcal{H}}_k(M)),d,\wedge),
    }
 \end{eqnarray}
 where  the  horizontal  maps  are   surjective  
 and  the vertical  maps  are   injective.  
 By  the  second  vertical  map  $(\pi_k)^*$  in  
    (\ref{diag-26-3-10-3}),  
    we  obtain  the  second  homomorphism  of  CDGAs  in 
  (\ref{eq-26-3-10-6}).  
  Let  $k$  run  over  all  nonnegative   integers  in  (\ref{eq-26-3-10-6}).  
 We  obtain    the  linear  maps 
 (\ref{eq-5.3.21})  of  double-graded  vector  spaces.  
  \end{proof}

  \begin{lemma}\label{le-26.5.3.5}
    Let  $\vec{\mathcal{H}}(M)$  be  a  hyperdigraph  on  $M$ 
  such  that  $ \vec{\mathcal{H}}_k(M)$  is   a    compact  submanifold  of   ${\rm  Conf}_k(M)$  for any  $k\geq  1$. 
  Then    
  \begin{enumerate}[(1)]
  \item
   $(\delta \vec{\mathcal{H}}(M))_k$  is  a    compact  submanifold 
    of   ${\rm  Conf}_k(M)$  for any  $k\geq  1$;  
  \item 
      If  $c(\vec{\mathcal{H}}(M))<+\infty$ 
       (cf.  Subsection~\ref{ss-7.2-geomr},~(\ref{eq-26-5-7-5})), 
  then   
 $(\Delta \vec{\mathcal{H}}(M))_k$  is  a   compact  submanifold 
    of   ${\rm  Conf}_k(M)$  for any  $k\geq  1$. 
    \end{enumerate}
  \end{lemma}
  
  \begin{proof}
  (1)     Let  $ \vec{\sigma}_i(M)$   ($i=1,2,\ldots$)   be  a  Cauchy  sequence  in  
  $(\delta \vec{\mathcal{H}}(M))_k$.  
  Since  $(\delta \vec{\mathcal{H}}(M))_k\subseteq  \vec{\mathcal{H}}_k(M)$,
  we  have  that  
  $ \vec{\sigma}_i(M)$,   $i=1,2,\ldots$, is  a  Cauchy  sequence  in  
  $ \vec{\mathcal{H}}_k(M)$.  
  Since  $ \vec{\mathcal{H}}_k(M)$  is   closed  in   ${\rm  Conf}_k(M)$,
  we  have  that 
   $ \vec{\sigma}_i(M)$,   $i=1,2,\ldots$,  converges  to  
   a  directed  $k$-hyperedge  $\vec\sigma(M)\in  \vec{\mathcal{H}}_k(M)$.

   We  want to  prove  $\vec\sigma(M)\in  (\delta \vec{\mathcal{H}}(M))_k$. 
      Suppose  to the  contrary,  
   $\vec\sigma(M)\notin  (\delta \vec{\mathcal{H}}(M))_k$. 
   Then  there  exists  $1\leq  l\leq  k-1$  and   a   non-empty  $l$-subsequence  
   $\vec\tau (M)$  of  $\vec\sigma(M)$  such  that  
    $\vec\tau (M)\in   {\rm  Conf}_l(M)$  and  
   $\vec\tau (M)\notin   \vec{\mathcal{H}}_l(M)$. 
   Without  loss  of  generality,  
   we  assume  that   $\vec\tau (M)$  is  obtained  from  $\vec\sigma(M)$
   by  taking  the  first  $l$-coordinates.   
    By  the  compactness of  $  \vec{\mathcal{H}}_l(M)$  in
    $ {\rm  Conf}_l(M)$,  
    there  exists  an  open  neighborhood  $\mathcal{U}$  of  
    $\vec\tau (M)$  in    $  {\rm  Conf}_l(M)$  such  that  
    \begin{eqnarray}\label{eq-26.5.5.1}
    \mathcal{U}\cap   \vec{\mathcal{H}}_l(M)=\emptyset.
    \end{eqnarray}  
    Let  $\vec\tau_i(M)$    be  
    the  first  $l$-coordinates  of  $\vec\sigma_i(M)$  for  $i=1,2,\ldots$.  
    Then  $\vec\tau_i(M)\in  (\delta \vec{\mathcal{H}}(M))_l$
    thus  $\vec\tau_i(M)\in  \vec{\mathcal{H}}_l(M)$.   
    Moreover,  $\vec\tau_i(M)$,   $i=1,2,\ldots$, 
    converges  to   $\vec\tau (M)$  in    $  {\rm  Conf}_l(M)$.  
    This  contradicts    (\ref{eq-26.5.5.1}).  
    Therefore,  $\vec\sigma(M)\in  (\delta \vec{\mathcal{H}}(M))_k$.

    It  follows   that   
     $(\delta \vec{\mathcal{H}}(M))_k$
      is  a  closed  subset  of  $ \vec{\mathcal{H}}_k(M)$. 
    Since  $ \vec{\mathcal{H}}_k(M)$  is  compact  in  ${\rm  Conf}_k(M)$, 
    we  obtain  that  
   $(\delta \vec{\mathcal{H}}(M))_k$  is  compact  
  in   ${\rm  Conf}_k(M)$  as  well.

   (2)  
   For  any  $1\leq  k\leq   c(\vec{\mathcal{H}}(M))+1
  $,   
   let  $\{\mathcal{U}^k_\alpha\mid  \alpha\in  A_k\}$,
   where  $A_k$  is  an  index  set,   be  
   a  family  of  open  subsets  in  ${\rm  Conf}_k(M)$  
   such  that   
   \begin{eqnarray}\label{eq-5.7.19}
   (\Delta \vec{\mathcal{H}}(M))_k \subseteq 
    \bigcup_{\alpha\in  A_k}  \mathcal{U}^k_\alpha.  
\end{eqnarray}
For  any   $k\leq  l\leq  c(\vec{\mathcal{H}}(M))+1$ 
and   any  subsequence   $(i_1,i_2,\ldots, i_k)$  of  
$(1,2,\ldots,    l)$, 
 we  can     choose   an  embedding  
\begin{eqnarray*}
\epsilon(i_1,i_2,\ldots, i_k):  {\rm  Conf}_k(M)\longrightarrow 
M^{l} 
\end{eqnarray*}
 sending   any  $(x_1,\ldots,x_k)\in {\rm  Conf}_k(M)$ 
 to  the  $(i_1,i_2,\ldots, i_k)$-th  coordinates  in
 $ M^{l} $.  
Then 
 \begin{eqnarray*}
 \vec{\mathcal{H}}_l(M)  
 \subseteq     (\Delta \vec{\mathcal{H}}(M))_l  \subseteq 
    \bigcup_{\alpha\in  A_k} ~ 
    \bigcup_{(i_1,i_2,\ldots, i_k)}\epsilon(i_1,i_2,\ldots, i_k)(\mathcal{U}^k_\alpha)\times 
    M^{l-k},      
\end{eqnarray*}
where   $\epsilon(i_1,i_2,\ldots, i_k)(\mathcal{U}^k_\alpha)$  is  homeomorphic 
to  $\mathcal{U}^k_\alpha$,   
 form   an  open  cover  of  $ \vec{\mathcal{H}}_l(M)$.  
Since  $  \vec{\mathcal{H}}_l(M) $  is  compact,
there   exists  a  finite  subset  $A_{k,l}\subseteq  A_k$  such that     
  \begin{eqnarray*}
 \vec{\mathcal{H}}_l(M)  
   \subseteq 
    \bigcup_{\alpha\in  A_{k,l}} ~ 
    \bigcup_{(i_1,i_2,\ldots, i_k)}\epsilon(i_1,i_2,\ldots, i_k)(\mathcal{U}^k_\alpha)\times 
    M^{l-k}.       
\end{eqnarray*}
On  the  other  hand, 
  for  any  $\vec{\sigma}(M)\in  (\Delta \vec{\mathcal{H}}(M))_k$,
  there  exists  $k\leq  l\leq  c(\vec{\mathcal{H}}(M))+1$,
   $\vec{\tau}(M)\in    \vec{\mathcal{H}}_l (M) $
and  a   subsequence   $(i_1,i_2,\ldots, i_k)$  of  
$(1,2,\ldots,    l)$  such  that  
$
\vec{\sigma}(M) 
$ 
is  obtained  by  taking  the   $(i_1,i_2,\ldots, i_k)$-th  coordinates  of  
 $\vec{\tau}(M)$.  
 Therefore,  
\begin{eqnarray*}
   (\Delta \vec{\mathcal{H}}(M))_k
   \subseteq  \bigcup_{k\leq  l\leq   c(\vec{\mathcal{H}}(M))+1}~
 \bigcup_{\alpha\in  A_{k,l}} ~ 
     \mathcal{U}^k_\alpha,  
    \end{eqnarray*} 
    which  gives  a  finite  sub-cover  of     (\ref{eq-5.7.19}).  
    Therefore,
    $(\Delta \vec{\mathcal{H}}(M))_k$  is  a   compact  submanifold 
    of   $M^k$  and  consequently  a  compact  submanifold  of  ${\rm  Conf}_k(M)$. 
  \end{proof}
   
 \begin{theorem}[Main  Result  II] 
 \label{th-26.5.3.1}
 Let  $\vec{\mathcal{H}}(M)$  be  a  hyperdigraph  on  $M$ 
  with  its  underlying  hypergraph  ${\mathcal{H}} (M)$
  such  that  
  $c(\vec{\mathcal{H}}(M))<+\infty$   and  
  $ \vec{\mathcal{H}}_k(M)$  is   a    compact  submanifold  of 
    ${\rm  Conf}_k(M)$  for any  $k\geq  1$.  
 Then  we  have  a  commutative  diagram   
    of  double-graded  vector  spaces
 \begin{eqnarray}\label{eq-5.7.21}
 \xymatrix{
  \Omega^\bullet({\rm  Conf}_\bullet(M)/\Sigma_\bullet) \ar[r]
 &  \Omega^\bullet{\delta\mathcal{H}}_\bullet(M)
 \ar[r] 
 & \Omega^\bullet{\rm  Sym}(\delta\vec{\mathcal{H}}_\bullet(M))
  \ar[r]
 & \Omega^\bullet\delta\vec{\mathcal{H}}_\bullet(M)\\
 \Omega^\bullet({\rm  Conf}_\bullet(M)/\Sigma_\bullet) \ar[r]\ar@{=}[u]
 &  \Omega^\bullet{\mathcal{H}}_\bullet(M)
 \ar[r] \ar[u]
 & \Omega^\bullet{\rm  Sym}(\vec{\mathcal{H}}_\bullet(M))
  \ar[r]\ar[u]
 & \Omega^\bullet\vec{\mathcal{H}}_\bullet(M) \ar[u]\\
  \Omega^\bullet({\rm  Conf}_\bullet(M)/\Sigma_\bullet) \ar[r]\ar@{=}[u]
 &  \Omega^\bullet{\Delta\mathcal{H}}_\bullet(M)
 \ar[r] \ar[u]
 & \Omega^\bullet{\rm  Sym}(\Delta\vec{\mathcal{H}}_\bullet(M))
  \ar[r]\ar[u]
 & \Omega^\bullet\Delta\vec{\mathcal{H}}_\bullet(M) \ar[u]
  }
 \end{eqnarray}
 where 
 \begin{enumerate}[(1)]
 \item
  in  each  row,   the  second  map   is   injective  and  the  first  as  well  as  the  third   map  is  surjective;
  \item
   in each  column,  the  maps   are  pull-backs;
    \item  
   for  any  lower  index  $k\geq  1$,  
(\ref{eq-5.7.21})   is  a  commutative  diagram      of   CDGAs.  
\end{enumerate}
 \end{theorem}
 
 \begin{proof}
  The  second  row  in  (\ref{eq-5.7.21})  follows  from  Proposition~\ref{le-26.5.1.8}. 
 Applying  Proposition~\ref{le-26.5.1.8}  to  $\delta\vec{\mathcal{H}} (M)$  
 and   $\Delta\vec{\mathcal{H}} (M)$  respectively,  
 we  obtain  the  first row  and  the  third  row  in  (\ref{eq-5.7.21}).  
 Moreover,  by     Proposition~\ref{le-26-3-3-1}~(\ref{26-3-3-diag1}),  
 the  canonical  inclusions  
 \begin{eqnarray}\label{eq-26.5.8.9}
 \delta\vec{\mathcal{H}} (M)  \subseteq  \vec{\mathcal{H}} (M)
 \subseteq 
 \Delta\vec{\mathcal{H}}(M)
 \end{eqnarray}
  induce  a  commutative  diagram 
 \begin{eqnarray}\label{eq-26.5.8.09}
 \xymatrix{
 \delta\vec{\mathcal{H}} (M)  \ar[r] \ar[d]
 &{\rm  Sym}( \delta\vec{\mathcal{H}} (M) )\ar[r] \ar[d]
 & \delta {\mathcal{H}} (M) \ar[r] \ar[d]
 & {\rm  Ind}(M)\ar@{=}[d]\\
   \vec{\mathcal{H}} (M)   \ar[r] \ar[d]
 &{\rm  Sym}(  \vec{\mathcal{H}} (M) )\ar[r] \ar[d]
 &  {\mathcal{H}} (M) \ar[r] \ar[d]
 & {\rm  Ind}(M)\ar@{=}[d]
 \\
 \Delta\vec{\mathcal{H}}(M)  \ar[r] 
 &{\rm  Sym}( \Delta\vec{\mathcal{H}} (M) )\ar[r] 
 & \Delta {\mathcal{H}} (M) \ar[r] 
 & {\rm  Ind}(M)
 }
 \end{eqnarray}
 where  
 \begin{enumerate}[(1)]
 \item
 in  each  row,  the  first as  well   as  the  third  horizontal map  is 
  the   canonical inclusion
  and the  second  horizontal map  is  the  canonical  projection; 
 \item
 the  vertical  maps  are  canonical  inclusions  induced  by 
 (\ref{eq-26.5.8.9}). 
 \end{enumerate}
 Taking  the  spaces  of  differential  forms  on  manifolds, 
 with  the  help  of  Lemma~\ref{le-26.5.3.5},   
 the  commutative  diagram  (\ref{eq-26.5.8.09})   induces  the  commutative  diagram
 (\ref{eq-5.7.21}).    
 \end{proof}

\begin{corollary} 
\label{th-26.6.10.2}
 Let  $\vec{\mathcal{H}}(M)$  be  a  hyperdigraph  on  $M$ 
  with  its  underlying  hypergraph  ${\mathcal{H}} (M)$
  such  that  
  $c(\vec{\mathcal{H}}(M))<+\infty$   and  
  $ \vec{\mathcal{H}}_k(M)$  is   a    compact  submanifold  of 
    ${\rm  Conf}_k(M)$  for any  $k\geq  1$.  
 Then  we  have  a  commutative  diagram   
    of  de Rham  cohomology  groups
 \begin{eqnarray}\label{eq-6.10.21}
 \xymatrix{
  H^\bullet({\rm  Conf}_\bullet(M)/\Sigma_\bullet) \ar[r]
 &  H^\bullet({\delta\mathcal{H}}_\bullet(M))
 \ar[r] 
 & H^\bullet({\rm  Sym}(\delta\vec{\mathcal{H}}_\bullet(M)))
  \ar[r]
 & H^\bullet(\delta\vec{\mathcal{H}}_\bullet(M))\\
 H^\bullet({\rm  Conf}_\bullet(M)/\Sigma_\bullet) \ar[r]\ar@{=}[u]
 &  H^\bullet({\mathcal{H}}_\bullet(M))
 \ar[r] \ar[u]
 & H^\bullet({\rm  Sym}(\vec{\mathcal{H}}_\bullet(M)))
  \ar[r]\ar[u]
 & H^\bullet(\vec{\mathcal{H}}_\bullet(M)) \ar[u]\\
  H^\bullet({\rm  Conf}_\bullet(M)/\Sigma_\bullet) \ar[r]\ar@{=}[u]
 &  H^\bullet({\Delta\mathcal{H}}_\bullet(M))
 \ar[r] \ar[u]
 & H^\bullet({\rm  Sym}(\Delta\vec{\mathcal{H}}_\bullet(M)))
  \ar[r]\ar[u]
 &H^\bullet(\Delta\vec{\mathcal{H}}_\bullet(M)). \ar[u]
  }
 \end{eqnarray}
 \end{corollary}

\begin{proof}
Apply  the  de  Rham  cohomology  functor  to 
the  commutative  diagram  (\ref{eq-5.7.21})  of  differential  forms  on  manifolds.  
 By  Theorem~\ref{th-26.5.3.1}, 
 we  obtain (\ref{eq-6.10.21}).  
\end{proof}

    \medskip

Shiquan Ren

Address:
School  of  Mathematics and Statistics,  Henan University,  Kaifeng   475004,  China.

e-mail:  renshiquan@henu.edu.cn

  \end{document}